\documentclass[10pt,leqno]{amsart}
\usepackage{amssymb}
\usepackage{mathrsfs}
\usepackage{color}
\usepackage{soul}
\usepackage[pdfborder={0 0 0}]{hyperref}

\newcommand\N{{\mathbb N}}
\newcommand\R{{\mathbb R}}

\newcommand\C{{\mathbb C}}

\def\AA{{\mathcal A}}
\def\BB{{\mathcal B}}
\def\CC{{\mathcal C}}
\def\DD{{\mathcal D}}

\def\FF{{\mathcal F}}

\def\HH{{\mathcal H}}

\def\KK{{\mathcal K}}

\def\OO{{\mathcal O}}

\def\RR{{\mathcal R}}

\def\TT{{\mathcal T}}
\def\UU{{\mathcal U}}
\def\VV{{\mathcal V}}
\def\WW{{\mathcal W}}

\def\ZZ{{\mathcal Z}}

\def\BBB{{\mathscr{B}}}

\def\MMM{{\mathscr{M}}}

\DeclareMathOperator{\Supp}{Supp}
\newcommand{\dom}{\mathrm D}

\def\eps{{\varepsilon}}

\newcommand{\wto}{\rightharpoonup}

\newtheorem{theo}{Theorem}[section]
\newtheorem{prop}[theo]{Proposition}
\newtheorem{lem}[theo]{Lemma}
\newtheorem{cor}[theo]{Corollary}
\newtheorem{defin}[theo]{Definition}
\theoremstyle{remark}
\newtheorem{rem}[theo]{Remark}
\newtheorem*{rem*}{Remark}

\numberwithin{equation}{section}

\newcommand{\beqn}{\begin{equation}}
\newcommand{\eeqn}{\end{equation}}
\newcommand{\bear}{\begin{eqnarray}}
\newcommand{\eear}{\end{eqnarray}}
\newcommand{\bean}{\begin{eqnarray*}}
\newcommand{\eean}{\end{eqnarray*}}
\newcommand{\bal}{\begin{aligned}}
\newcommand{\eal}{\end{aligned}}

\newcommand{\la}{\langle}
\newcommand{\ra}{\rangle}

\def\Nt{|\hskip-0.04cm|\hskip-0.04cm|}

\newcommand{\indiq}{{\bf 1}}

\begin{document}

\title[Keller-Segel equation]{Uniqueness and long time asymptotics for the parabolic-parabolic 
Keller-Segel equation}

\author{K. Carrapatoso$^{1}$, S. Mischler$^2$}

\thanks{(1) \'Ecole Normale Sup\'erieure de Cachan, 
CMLA, UMR CNRS 8536, 
61 av. du pr\'esident Wilson
94235 Cachan, 
FRANCE. E-mail: {\tt{carrapatoso@cmla.ens-cachan.fr}}
}
\thanks{(2)
IUF \& Ceremade, UMR CNRS 7534, UniversiteÃ Paris-Dauphine, PSL research university, Place de Lattre de Tassigny, 75775 Paris 16, France.
E-mail: {\tt{mischler@ceremade.dauphine.fr}}
}

\begin{abstract}
The present paper deals with the parabolic-parabolic Keller-Segel equation in the plane in
the general framework of weak (or ``free energy") solutions associated to initial data with finite mass $M< 8\pi$,  finite second log-moment and 
finite entropy. The aim of the paper is twofold:

(1) We prove the uniqueness of the ``free energy" solution. The proof uses a DiPerna-Lions renormalizing argument which makes possible 
to get the  ``optimal regularity" as well as an estimate of the difference of two possible solutions in the critical $L^{4/3}$ Lebesgue norm 
similarly as for the $2d$  vorticity Navier-Stokes equation. 

(2) We prove a radially symmetric and polynomial weighted $H^1 \times H^2$ exponential stability of the self-similar profile in the quasi parabolic-elliptic regime. The proof is based on a perturbation argument which takes advantage of the  exponential stability of the self-similar profile for the parabolic-elliptic Keller-Segel equation established by 
Campos-Dolbeault and Egana-Mischler. 

\end{abstract}

\maketitle

\begin{center} {\bf \today}
\end{center}


 \smallskip


%
%
%
%
%
%
%
%
	
\vspace{0.3cm}

\bigskip

\tableofcontents

 
\section{Introduction} 
\label{sec:intro}


\subsection{The KS equation, motivation and main biological result}
The Keller-Segel (KS) system (or Patlak-Keller-Segel system) for chemotaxis describes the collective motion of cells that are attracted by a chemical substance that 
they are able to emit (\cite{P,KS}).  In this paper we are concerned with the parabolic-parabolic KS model in the plane which takes
the form
\bear \label{eq:KS}
\partial_t f &=& \Delta f - \nabla (f \, \nabla u) \quad \hbox{in} \quad (0,\infty) \times \R^2, 
\\ \nonumber
\eps \partial_t u &=& \Delta u+ f - \alpha \, u  \quad \hbox{in} \quad (0,\infty) \times \R^2,
\eear
and which is complemented with an initial condition 
\beqn\label{eq:KSt=0}
f(0,\cdot) = f_0 \ge 0\quad \text{and} \quad
u(0,\cdot) = u_0  \ge 0
\quad \hbox{in}\quad \R^2. 
\eeqn 
Here $t \ge 0$ is  the time variable, $x \in \R^2$ is the space variable, $f=f(t,x) \ge 0$ stands for the {\it mass density of cells} while $u = u(t,x) \ge 0$ is the
{\it chemo-attractant concentration} and $\eps >0$, $\alpha\geq 0$ are constants. We refer to the work \cite{MR2433703} as well as to the reviews \cite{HillenPainter,TindallMPA} and the references quoted therein for biological motivation and mathematical introduction.

\smallskip

In short, the KS  equation models a cells population which is subject to two inverse mechanisms: 

\smallskip
- a brownian motion (responsable to the diffusion term $\Delta f$ in the first equation of \eqref{eq:KS}) modeling the fact that any cell  change of direction and move in a completely erratic way and which global effect is to spread out the population all over the plane $\R^2$; 

\smallskip 
- an aggregation mechanism (responsable to the drift term  $\nabla (-  \nabla u \, f )$  in the first equation of  \eqref{eq:KS})  
modeling the fact that cells  have a tendency to  follow the gradient lines of the chemo-attractant, which is itself produced and diffused according to the 
second equation in \eqref{eq:KS}. That mechanism has a concentration effect, which is quite strong due to the fact that the associated {\it interaction kernel} is singular. 

\smallskip 
From a mathematical point of view, both mechanisms are almost at the same order, and that makes the rigorous analysis of the model particularly difficult and interesting.

\medskip

Let us first discuss the case when we make the strong modeling and mathematical simplification $\eps=0$, that we refer to as the {\it parabolic-elliptic KS model}, which corresponds to the situation when the diffusion of chemo-attractant occurs with infinite speed (quasistatic approximation). 

\smallskip
The parabolic-elliptic KS system 
has been introduced by Patlak \cite{P} in the 1950's and by Keller, Segel \cite{KS}. It is has latter been rigorously justified from a more microscopical level of description. In particular, the parabolic-elliptic KS model has been obtained as a mean-field limit of a system of a  finite number of cells in interaction for regularized interaction kernel (at least at the level of the microscopic description) by Stevens \cite{Stevens}, Ha{\v{s}}kovec, Schmeiser \cite{HaskovecSchmeiser1,HaskovecSchmeiser2} and Godinho, Qui\~ninao \cite{GodinhoQuininao}.  The true interaction kernel at the level of the microscopic description has  been considered by Fournier, Jourdain \cite{FournierJourdain}, but they were only able to get a consistency result in the mean-field limit. 
The parabolic-elliptic KS has also been obtained as a diffusion limit of a run-and-tumble kinetic equation by Chalub et al. \cite{CMPS}, that last equation describing the cells population at a mesoscopic (or statistical) level.

\smallskip


During the last decades, a huge literature concerning the mathematical analysis of the parabolic-elliptic KS has grown up. 
We only give below some part of the main aspects. Probably the most important feature of the parabolic-elliptic KS equation is the existence of a mass threshold, the mass $M$ (or total number) of cells being conserved along time. 

For a supercritical mass $M > 8\pi$, there does not exist global in time nonnegative function solution:  {\it chemotactic collapse} occurs in finite time, or mathematically speaking, any solution blows up in finite time. Particular blowing up solutions have been exhibited by Herrero, Vel{\'a}zquez \cite{MR1627338} and the universality of that phenomenon (the stability under small perturbation of these blowing up solutions) has been considered by Rapha\"el,  Schweyer \cite{RS}. In other words, for large initial mass, the aggregation mechanism prevails on the spreading mechanism: the diffusion is not strong enough to prevent  mass concentration. 

On the other hand, for  a subcritical mass $M < 8\pi$, function solutions are well-defined globally in time and behave more similarly as for a pure diffusion. More precisely, after self-similar change of variables, any solution converges to the unique self-similar profile with same mass. In other words, for small initial mass, the spreading mechanism prevails on the aggregation mechanism, and no concentration occurs in finite nor infinite time.
We refer to Blanchet et al. \cite{BDP}, Campos,  Dolbeault \cite{MR2996772} and Ega{\~n}a, Mischler \cite{EM} for a detailed description of these results. For the critical mass $M=8\pi$, Blanchet et al. \cite{BCM} established that solutions are global in time and a complete concentration occurs in infinite time (any solution convergences to a Dirac distribution with mass $8\pi$). That qualitative analysis is mainly based on the existence of a  free-energy functional which behaves nicely along the flow of the parabolic-elliptic KS equation. 


\medskip

When the ratio of the speed of the cells over the speed of the chemo-attractant is very small, the parabolic-elliptic KS equation can be considered as a good approximation of \eqref{eq:KS}. However, from a biological modeling point of view the simplification $\eps=0$ is not completely satisfying. It is clearly irrelevant in the case when the speed of cells and chemo-attractant are of comparable  order, as it can be the case for a {\it Escherichia coli} population, see  Saragosti et al.  \cite{SaragostiCBBSP}. 

\smallskip
There are only very few works on the parabolic-parabolic KS system compared to the parabolic-elliptic KS system, and the parabolic-parabolic KS system is far from being well-understood.  
%
%
%
Let us however present some of available results.  

The parabolic-parabolic KS system  has been obtained as a diffusion limit from a kinetic equation by Erban, Hillen and Othmer \cite{HillenOthmer1,OthmerHillen2,ErbanOthmer4,ErbanOthmer3,ErbanOthmer2}. But to our knowledge, no derivation as a mean-field limit of a microscopic cell-system has been performed. Concerning the qualitative behaviour of solutions, the  threshold between blowing up solutions and solutions that spread out is not clear. 
It is known though that solutions are global in time for sub-critical masses $M \le 8\pi$, see Calvez, Corrias \cite{MR2433703}. It is also known that solutions are global for any mass when $\eps>0$ is large enough, see Biler et al. \cite{BGK} and Corrias et al.  \cite{CEM}, that regime corresponds to a small chemo-attractant production by cells, and thus to a small nonlinearity. 
Chemotactic collapsing solutions have been exhibited by Herrero, Vel{\'a}zquez \cite{MR1627338} for supercritical masses. On the other hand, for any sub-critical mass (and for any mass if $\eps >0$ is large enough) unique self-similar solutions have been constructed by Biler et al. \cite{BCD}, see also Corrias et al. \cite{CEM}.

One of the major difficulty for performing a qualitative analysis of generic solutions to the parabolic-parabolic KS system is that the free-energy functional does not behave as nicely as for  the parabolic-elliptic KS system, in particular it does not provide any Lyapunov functional after self-similar rescaling of variable. 

\smallskip
Our work concerns the general parabolic-parabolic KS equation $\eps > 0$ and  the quasi-parabolic-elliptic regime, corresponding to the case when $\eps >0$ is small,  for which a deeper mathematical analysis of the qualitative  behavior of solutions can be performed. 
Our main result shows that  for some class of initial data with sub-critical mass and in a quasi-parabolic-elliptic regime, 
the associated solution $(f,u)$ to the parabolic-parabolic KS system \eqref{eq:KS} satisfies
\beqn\label{eq:fuTOGV}
f(t,x) \sim \frac1t \, G_\eps \left( \frac{x}{t^{1/2}} \right), \quad
u(t,x) \sim   V_\eps \left( \frac{x}{t^{1/2}} \right), \quad\text{as}\quad t \to \infty,
\eeqn
for some self-similar profile $(G_\eps,V_\eps)$. In other words, we prove, for the very first time, that the parabolic-parabolic KS system \eqref{eq:KS} behaves in that regime similarly as a diffusion: no concentration occurs (even in large time) and the diffusion phenomenon is really the dominant phenomenon for any time.

\subsection{Mathematical analysis of the KS equation}

The two fundamental identities associated to the Keller-Segel equation \eqref{eq:KS} are that any solution satisfies, at least formally, the conservation of {\it ``mass"}
\beqn\label{eq:MassConserv}
M(t) := \langle f(t,.) \rangle = \langle f_0 \rangle = : M, \quad\hbox{with} \quad \langle g \rangle := \int_{\R^2}g(x) \, dx, 
\eeqn
and the {\it ``free energy-dissipation of the free energy identity"} 
\beqn\label{eq:FreeEnergy}
\FF(t) + \int_0^t \DD_\FF(s) \, ds  = \FF_0,
\eeqn
where the free energy $\FF(t) = \FF(f(t),u(t))$, $\FF_0 = \FF(f_0,u_0)$ is defined by 
\beqn\label{def:FF}
\FF = \FF(f,u) := \int_{\R^2} f \log f \, dx -  \int_{\R^2} f u \, dx + {1 \over 2} \int_{\R^2} |\nabla u |^2  \, dx +{\alpha \over 2}  \int_{\R^2} u^2dx, 
\eeqn
and the dissipation of free energy $\DD_\FF(s) = \DD_\FF(f(s),u(s))$ by 
\beqn\label{def:DD}
\DD_\FF = \DD_\FF(f,u) := \int_{\R^2} f \, |\nabla (\log f) - \nabla u |^2 \, dx + {1 \over \eps} \int_{\R^2} |\Delta u + f - \alpha \, u |^2 \, dx. 
\eeqn

\smallskip
Following \cite{MR2433703}, throughout this paper, we shall assume that the initial data 
$(f_0,u_0)$ satisfy
\beqn\label{eq:BdsInitialDatum}
\left\{
\bal
& f_0 \, (1 + \log \langle x \rangle^2 ) \in L^1 (\R^2)
\quad\text{and}\quad
f_0 \log f_0 \in L^1(\R^2);  \\
& u_0 \in H^1(\R^2)\; \text{if } \alpha > 0
\quad\text{or}\quad
u_0 \in L^1(\R^2) \cap \dot H^1(\R^2)  \; \text{if } \alpha=0;  \\
& f_0 \, u_0 \in L^1(\R^2),
\eal
\right.
\eeqn
where here and below we define the weight function $\langle x \rangle := (1 + |x|^2)^{1/2}$ and the homogeneous Sobolev space $\dot H^1(\R^2)$
is defined by $\dot H^1(\R^2) := \{ u \in L^1_{loc}(\R^2); \,  \nabla u  \in L^2(\R^2) \}$.
We also make the important restriction of subcritical mass 
$$
M := \langle f_0 \rangle \in (0,8\pi),
$$
as a suitable global existence theory is available in that case (see \cite{MR2433703,MR3116019}), whereas for $M > 8\pi$ there exist solutions which blow up in finite time (see \cite{MR1627338,MR1799300,MR1970697} and the discussion in \cite[1.~Introduction]{MR3116019}). 
We also refer to \cite{BGK,CEM} where a global existence theory is developed in the possible supercritical case $M > 8\pi$ and the condition that   $\eps $ is large enough (which corresponds to a case where the nonlinearity in \eqref{eq:KS} is small). 


\smallskip
As in \cite{MR2433703}, we consider the following definition of weak solution. 

\begin{defin}\label{def:sol} For any initial datum $(f_0, u_0)$ satisfying \eqref{eq:BdsInitialDatum} with $M < 8\pi$, we say that the couple $(f,u)$ of nonnegative functions satisfying
\beqn\label{eq:bddL1}
\bal
&\quad
f \in L^\infty(0,T; L^1(\R^2)) \cap C([0,T); \DD'(\R^2)), \quad \forall \, T \in (0,\infty), \\
&
\begin{cases}
u \in L^\infty(0,T; H^1(\R^2)) & \text{if } \alpha>0 ;\\
u \in L^\infty(0,T; L^1(\R^2) \cap \dot H^1(\R^2)) & \text{if } \alpha = 0; 
\end{cases} 
\\
&\quad
fu \in L^\infty(0,T; L^1(\R^2))
\eal
\eeqn
is a global in time weak solution to the Keller-Segel equation associated to the initial condition $(f_0, u_0)$ whenever $(f,u)$ satisfies the mass conservation 
\eqref{eq:MassConserv},  the bound  
\beqn\label{eq:bddFD}
\sup_{[0,T]} \FF(t) + \sup_{[0,T]}  \int_{\R^2} f \,  \log \langle x \rangle^2 \, dx  + \int_0^T \DD_\FF(t) \, dt \le C_T, 
\eeqn
as well as the Keller-Segel system of equations \eqref{eq:KS}-\eqref{eq:KSt=0} in the distributional sense, namely
\bear\label{eq:KSweak}
 \int_{\R^2} f_0(x) \, \varphi(0,x) \, dx  
 &=& \int_0^T \int_{\R^2} f  \, \Bigl\{ (\nabla_x (\log f) - \nabla_x u  )\cdot \nabla_x \varphi - \partial_t \varphi   \Bigr\} \, dx \, dt \\
\eps \int_{\R^2} u_0(x) \, \psi(0,x) \, dx  
&=&  \int_0^T \int_{\R^2} \{ u  \, ( -\Delta \psi + \alpha \psi - \eps \partial_t \psi   ) - f(t,x) \psi \} \, dx \, dt , 
\eear
for any $T > 0$ and $\varphi, \psi \in C^2_c([0,T) \times \R^2)$. 
\end{defin}

It is worth emphasizing  that it is not assumed that the free energy-dissipation of the free energy identity  \eqref{eq:FreeEnergy} holds, but with \eqref{eq:bddFD}, only that the  feee energy and the dissipation of the free energy are  bounded. 
Thanks to the Cauchy-Schwarz inequality, we have 
\bean
\int_{\R^2}  f \,  | \nabla_x (\log f) -  \nabla_x u  | \, dx 
\le M^{1/2} \, \DD_\FF^{1/2},
\eean
and the RHS of \eqref{eq:KSweak} is then well defined thanks to \eqref{eq:MassConserv} and  \eqref{eq:bddFD}. 

\smallskip
This framework is well adapted for a global existence theory in the subcritical mass case.

\begin{theo}(\cite[Theorem 1]{MR2433703}) \label{theo:exist} For any initial datum $(f_0,u_0)$ satisfying \eqref{eq:BdsInitialDatum}
and $M < 8\pi$, there exists at least one global in time weak solution in the sense of Definition~\ref{def:sol} to the Keller-Segel equation  \eqref{eq:KS}-\eqref{eq:KSt=0}. 
  \end{theo}

Our first main result establishes that this framework is also well adapted for the well-posedness issue. 

\begin{theo}\label{theo:uniq} For any initial datum $(f_0,u_0)$ satisfying \eqref{eq:BdsInitialDatum} with $M < 8 \pi$, there 
exists at most one weak solution in the sense of Definition~\ref{def:sol} to the Keller-Segel equation \eqref{eq:KS}-\eqref{eq:KSt=0}.
This one is furthermore a classical solution in the sense that 
\beqn\label{eq:f&uC2b}
f,u \in C^2_b((0,\infty) \times \R^2)
\eeqn
and satisfies the accurate small time estimate
\beqn\label{eq:intro:fto0}
\forall\, q \in (1,\infty), \qquad 
t^{1 - \frac1q} \, \| f (t) \|_{L^{q}} \to 0 \quad\hbox{as}\quad t \to 0.
\eeqn
Finally, the free energy-dissipation of the free energy identity \eqref{eq:FreeEnergy} holds. 
\end{theo}

Theorem~\ref{theo:uniq} improves the uniqueness result proved in \cite{CLM} in the class of solutions $f \in C([0,T]; $ $L^1_2(\R^2)) \cap L^\infty((0,T) \times \R^2)$, $\forall \, T > 0$, which can be built under the additional assumption $f_0 \in L^\infty(\R^2)$ as well as the uniqueness part of the well-posedness results   \cite{Bil98,MR2785976,CEM,BGK}  for solutions satisfying \eqref{eq:intro:fto0} which are established  in some particular regimes (smallness assumption on the initial datum or on some parameters). We also refere to \cite{MR1654677} where a uniqueness result is established for a related model.
Our proof follows a strategy introduced in \cite{FHM} for the 2D viscous vortex model and generalizes a similar result obtained in \cite{EM} for the parabolic-elliptic model (which corresponds to the case $\eps=0$). 
It is based on a DiPerna-Lions renormalization process (see \cite{dPL}) which makes possible to get the optimal regularity of solutions for small time \eqref{eq:intro:fto0} and 
then to follow the uniqueness argument introduced by Ben-Artzi for the 2D viscous vortex model (see \cite{MR1308857,MR1308858}) and  also used in \cite{Bil98,MR2785976,CEM,BGK} for the parabolic-parabolic Keller-Segel equation. Ben-Artzi's argument consists in writing the mild formulation of the difference of two solutions and to establish a contraction estimate for the norm introduced in \eqref{eq:intro:fto0} for $q=4/3$. The choice of the exponent $q=4/3$ is crucial and it is made in order to handle the singularity of the force field (thanks to sharp estimates of the smoothing effect of  the heat semigroup). 
It is worth emphasizing that such an argument is related to famous Kato's work on the Navier-Stokes equation (see e.g. \cite{MR760047}). 

\smallskip

The smoothing effect and the free energy-dissipation of the free energy identity established in Theorem~\ref{theo:uniq} are natural and physically relevant but new, even for stronger (and possibly local in time) notions of solutions.  

\bigskip

From now on in this introduction, we definitively restrict ourselves to the case $\alpha = 0$ and we focus on the long time asymptotic of the solutions. 
For that last purpose it is convenient to work with self-similar variables. We  introduce the rescaled functions $g$ and $v$ defined by 
\beqn\label{eq:gt=}
 f(t,x) := R(t)^{-2} g(\log R(t), R(t)^{-1} x), \quad 
u(t,x) :=  v(\log R(t), R(t)^{-1} x),  
\eeqn
with $R(t) := (1 + t)^{1/2}$. For these new unknowns, the rescaled parabolic-parabolic Keller-Segel system reads  
\bear \label{eq:KSresc}
&& \partial_t g =  \Delta g + \nabla ({1 \over 2} \, x \, g - g  \, \nabla v)  \quad \hbox{in} \quad (0,\infty) \times \R^2,
\\ \label{eq:KSresc2eme}
&&\eps \partial_t v =      \Delta v + g + {\eps \over 2} \, x \cdot \nabla v  \quad \hbox{in} \quad (0,\infty) \times \R^2.
\eear

We are interested in self-similar solutions to the Keller-Segel parabolic-parabolic equation  \eqref{eq:KS}, that is solutions which write as
$$
f(t,x) = {1 \over t} G_\eps( {x \over t^{1/2}}), \quad u(t,x) = V_\eps( {x \over t^{1/2}}),
$$
with 
\beqn\label{eq:KSprofileMass-intro}
\int_{\R^2} f(t,x) \, dx = \int_{\R^2} G_\eps(y) \, dy = M \in (0,8\pi).
\eeqn
Such a couple of functions $(f,u)$ is a solution to \eqref{eq:KS} if, and only if, the associated ``self-similar profile" $(G_\eps,V_\eps)$ satisfies the elliptic system
\bear \label{eq:KSprofile-intro}
&& \Delta G_\eps - \nabla (G_\eps \, \nabla V_\eps - {1 \over 2} \, x \, G_\eps) = 0 \quad \hbox{in} \quad   \R^2, 
\\ \nonumber
&& \Delta V_\eps +{\eps \over 2} \, x \cdot \nabla V_\eps + G_\eps = 0   \quad \hbox{in} \quad  \R^2,
\eear
and thus corresponds to a stationary solution to the  rescaled parabolic-parabolic Keller-Segel system  \eqref{eq:KSresc}. 
It is known that, for any $\eps \in (0,1/2)$ and any $M \in (0,8\pi)$, there exists a unique solution 
$(G_\eps,V_\eps)$ to \eqref{eq:KSprofile-intro} such that the mass of $G_\eps$ equals $M$, and which is furthermore radially symmetric and smooth (say $C^2(\R^2)$), see \cite{NSY,BCD,CEM}. 

\medskip

Our second main result concerns the exponential nonlinear stability of the self-similar profile for any given mass $M \in (0,8\pi)$ under the strong
restriction of radial symmetry and closeness to the parabolic-elliptic regime. We define the norm
$$
\Nt (g,v)   \Nt := \| g \|_{H^1_k} + \| v \|_{H^2}, \quad k>7,
$$
where the weighted Lebesgue space $L^p_k(\R^2)$, for $1 \le p \le \infty$, $k \ge 0$,  is defined by 
$$
L^p_k(\R^2) := \{ f \in L^1_{loc}(\R^2); \,\, \| f \|_{L^p_k} := \| f \, \langle x \rangle^k \|_{L^p} < \infty \},
$$
and the norm of the higher-order Sobolev spaces $W^{\ell,p}_k (\R^2)$ is defined by 
$$
\| f \|_{W^{\ell,p}_k}^p := \sum_{|\alpha| \le \ell} \| \la x \ra^k \, \partial^\alpha f \|_{L^p}^p.
$$

\begin{theo}\label{theo:stab} For any given mass $M \in (0,8 \pi)$, there exist $\eps^* > 0$ and $\delta^* > 0$ such that for any $\eps \in (0,\eps^*)$ and any radially symmetric initial datum $(g_0,v_0)$ satisfying   
$$
\Nt (g_0,v_0) - (G_\eps,V_\eps) \Nt \le \delta^*, \quad \int_{\R^2} g_0 \, dx = \int_{\R^2} G_\eps  \, dx = M, 
$$
the associated solution $(g,v)$ to \eqref{eq:KSresc}-\eqref{eq:KSresc2eme} satisfies 
$$
\Nt (g(t),v(t)) - (G_\eps,V_\eps) \Nt \le C_a \, e^{at}  \quad  \forall \, a \in  (-1/2,\infty), \,\, \forall \, t \ge 0,
$$
for some constant $C_a = C_a(g_0,v_0) \ge 1$.  \end{theo}

Coming back to the original unknowns $(f,u)$, this theorem asserts \eqref{eq:fuTOGV} holds. 

\medskip

That result extends to the parabolic-parabolic Keller-Segel equation similar results known on the parabolic-elliptic Keller-Segel equation,
see \cite{EM}. 
To our knowledge, Theorem~\ref{theo:stab} is the first exponential stability result for the system~\eqref{eq:KS} even under the two strong restrictions of radial symmetry and quasi parabolic-elliptic regime (we mean $\eps > 0$ small), and, moreover, the rate obtained here can be taken as close as we want to the optimal rate $- 1/2$ of the parabolic-elliptic case (see Theorem~\ref{Theorem:theoSOmega}). 
However, we refer again to the recent work \cite[Section 4]{CEM} where some results of convergence (without rate) of some solutions to the associated self-similar profile are established. 
We also refer to that work for further discussion and additional references. We finally refer to \cite{ZinslMatthes} where an exponential convergence to
the equilibrium for a somehow similar chemotaxis model is established. 
 
\smallskip

The proof of Theorem~\ref{theo:stab} is based on the spectral analysis of the associated linearized operator and on growth estimates on the related linear semigroup. We take advantage of the similar analysis made on the linearized parabolic-elliptic Keller-Segel equation performed in \cite{EM}
and we use a (quite singular) perturbation argument. The restriction to radially symetric functions is made in order to get sharp and convenient estimates on solutions to auxilliary elliptic equations (see Lemma~\ref{lem:du=fbis}) used in the accurate spectral analysis performed in Theorem~4.7. We think that this additional assumption is only technical and can be circumvented, but we were not able to fix it.

\medskip
Let us end the introduction by describing the plan of the paper. 
In Section~\ref{sec:aposteriori} we present some functional inequalities which will be useful in the sequel of the paper and 
we establish several a posteriori bounds satisfied by any weak solution.
Section~\ref{sec:uniqueness} is dedicated to the proof of the uniqueness result stated in Theorem~\ref{theo:uniq}. 
In Section~\ref{sec:autosimilaire},  we prove the exponential stability of the linearized problem associated to  \eqref{eq:KSresc}-\eqref{eq:KSresc2eme}.
Finally, in Section~\ref{sec:nonlinear},  we prove the long-time behaviour result as stated in Theorem~\ref{theo:stab}. 

\medskip
\paragraph{\bf Acknowledgments.} We thank J.-Y. Chemin, J. Dolbeault, I. Gallagher, G. Jankowiak and O.~Kavian for fruitful discussions and for having pointed out some interesting references related to our work. The research leading to this paper was (partially) funded by the French "ANR blanche" project Kibord: ANR-13-BS01-0004. K. C.\ is supported by the Fondation Math\'ematique Jacques Hadamard.

\section{Local in time a priori and a posteriori estimates  } 
\label{sec:aposteriori}


\subsection{A priori estimates}  In this short paragraph, we follow  \cite{MR2433703} and we explain  how to obtain the basic estimates which lead to the notion of weak solution as presented in Definition~\ref{def:sol}. 
We first observe that the following space logarithmic moment control holds true
\bean
{d \over dt} \int_{\R^2} f \,  (- \log H) dx 
&=& \int_{\R^2} f  \, \nabla (\log f - u) \cdot \nabla (\log H) \, dx
\\
&\le&   {\delta \over 2} \, \DD_\FF (f,u) + {1 \over 2\delta} \int_{\R^2} f  \, | \nabla \log H|^2  \, dx,
\eean
where 
$$
H(x) := {1 \over \pi}\, {1 \over \langle x \rangle^4} \quad \hbox{and then} \quad
|\nabla \log H(x)| \le 4 ,
$$
which together with \eqref{eq:FreeEnergy} implies that the modified free energy functional
$$
\FF_H = \FF(f,u) -  \int_{\R^2} f \, \log H 
$$
 satisfies
\beqn\label{eq:ineqFH-D}
{d \over dt}\FF_H (t) + {1 \over 2} \DD_\FF(t) \le M.
\eeqn

%

On the one hand, introducing the Laplace kernel $\kappa_0 (z) := - {1 \over 2\pi} \log |z|$ and the  Bessel kernel  $\kappa_\alpha (z) := \frac{1}{4\pi} \int_0^\infty t^{-1} \exp (-|z|^2/(4t) - \alpha t) \, dt$ for $\alpha > 0$, so that $\bar u = \bar u_\alpha := \kappa_\alpha * f$ is a solution to the Laplace type equation 
$$
- \Delta \bar u = f - \alpha \bar u \quad \hbox{in} \quad \R^2,
$$
and introducing as well  the {\it chemical energy} and the modified entropy
$$
F_\alpha (f,u) :=   {1 \over 2} \int |\nabla u|^2   + {\alpha \over 2}  \int u^2  - \int f \, u ,
\quad \HH_H(f) := \int f \log ( f/H),
$$
one can easily show (see e.g. \cite[Lemma 2.2]{MR2433703}) that
\beqn\label{eq:Hbd2}
\FF_H(f,u) = \HH_H(f) + F_\alpha(f,\bar u_\alpha) + {1 \over 2} \int |\nabla (u-\bar u_\alpha)|^2 + {\alpha \over 2} \int (u-\bar u_\alpha)^2
\eeqn
and 
\beqn\label{eq:Hbd3}
F_\alpha(f,\bar u_\alpha) = - {1 \over 2} \int\!\! \int_{\R^2 \times \R^2} f(x) \, f(y) \, \kappa_\alpha (x-y)\, dxdy.
\eeqn
On the other hand, we know from the classical logarithmic Hardy-Littlewood-Sobolev inequality (see e.g. \cite{MR1230930,MR1143664})
or its generalization for the Bessel kernel (see \cite[Lemma 4.2]{MR2433703}) that
\bear\label{eq:logHLSineqBessel}
\forall \, f \ge 0, \quad 
\int_{\R^2} f (x) \log f(x) \, dx &-& {4\pi \over M} \int\!\! \int_{\R^2 \times \R^2} f(x) \, f(y) \, \kappa_\alpha (x-y)\, dxdy \\
\nonumber
&-& \int_{\R^2} f (x) \log H(x) \, dx 
\ge -C_1(M),
\eear
where here and below $C_i(M)$ denotes a positive constant which only depends on the mass $M$. 


Then from  \eqref{eq:Hbd2}, \eqref{eq:Hbd3} and \eqref{eq:logHLSineqBessel} together with the 
 very classical functional inequality (see e.g. \cite[Lemma 2.4]{MR2433703}) 
\beqn\label{eq:H+<H}
\HH^+ := \HH^+(f) = \int f (\log f)_+ \le \HH_H (f) -  {1 \over 4} \int f \log \langle x \rangle^2+ C_2(M),  
\eeqn
one immediately obtains, for $M < 8\pi$,
\bean 
\FF_H(f,u) &\ge&   (1 - {M \over 8 \pi}) \, \HH_H(f) + {M \over 8\pi} \Bigl( \HH_H(f) -  {4\pi \over M} \int\!\! \int_{\R^2 \times \R^2} f(x) \, f(y) \, \kappa_\alpha (x-y)\, dxdy \Bigr) 
\\
&\ge&  C_3(M) \, \HH_+(f)  + C_4(M) \int f \log \langle x \rangle^2 - C_5(M).
\eean
One concludes that under the assumption \eqref{eq:BdsInitialDatum} on the initial datum, the identity  \eqref{eq:MassConserv} and the inequality \eqref{eq:ineqFH-D} provide a convenient family of a priori  estimates in order to define weak solutions, namely
\bear\label{eq:F<H}\quad
 &&C_3(M) \, \HH^+(f(t)) + C_4(M) \int f(t) \log \la x \ra^2 + {1 \over 2}  \int_0^t \DD_\FF(f(s), u(s)) \, ds  
   \\ \nonumber
 &&\quad \le    \FF_H(0)  + C_5(M) +  M \, t,
\eear
and one remarks that the RHS term is finite under assumption \eqref{eq:BdsInitialDatum} on $(f_0,u_0)$, because 
$$
\bal
\FF_H(0) &= \FF(f_0,u_0) - \int f_0 \log H \\
&= \FF(f_0,u_0) + M \log \pi + 2 \int f_0 \log \la x \ra^2 < + \infty.
\eal
$$
It is worth emphasizing that in order to get the bounds announced in Definition~\ref{def:sol} in the case $\alpha > 0$ one may use the inequality
\beqn\label{eq:ineqFH-u}
\FF_H \ge C_6(M) \int |\nabla u|^2 + C_7(M) \alpha \int u^2 + C_8(M) \int f \, u -  C_9(M) (1 + 1/\alpha)
\eeqn
which is established in \cite[(3.5)]{MR2433703}.

\subsection{Local in time a posteriori estimates  } 
\label{ssec:aposteriori}

We start by presenting some elementary functional inequalities which will be of main importance in the sequel. 
The two first estimates are picked up from \cite[Lemma 3.2]{FHM} but are probably classical and the third one
is a variant of the Gagliardo-Niremberg-Sobolev inequality.  

\begin{lem} \label{lem:FishInteg1}
For any $0 \le f \in L^1(\R^2)$ with finite mass $M$ and finite Fisher information
$$
I = I(f) := \int_{\R^2} {|\nabla f |^2 \over f}, 
$$
 there holds 
\bear
\label{eq:LpbdFisher} 
&\forall \; p\in [1,\infty), \quad \| f \|_{L^p(\R^2)} \leq C_p \,  M^{1/p} \,  I(f)^{1 - 1/p},\\
\label{eq:nablaLqbdFisher}
&\forall \; q\in [1,2), \quad
\| \nabla f \|_{L^q(\R^2)} \leq C_q  \, M^{1/q - 1/2} \,   I(f)^{ {3/2}-{1/q}}.
\eear
For any $0 \le f \in L^1(\R^2)$ with finite mass $M$, there holds 
\bear \label{eq:GNp+1}
&&\forall \; p \in [2,\infty) \quad
\| f \|_{L^{p+1}(\R^2)} \leq C_p  \, M^{1/(p+1)} \,   \|\nabla (f^{p/2}) \|_{L^2}^{2/(p+1)}.
 \eear

\end{lem}

We refer to \cite[Lemma 3.2]{FHM} and \cite[Lemma 2.1]{EM} for a proof.

\medskip
 
The proof of \eqref{eq:f&uC2b} in Theorem~\ref{theo:uniq} is split into several steps that we present as some 
intermediate autonomous a posteriori bounds.

\begin{prop}\label{prop:Fisher} 
For any weak solution $(f,u)$, we have 
$$
I(f(t)) \in L^1(0,T), \quad \forall \, T >0. 
$$
\end{prop}
 
 \noindent
{\sl Proof of Proposition~\ref{prop:Fisher}. } We write 
\beqn\label{eq:D>I}
\DD_\FF(f)  \ge I(f) + 2 \int f \, \Delta u.
\eeqn
Next, by Young's inequality, we have
\bean
 \int f \, \Delta u 
 &=& \int f \, (\eps \partial_t u - f + \alpha u)
 \\
 &\ge&- (1+\alpha/2+\eps/2)  \int f^2  -  \eps/2 \int(\partial_t u)^2  - \alpha/2 \int u^2 .
 \eean
The second and third terms belong to $L^1(0,T)$ from \eqref{eq:bddFD}, so we only need to estimate the first term. 
 
For any $A > 1$, using the Cauchy-Schwarz inequality and the inequality \eqref{eq:LpbdFisher}
for $p=3$, we have 
\bean
\int f^2 \, {\bf 1}_{f \ge A} 
&\le&
 \Bigl( \int f \, {\bf 1}_{f \ge A} \Bigr)^{1/2}  \Bigl( \int f^3\Bigr)^{1/2} 
\\ \nonumber
&\le&
 \Bigl( \int f \, {(\log f )_+\over \log A}  \Bigr)^{1/2}  \Bigl( C^3_3 \, M  \, I(f)^{2}\Bigr)^{1/2} ,
\eean
from what we deduce for $A = A(M,\HH^+)$ large enough, and more precisely taking $A$ such that $\log A = 16 \, \HH_+ \, C_3^3 \, M (1+\alpha/2+\eps/2)^2$, 
\beqn\label {eq:f2<I}
\int f^2 \, {\bf 1}_{f \ge A}  \le  C_3^{3/2} \, M^{1/2} \, {\HH_+(f)^{1/2} \over (\log A)^{1/2}} \, I(f) \le {(1+\alpha/2+\eps/2)^{-1} \over 4} \, I(f).
\eeqn
Denoting $\Phi(u) = \eps \int (\partial_t u)^2 + \alpha \int u^2 \in L^1(0,T)$ and putting together the last estimate with \eqref{eq:D>I}, it follows  
\bean
{1 \over 2} I(f) 
&\le& \DD_\FF  + C \, \int f^2 \, {\bf 1}_{f \le A} + \Phi(u)
\\
&\le& \DD_\FF  + 2 \, M \, \exp (  C \, \HH_+  \, M) + \Phi(u), 
\eean
and we conclude thanks to \eqref{eq:MassConserv}--\eqref{eq:F<H}. 
\qed

 \begin{rem}\label{rem:Msupercritic} The logarithmic Hardy-Littlewood Sobolev inequality \eqref{eq:logHLSineqBessel} in the supercritical case $M \ge 8\pi$ does not lead to a global estimate as for the subcritical case $M \in (0,8\pi)$. However, introducing the function $\MMM := M H$ of mass $M$ and the modified free energy
$$
\tilde \FF_M(f,u) := \int_{\R^2} (f \, \log (f/\MMM) - f + \MMM) \, dx  -  \int_{\R^2} f u \, dx + {1 \over 2} \int_{\R^2} |\nabla u |^2  \, dx +{\alpha \over 2}  \int_{\R^2} u^2 \, dx ,
$$
one shows that any solution $(f,u)$ to the Keller-Segel equation \eqref{eq:KS} formally satisfies 
\bean
{d \over dt}  \tilde \FF_M(f,u)
&\le& - {1 \over 2} \DD_\FF(f,u) +  M
\\
&\le& - {1 \over 2} I(f) - \int f \Delta u - {\eps \over 2} \int (\partial_t u)^2 + M
\\
&\le& - {1 \over 2} I(f) -   {\eps \over 4} \int (\partial_t u)^2 + (1+\eps) \int f^2 + {\alpha \over 2} \int u^2 + M,
\eean
where we have just used \eqref{eq:ineqFH-D}, the estimate \eqref{eq:D>I} and the estimates at the beginning of the proof of Proposition~\ref{prop:Fisher}. We also observe that from \eqref{eq:H+<H} and \eqref{eq:ineqFH-u}, we may deduce
$$
\HH_+(f) + \int |\nabla u|^2 + \alpha \int u^2 \le K_1 \tilde \FF_M(f,u) + K_2 ,
$$
where $K_i$, $i=1,2$, are constants which may depend on $M > 0$ and $\alpha \ge 0$. Arguing then as in the proof of Proposition~\ref{prop:Fisher}, we easily get 
\bean
{d \over dt}  \tilde \FF_M(f,u)
&\le&  - {1 \over 2}  I(f) -   {\eps \over 4} \int (\partial_t u)^2 + (1+\eps)  
\Bigl\{  M A + C_3^{3/2} \, M^{1/2} \, {(K_1 \tilde \FF_M(f,u))^{1/2} \over (\log A)^{1/2}} \, I(f) \Bigr\}
\\
&& + {K_1 \over 2}  \tilde \FF_M(f,u) + {K_2 \over 2}  + M   \qquad (\forall \, A > 0)
\\
&\le&  - {1 \over 4}  I(f) -   {\eps \over 4} \int (\partial_t u)^2 +  K_3 \, \exp ( K_4 \,  \tilde \FF_M(f,u)) + K_5 ,
\eean
by making the appropriate  choice $\log A = K' \,  \tilde \FF_M(f,u)$ for $A$. 
This differential inequality provides a local a priori estimate on the modified free energy which can be used  in order to prove local existence result for supercritical mass. Because we will prove in Theorem~\ref{theo:uniq} that the above resulting bound is suitable in order to get the uniqueness of the solution, we can classically obtain the existence and uniqueness of maximal solutions (in the weak sense of Definition~\ref{def:sol}) $(f,u) \in C([0,T^*);\DD'(\R^2) \times \DD'(\R^2))$ such that 
$$
\sup_{[0,T)}  \tilde \FF_M(f(t),u(t)) + \int_0^T \Bigl\{  I(f(t)) + \DD_\FF(f(t),u(t)) \Bigr\} \, dt < \infty \quad \forall \, T \in (0,T^*)
$$
and the alternative 
$$
T_* = +\infty   \quad\hbox{or}\quad (T_* < \infty, \,\, \tilde \FF_M(f(t),u(t)) \to \infty \,\, \hbox{ as } \,\, t \to T^*).
$$
\end{rem}

\medskip
 
As an immediate consequence of Lemma~\ref{lem:FishInteg1} and Proposition~\ref{prop:Fisher}, we have 

\begin{lem}\label{lem:Bdf&K} For any $T > 0$, any  weak solution $(f,u)$ satisfies 
\bear \label{eq:wLqtLpx}
&& f \in L^{p/(p-1)}(0,T;L^p(\R^2)), \quad \forall \; p\in(1,\infty), 
\\
\label{eq:gradfLqtLpx}
&&\nabla f \in L^{2p/( 3p-2)} (0,T; L^p(\R^2)), \quad \forall \, p \in [1,2), 
\\
\label{eq:DeltauL2} 
&&  \Delta u   \in L^{2}(0,T;L^2(\R^2)).
\eear

\end{lem}

\begin{proof}[Proof of Lemma~\ref{lem:Bdf&K}] 
The bound \eqref{eq:wLqtLpx} is a direct consequence of \eqref{eq:LpbdFisher} and Proposition~\ref{prop:Fisher}. 
The bound \eqref{eq:gradfLqtLpx} is a consequence of \eqref{eq:nablaLqbdFisher} and Proposition~\ref{prop:Fisher}.
From \eqref{eq:bddFD} we have $\DD_\FF \in L^1(0,T)$, which implies $\Delta u + f - \alpha u \in L^2(0,T; L^2(\R^2))$ and then, using \eqref{eq:wLqtLpx}, we obtain \eqref{eq:DeltauL2}.
\end{proof}

\begin{lem}\label{lem:Bdbeta} Any weak solution $(f,u)$ satisfies 
\bear\label{eq:betaf}
&&\int_{\R^2} \beta(f_{t_1})  \, dx  + \int_{t_0}^{t_1} \int_{\R^2} \beta''(f_s) \, |\nabla f_s|^2 \, dx ds 
  \\ \nonumber
  &&\quad\le \int_{\R^2} \beta(f_{t_0})  \, dx  + \int_{t_0}^{t_1} \int_{\R^2} \{\beta(f_s) - f_s \beta'(f_s) \} \Delta u_s   \, dx ds ,
\eear
for  any times $0 \le t_0 \le t_1 < \infty$ and any renormalizing function $\beta : \R \to \R$ which is convex, piecewise of class $C^1$ and such that 
$$
|\beta (\xi)| \le C \, (1 + \xi \, (\log \xi)_+), \quad |\beta(\xi)  - \xi \beta'(\xi) |  \le C \,  \xi \quad \forall \, \xi \in \R.
$$
\end{lem}
 
\noindent
{\sl Proof of Lemma~\ref{lem:Bdbeta}.} We consider a weak solution $(f,u)$ to the Keller-Segel equation, and we write, in the distributional sense, 
$$
\partial_t f = \Delta f  - \nabla u  \cdot \nabla f  - (\Delta u) \, f.
$$
We split the proof into three steps. 
 
\smallskip\noindent
{\sl Step 1. Continuity. } 
Consider a mollifier sequence $(\rho_n)$ on $\R^2$, that is $\rho_n(x) := n^2 \rho(nx)$, $0 \le \rho \in \DD(\R^2)$, $\int \rho = 1$,  and introduce the
mollified function $f^n_t := f_t *_x \rho_n$. Clearly, $f^n  \in C([0,T); L^1(\R^2))$.
Using \eqref{eq:wLqtLpx} and \eqref{eq:bddFD}, a variant of the commutation
Lemma \cite[Lemma II.1 and  Remark 4]{dPL} tells us that 
\beqn\label{eq:weps}
\partial_t f^n = \Delta f^n  - \nabla u  \cdot \nabla f^n  - (\Delta u) \, f^n + r^n, 
\eeqn
with $r^n = r^n_1 + r^n_2$ given by
$$
\bal
r^n_1 &:= \nabla u \cdot \nabla f^n - (\nabla u \cdot \nabla f)* \rho_n \to 0  \quad\hbox{in}\quad L^1(0,T; L^1_{loc}(\R^2)), \\
r^n_2 &:= (\Delta u) f^n - [(\Delta u) f]*\rho_n \to 0 \quad\hbox{in}\quad L^1(0,T; L^1_{loc}(\R^2)).
\eal
$$
The important point here is that $f \in L^2(0,T; L^2(\R^2))$ thanks to \eqref{eq:wLqtLpx} and $\nabla u \in L^2(0,T; W^{1,2}(\R^2))$ thanks to \eqref{eq:DeltauL2}, hence the commutation lemma holds true.  

As a consequence,  the chain rule applied to the smooth function $f^n$ reads 
\beqn\label{eq:betaweps}
\partial_t \beta(f^n) =  \Delta  \beta(f^n) -  \beta''(f^n) \, |\nabla f^n|^2 
- \nabla u \cdot \nabla \beta(f^n) - (\Delta u) f^n \beta'(f^n) +
\beta'(f^n) \, r^n,
\eeqn
for any $\beta \in C^1(\R) \cap W^{2,\infty}_{loc}(\R)$ such that $\beta''$ is piecewise continuous and  
vanishes
outside of a compact set. Because the equation \eqref{eq:weps} with $u$
fixed is linear, the difference 
$f^{n,k} := f^{n} - f^{k}$ satisfies \eqref{eq:weps} with $r^n$ replaced by $r^{n,k} := r^{n} - r^{k} \to 0$
in $L^1(0,T;L^1_{loc}(\R^2))$, and then also \eqref{eq:betaweps} (with again $f^n$ and $r^n$ changed in 
$f^{n,k}$ and $r^{n,k}$). For any non-negative function  $\chi \in C^2_c(\R^d)$ and any time $t \in (0,T]$, we obtain
\bean
\int_{\R^2} \beta(f^{n,k}(t)) \, \chi  
&=& \int_{\R^2} \beta(f^{n,k}(0))\, \chi
- \int_0^t \!\!\! \int_{\R^2} \beta''(f^{n,k}(s)) |\nabla f^{n,k}(s)|^2 \, \chi \\
&+& \int_0^t \!\!\! \int_{\R^2} \beta(f^{n,k}(s))\, \{\Delta \chi + \nabla u(s)\cdot \nabla \chi\} \\
&+&\int_0^t \!\!\! \int_{\R^2} \{\beta(f^{n,k}(s)) - f^{n,k}(s) \beta'(f^{n,k}(s)) \}\, \Delta u(s)\,   \chi \\  
&+&  \int_{0}^{t} \!\!\!  \int_{\R^2}  \beta'(f^{n,k}(s)) \, r^{n,k}(s) \, \chi.
\eean
In that last 
equation, we  choose $\beta(\xi) = \beta_1(\xi) = \xi^2/2$ for $|\xi| \le 1$,  
$\beta_1(\xi) =  |\xi| - 1/2$ for $|\xi| \ge 1$. Using that $|\beta'_1|\leq 1$ and $\beta''_1 \ge 0$, it follows
\bean
\int_{\R^2} \beta_1(f^{n,k}(t)) \, \chi  
&\le& \int_{\R^2} \beta_1(f^{n,k}(0))\, \chi 
+ \int_0^t \!\!\! \int_{\R^2} \beta_1(f^{n,k}(s))\,\{ |\Delta \chi| + |\nabla u(s)| |\nabla \chi| \} \\
&+&
\int_0^t \!\!\! \int_{\R^2} |\beta_1(f^{n,k}(s)) - f^{n,k}(s) \beta'_1(f^{n,k}(s)) | \, |\Delta u(s)|\,   \chi
+  \int_{0}^{t} \!\!\!  \int_{\R^2}  |r^{n,k}(s)| \, \chi .
\eean

Since $f_0 \in L^1$, we have $f^{n,k}(0) \to 0$
in $L^1(\R^2)$ and we deduce from the previous inequality and the following convergences: $r^{n,k}  \to 0$ in  
$L^1(0,T; L^1_{loc}(\R^2))$; $ \beta_1(f^{n,k}) |\nabla u|  \to 0$ in $L^1(0,T;L^1_{loc}(\R^2))$, because $\beta_1(\xi)\leq |\xi|$, $f^{n,k}\to 0$ in $L^{2}(0,T,L^{2}(\R^2))$ 
by \eqref{eq:wLqtLpx} with $p=2$ and $\nabla u \in L^{2}(0,T;L^2(\R^2))$ 
 by Definition \ref{def:sol}; $ \beta_1(f^{n,k}) |\Delta u|  \to 0$ in $L^1(0,T;L^1 (\R^2))$, because $\beta_1(\xi)\leq |\xi|$, $f^{n,k}\to 0$ in $L^{2}(0,T,L^{2}(\R^2))$ and  $\Delta u \in L^{2}(0,T;L^2(\R^2))$ 
 by \eqref{eq:DeltauL2}; and $ f^{n,k} \beta'_1(f^{n,k}) |\Delta u|  \to 0$ in $L^1(0,T;L^1 (\R^2))$, because $|\beta'_1|\leq 1$, 
 $f^{n,k}\to 0$ in $L^{2}(0,T,L^{2}(\R^2))$ and $\Delta u \in L^{2}(0,T;L^2(\R^2))$, that
$$
\sup_{t \in [0,T]} \int_{\R^2} \beta_1(f^{n,k}(t,x)) \, \chi(x) \, dx  
\, \mathop{\longrightarrow}_{n,k \to \infty} \, 0.
$$
Since $\chi$ is arbitrary, we deduce that there exists $\bar f \in C([0,\infty); L^1_{loc}(\R^2))$ so that 
$f^n \to \bar f$ in $C([0,T]; L^1_{loc}(\R^2))$, $\forall \, T > 0$. Together with the
convergence $f^n \to f$ in  $C([0,\infty);\DD'(\R^2))$ and the bound \eqref{eq:bddFD}, we deduce  
that $f=\bar f$ and  
\beqn\label{eq:wcont}
f^n \to  f \quad \hbox{in} \quad C([0,T]; L^1(\R^2)), \quad \forall \, T > 0.
\eeqn

\smallskip\noindent
{\sl  Step 2. Linear estimates.} 
We come back to \eqref{eq:betaweps}, which implies, for all $0\le t_0<t_1$, all $\chi\in C^2_c(\R^2)$,
\beqn\label{eq:betawnbis}
\bal
\int_{\R^2} \beta(f^{n}_{t_1}) \, \chi   +    \int_{t_0}^{t_1} \! 
\int_{\R^2} 
\beta''(f^n_s) \, |\nabla_x f^n_s|^2 \, \chi
=  \int_{\R^2} \beta( f^{n}_{t_0}) \, \chi   
+ \int_{t_0}^{t_1} \!  \int_{\R^2} \beta(f^n_s) \, \{ \Delta \chi +\nabla u \cdot \nabla \chi \}
\\
+\int_{t_0}^{t_1} \! \int_{\R^2} \{\beta(f^{n}_s) - f^{n}_s \beta'(f^{n}_s) \}\, \Delta u_s\,   \chi
+ \int_{t_0}^{t_1} \!  \int_{\R^2} \beta'(f^n _s) \, r^n \, \chi .  
\eal
\eeqn
Choosing $0 \le \chi \in C^2_c(\R^2)$ and 
$\beta \in C^1(\R) \cap 
W^{2,\infty}_{loc}(\R)$ such that $\beta''$ is non-negative and vanishes
outside of a compact set,  and passing to the limit as $n\to\infty$,  we get 
\beqn\label{eq:betafchi}
\bal
\int_{\R^2} \beta(f_{t_1}) \, \chi   +    \int_{t_0}^{t_1} \! 
\int_{\R^2} \beta''(f_s) \, |\nabla_x f_s|^2 \, \chi
\le  \int_{\R^2} \beta( f_{t_0}) \, \chi 
+\int_{t_0}^{t_1} \! \int_{\R^2} \{\beta(f_s) - f_s \beta'(f_s) \}\, \Delta u_s\,   \chi \\ 
+ \int_{t_0}^{t_1} \!  \int_{\R^2} \beta(f_s) \, \{ \Delta \chi + \nabla u \cdot \nabla \chi \} .  
\eal
\eeqn
By approximating $\chi\equiv 1$ by the sequence $(\chi_R)$ with $\chi_R(x) = \chi(x/R)$, $0 \le \chi \in \DD(\R^2)$, 
we see that the last term in \eqref{eq:betafchi} vanishes and we get \eqref{eq:betaf} in the limit $R \to \infty$ for 
any renormalizing function $\beta$ with linear growth at infinity.

\smallskip\noindent
{\sl  Step 3. Super-linear estimates.} Finally, for any $\beta$ satisfying the growth condition as in the statement of the Lemma, 
we just approximate $\beta$ by an increasing sequence of smooth renormalizing functions $\beta_R$ with linear growth at infinity,
and we pass to the limit in \eqref{eq:betaf} in order to conclude.  \qed

\medskip

As a first application of the previous lemma we obtain the following estimate.

\begin{lem}\label{lem:BdH2}
For any weak solution $(f,u)$ there exists a constant $C := C (M,\HH_0,\FF_0,T,p)$ such that, for any $0 \le t_0 < t_1 \le T$, there holds
\beqn\label{eq:boundH2}
\int_{\R^2} f_{t_1} (\log f_{t_1})_+^2
+  \int_{t_0}^{t_1} \! \int_{\R^2} \frac{|\nabla f|^2}{f} \, \log f \, {\bf 1}_{f>e}
 \le \int_{\R^2} f_{t_0} (\log f_{t_0})_+^2 + C.
\eeqn

\end{lem}

\begin{proof}[Proof of Lemma~\ref{lem:BdH2}]
Let $K \ge e^2$ and define the renormalizing function $\beta_K : \R_+ \to \R_+$ by
$$
\beta_K(\xi) := 
\begin{cases}
\frac{\xi^2}{e}, \quad & \text{ if } \xi \le e; \\
\xi (\log \xi)^2, \quad & \text{ if } e \le \xi \le K; \\
(2+ \log K) \xi \log \xi - 2K \log K, \quad & \text{ if } \xi \ge K; \\
\end{cases}
$$
so that $\beta_K$ is convex and piecewise $C^1$. Moreover it holds
$$
\bal
|\beta_K(\xi) - \beta'_K (\xi) \xi |
&\le  \frac{\xi^2}{e} \, {\mathbf 1}_{\xi < e} + 2 \xi \log \xi \, {\mathbf 1}_{e < \xi < K}
+ 4 \, \xi \log K  \, {\mathbf 1}_{\xi > K} \\
&\le 4 \{ \xi \, {\mathbf 1}_{\xi < e} + \xi \log \xi \, {\mathbf 1}_{e < \xi < K } + \xi \log K \, {\mathbf 1}_{\xi >K} \} =: 4 \gamma_K(\xi),
\eal
$$
and
$$
\bal
\beta''_K(\xi) 
&\ge \frac{2}{e} \, {\mathbf 1}_{\xi < e} + 2\frac{\log \xi}{\xi} \, {\mathbf 1}_{e<\xi<K }
+ \frac{\log K}{\xi} \, {\mathbf 1}_{\xi >K}.
\eal
$$
We deduce from Lemma \ref{lem:Bdbeta} that
\beqn\label{eq:flogf2}
\bal
\int_{\R^2} \beta_K(f_{t_1})    
&+ \frac{2}{e} \int_{t_0}^{t_1} \! \int_{\R^2}  |\nabla f|^2  \, {\mathbf 1}_{f < e}
+ 2 \int_{t_0}^{t_1} \! \int_{\R^2}  \frac{|\nabla f|^2}{f} \log f  \, {\mathbf 1}_{e < f < K}\\&
+ \log K \int_{t_0}^{t_1} \! \int_{\R^2}  \frac{|\nabla f|^2}{f}  \, {\mathbf 1}_{f >K} 
\le \int_{\R^2} \beta_K(f_{t_0})    + 4 \int_{t_0}^{t_1} \! \int_{\R^2}   \gamma_K(f)  |\Delta u| .
\eal
\eeqn
On the one hand, for any $\delta>0$, we have
$$
\int_{t_0}^{t_1} \! \int_{\R^2}   \gamma_K(f)  |\Delta u| 
\le \delta \int_{t_0}^{t_1} \! \int_{\R^2} \gamma_K(f)^2 + C(\delta) \| \Delta u \|_{L^2(0,T;L^2(\R^2))}^2.
$$
On the other hand, defining $\widetilde \log_K \xi :=  {\mathbf 1}_{\xi \le e} + \log \xi \, {\mathbf 1}_{e < \xi \le K} + \log K \, {\mathbf 1}_{\xi >K}$, we have
$$
\int_{\R^2} \gamma_K(f)^2 \le \int_{\R^2} f^2 + \int_{\R^2} (f \widetilde \log_K f)^2,
$$
and thanks to inequality \eqref{eq:LpbdFisher} with $p=2$
\bear\nonumber
\int_{\R^2}  (f \widetilde \log_K  f )^2 
&\le& C_2^2 \left(  \int_{\R^2}  f \widetilde \log_K  f   \right)   \left( \int_{\R^2}{ |\nabla  (f  \widetilde \log_K  f ) |^2 \over f  \widetilde \log_K  f}  \right)
\\ \label{eq:flogKf2}
&\le& C' (M + \HH^+(f)) \left(  \int_{\R^2}  \frac{|\nabla f|^2}{f} \widetilde \log_K f \, {\mathbf 1}_{f \ge e} + I(f)  \right).
\eear
Coming back to \eqref{eq:flogf2}, Estimate \eqref{eq:boundH2} follows by taking $\delta>0$ small enough, using Lemma~\ref{lem:Bdf&K} 
and Proposition~\ref{prop:Fisher}, and then letting $K \to \infty$.
\end{proof}

\begin{lem}\label{lem:BdLp} For any weak solution $(f,u)$, any $p \ge 2$ and any $t_0 \in [0,T) $ such that $f(t_0) \in L^p(\R^2)$, there exists a constant $C := C (M,\HH_0,\FF_0,T,p, \| f(t_0)\|_{L^p})$ such that, for all $t_0 < t_1 \le T$, there holds
\beqn\label{eq:boundLp}
\| f (t_1) \|^p_{L^p}  + {1 \over 2} \int_{t_0}^{t_1} \| \nabla_x f ^{p/2} \|_{L^2}^2 \, dt \le C  .
\eeqn
\end{lem}

\noindent
{\sl Proof of Lemma~\ref{lem:BdLp}.} We split the proof into four steps.

\medskip
\noindent{\it Step 1.} 
We define the renormalizing function $\beta_K : \R_+ \to \R_+$, $K \ge e^2$,  by
$$
\beta_K(\xi) := \frac{\xi^p}{p} \,\,\, \hbox{if} \,\,\,  \xi \le K, \quad 
\beta_K(\xi)  := \frac{K^{p-1}}{\log K} \, (\xi  \log \xi - \xi) - \frac{K^p}{p'} + \frac{K^p}{\log K}  \,\,\, \hbox{if} \,\,\,  \xi \ge K,
$$
so that $\beta_K$ is 
convex, increasing and piecewise of class $C^1$. Moreover the following estimates hold
$$
|\beta_K(\xi) - \beta_K'(\xi) \xi | \le \frac{1}{p'} \, \xi^{p} \, {\bf 1}_{\xi < K} 
+  3 K^{p-1}  \, \xi  \, {\bf 1}_{\xi > K} , 
$$
$$ 
\beta''_K(\xi)  = (p-1) \, \xi^{p-2} \, {\bf 1}_{\xi < K} + {K^{p-1} \over \log K}  \, {1 \over \xi} \, {\bf 1}_{\xi > K}  \quad \text{ and } \quad
\beta_K(\xi) \ge 2^{1-p}p^{-1} K^{p-1} \, \xi \, {\bf 1}_{\xi > K/2}.
$$
Thanks to Lemma~\ref{lem:Bdbeta}, we may write
\beqn\label{eq:betaKf} 
\bal
\int_{\R^2} \beta_K(f_{t_1})    
+  {4 \over p p'}   \int_{t_0}^{t_1} \!  \int_{\R^2} |\nabla_x (f^{p/2})  |^2  \, {\bf 1}_{f \le K} 
+   {K^{p-1} \over \log K}   \int_{t_0}^{t_1} \!  \int_{\R^2} {|\nabla_x f  |^2 \over f}  \, {\bf 1}_{f \ge K} 
\\
\le  \int_{\R^2} \beta_K(f_{t_0})     
 +  \frac{1}{p'} \int_{t_0}^{t_1} \!  \int_{\R^2} |\Delta u| \, f^{p} \,  {\bf 1}_{f < K} 
  + 3 K^{p-1}   \int_{t_0}^{t_1} \!  \int_{\R^2}  |\Delta u| \,  f  \, {\bf 1}_{f > K}  .
\eal
\eeqn

\medskip
\noindent{\it Step 2.}
For the second term on the right-hand side of \eqref{eq:betaKf}, using the Gagliardo-Niremberg-Sobolev inequality
$$
\int_{\R^2} g^4 \, dx \le C \, \int_{\R^2} g^2 \, dx \, \int_{\R^2} |\nabla g|^2 \, dx,
$$
we have 
\bean
\TT_1 
&:= & \frac{1}{p'}  \int_{\R^2} |\Delta u| \, f^p  \,  {\bf 1}_{f < K} 
\\
&\le&  \frac{1}{p'}  \|  \Delta u \|_{L^2_x} \Bigl(  \int_{\R^2} \, (f \wedge K)^{2p}   \Bigr)^{1/2}
\\
&\le&   C \,  \|  \Delta u \|_{L^2_x} \Bigl(  \int_{\R^2} \, (f \wedge K)^{p}   \int_{\R^2} |\nabla (f \wedge K)^{p/2}|^2   \Bigr)^{1/2}
\\
&\le&   C \,  \|  \Delta u \|_{L^2_x}^2  \int_{\R^2} \, \beta_K(f)    +{1 \over p p'}  \int_{\R^2}  |\nabla (f^{p/2})  |^2  \, {\bf 1}_{f < K}. 
\eean

\medskip
\noindent{\it Step 3.} Let us define the convex function $\Phi$ by $\Phi(\xi) := \xi^2 (\widetilde{\log}\, \xi)^2$ for any $\xi \ge 0$, with $\widetilde{\log}\, \xi := {\bf 1}_{\xi \le e} + \log \xi \, {\bf 1}_{\xi >e}$. We observe that  $\Phi$ is a $N$-function so that we may associate to $\Phi$ the Orlicz spaces $L_\Phi $ (see Appendix \ref{app:orlicz}). We have already obtained that $\| \Delta u \|_{L^2_x}^2 \in L^1(0,T)$ in Lemma \ref{lem:Bdf&K} 
and we claim that we also have 
\beqn\label{eq:DeltauLPhi}
\| \Phi( |\Delta u| ) \|_{L^1_x} \in L^1(t_0,T).
\eeqn
Indeed, passing to the limit $K\to\infty$ in Estimate~\eqref{eq:flogKf2}, and using 
Proposition~\ref{prop:Fisher}  and Lemma~\ref{lem:BdH2}, 
we immediately deduce 
\beqn\label{eq:fLPhi}
f \in L_{\Phi} ((t_0,T) \times \R^2).
\eeqn
We now consider the linear operator $f \mapsto U(f) := \Delta u$ where $u$ is the solution to the linear parabolic equation $\eps \partial_t u - \Delta u + \alpha u = f$. Thanks to standard results (see e.g. Theorem X.12 stated in \cite{BrezisBook} and the quoted references therein)  
$$
U : L^p ((t_0,T) \times \R^2) \to L^p((t_0,T) \times \R^2)
$$
 is a bounded operator for any $p \in(1,\infty)$. Since $L_{\Phi} ((t_0,T) \times \R^2)$ is an interpolation space between $L^{3/2} ((t_0,T) \times \R^2)$ and $L^{9/2} ((t_0,T) \times \R^2)$ (see Appendix \ref{app:inter}), we also get 
 \beqn\label{eq:ULPhiLPhi}
U : L_{\Phi} ((t_0,T) \times \R^2) \to L_{\Phi} ((t_0,T) \times \R^2)
\eeqn
 bounded. We then deduce \eqref{eq:DeltauLPhi} from \eqref{eq:fLPhi}, \eqref{eq:ULPhiLPhi} and the fact that $\Phi$ satisfies the $\Delta_2$-condition (see Appendix~\ref{app:orlicz}).
 
%

\medskip
\noindent{\it Step 4.}
Now we estimate the last term in \eqref{eq:betaKf}. We denote by  $\Phi^*$ the conjugate function of $\Phi$ and we observe that  $\Phi^*(\eta) \le C \, \eta^2 (\widetilde\log \eta)^{-2} $ for any $\eta > 0$ and for some fixed constant $C \in (0,\infty)$ (see Appendix \ref{app:example}). 
We introduce the notation $f_K := f {\bf 1}_{f >K}$ and $A_K(f) := (\int_{\R^2} f {\bf 1}_{f >K/2} )^{1/2}$. 
Using Young's inequality $\xi \eta  \le \Phi(\xi) + \Phi^*(\eta)$,  we obtain
\bean
\TT_2 
&:=&  2 K^{p-1}   \int_{\R^2}  |\Delta u | \,  f \, {\bf 1}_{f \ge K} \\
&\le& 2  K^{p-1} \left\{ \int_{\R^2} \Phi( A_K(f) |\Delta u|) + \int_{\R^2} \Phi^*( A_K(f)^{-1} f_K)    \, {\bf 1}_{f \ge K}   \right\} =: \TT_{2,1} + \TT_{2,2}.
\eean
For the term $\TT_{2,1}$, using that $A_K(f) \le M^{1/2}$ and the elementary inequality $\widetilde\log (\xi\eta) \le \widetilde\log \, \xi + \widetilde\log \, \eta$,  we may write
$$
\bal
\TT_{2,1} 
&:= 2 K^{p-1} \int_{\R^2} \Phi( A_K(f) |\Delta u| ) \\
&\le C\, K^{p-1} \int_{\R^2} |\Delta u|^2 \, |A_K(f)|^2 \, (\widetilde\log (A_K(f) |\Delta u|))^2 \\
&\le C \, K^{p-1} \left(\int_{\R^2} f \, {\bf 1}_{f >K/2} \right)  \left( \int_{\R^2} |\Delta u|^2 (\widetilde\log A_K(f))^2 + \int_{\R^2} |\Delta u|^2 (\widetilde\log |\Delta u|)^2 \right) \\
&\le C  \left(\int_{\R^2} \beta_K(f) \right)  \left( C(M) \| \Delta u \|_{L^2_x}^2 + \| \Phi( |\Delta u| ) \|_{L^1_x} \right). 
\eal
$$
For the term $\TT_{2,2}$, using that $f_K > K$ on $\{ f > K \}$, we obtain, for $K > eM^{1/2} $,
$$
\bal
\TT_{2,2} 
&:= 2  K^{p-1} \int_{\R^2}  \Phi^* (  f_K /  A_K(f)) \, {\bf 1}_{f \ge K}   \\
&\le C \,  K^{p-1} \int_{\R^2}   \frac{(f_K / A_K(f))^2}{( \widetilde\log (f_K / A_K(f)))^2}  \, {\bf 1}_{f \ge K}  \\
&\le C \, \frac{K^{p-1} \, |A_K(f)|^{-2} }{(\log \frac{K}{M^{1/2}})^2} \, \int_{\R^2} (f - K/2)^2_+ \\
&\le C \, \frac{K^{p-1} \, |A_K(f)|^{-2}}{(\log \frac{K}{M^{1/2}})^2} \left( \int_{\R^2} |\nabla f| \, {\bf 1}_{f \ge K/2} \right)^2\\
&\le C \, \frac{K^{p-1} \,   |A_K(f)|^{-2}} {(\log \frac{K}{M^{1/2}})^2}   \left( \int_{\R^2}  f \, {\bf 1}_{f \ge K/2} \right)  \left( \int_{\R^2}  \frac{ |\nabla f|^2 }{f} \, {\bf 1}_{f \ge K/2} \right)\\
&\le C \, \frac{K^{p-1} }{(\log \frac{K}{M^{1/2}})^2}  \left( \int_{\R^2}  \frac{ |\nabla f|^2 }{f} \, {\bf 1}_{K/2 \le f \le K/2}
+ \int_{\R^2}  \frac{ |\nabla f| }{f} \, {\bf 1}_{ f > K} \right)\\
&\le C \, \frac{K^{p-1}}{(\log \frac{K}{M^{1/2}})^2}  \,  \left( K^{1-p}\int_{\R^2}   |\nabla (f^{p/2})|^2  \, {\bf 1}_{f \le K}
+ \int_{\R^2}  \frac{ |\nabla f| }{f} \, {\bf 1}_{ f > K} \right),\\
\eal
$$
where we have used the elementary inequality $f_K^2 \le C (f - K/2)_+^2$ in the second line, Sobolev's inequality in the third line and Cauchy-Schwarz inequality in the fourth line. 
Hence, for $K$ large enough, we have
$$
\bal
\TT_{2,2} 
&\le C \, \frac{1}{(\log \frac{K}{M^{1/2}})^2} \,   \int_{\R^2}   |\nabla (f^{p/2})|^2  \, {\bf 1}_{f \le K}
+ C \, \frac{K^{p-1}}{(\log \frac{K}{M^{1/2}})^2} \, \int_{\R^2}  \frac{ |\nabla f| }{f} \, {\bf 1}_{ f > K} \\
&\le \frac{1}{p p'} \,  \int_{\R^2}   |\nabla (f^{p/2})|^2  \, {\bf 1}_{f \le K}
+ \frac{1}{2} \, \frac{K^{p-1}}{\log K} \,\int_{\R^2}  \frac{ |\nabla f|^2 }{f} \, {\bf 1}_{ f > K} .
\eal
$$
Gathering $\TT_1$, $\TT_{2,1}$ and $\TT_{2,2}$, it follows that
\beqn\label{eq:boundLp0}
\bal
&\int_{\R^2} \beta_K(f_{t_1})    
+  {2 \over p p'}   \int_{t_0}^{t_1} \!  \int_{\R^2} |\nabla_x (f^{p/2})  |^2  \, {\bf 1}_{f \le K} \\
&\qquad
\le  \int_{\R^2} \beta_K(f_{t_0})    +  C \,  \int_{t_0}^{t_1} \left( \| \Delta u  \|_{L^2_x}^2 +   \| \Phi( |\Delta u|) \|_{L^1_x} \right)\int _{\R^2} \beta_K(f). 
\eal
\eeqn
Using the fact that $h(t) := \int_{\R^2} |\Delta u|^2 + \Phi (|\Delta u|) \in L^1(t_0,T)$ from Lemma \ref{lem:Bdf&K} and Step 3, we can conclude to \eqref{eq:boundLp} by applying first the Gronwall's lemma and by passing then to the limit $K \to \infty$.
%
%
\qed

\medskip

\begin{lem}\label{lem:fsmooth} Any weak solution $(f,u)$ satisfies 
$$
\partial_t f, \partial_x f, \partial^2_{x_ix_j} f, \partial_t u, \partial_x u, \partial^2_{x_ix_j} u   \in C_b((0,T] \times \R^2), \,\, \forall \, T >0, 
$$
so that it is  a ``classical solution" for positive time.
\end{lem}

%
%

\noindent
{\sl Proof of Lemma~\ref{lem:fsmooth}.} For any time $t_0 \in (0,T)$ and any exponent $p \in (1,\infty)$,
there exists $t'_0 \in (0,t_0)$ such that $f(t'_0) \in L^p(\R^2)$ thanks to \eqref{eq:wLqtLpx},
from what we deduce using \eqref{eq:boundLp} on the time interval $(t'_0,T)$ that  
\beqn\label{eq:uniqvortexregL2nabla}
 f \in L^\infty(t_0,T; L^p(\R^2))
\quad \hbox{and}\quad 
\nabla_x f \in L^2((t_0,T) \times \R^2).
\eeqn 
Since $u$ satisfies the parabolic equation 
$$
 \eps \partial_t u - \Delta u + \alpha u = f, 
 $$
 the maximal regularity of the heat equation in $L^p$-spaces (see Theorem X.12 stated in 
\cite{BrezisBook} and the quoted references) and the fact that 
$$
u(t) = \gamma_{t/\eps}^\alpha *_{t,x} f + \gamma_{t/\eps}^\alpha *_x  u_0
\quad\text{and}\quad 
\nabla u = \Gamma_{t/\eps}^\alpha *_{t,x}  f + \gamma_{t,\eps}^\alpha *_x \nabla u_0,
$$ 
where we denote $\gamma^{\alpha}_s = e^{-\alpha s}\,  \gamma_s$ and similarly for $\Gamma$, $\gamma_t$ is the heat kernel given by
$$
  \gamma_t (x) := {1 \over 4 \pi t} \, \exp\left( - \frac{|x|^2}{4t} \right) \in L^{z_1}(0,T;L^{z_2}(\R^2)),
\quad \forall \, z_1, z_2 \ge 1, \,\,  1/z_1+1/z_2 > 1, 
$$
and 
$$
\Gamma_t(x) := \nabla_x \gamma_t(x) 
\in L^{s_1}(0,T;L^{s_2}(\R^2)),
\quad \forall \, s_1, s_2 \ge 1, \,\,  1/s_1+1/s_2 > 3/2, 
$$
provide the bound 
\beqn\label{eq:uLpMaximal}
 u  \in L^\infty(t_0,T; L^p( \R^2)),
  \quad  \nabla u  \in L^\infty(t_0,T; L^p( \R^2)),
 \quad \partial_t u, D^2 u  \in L^p((t_0,T) \times \R^2), 
 \eeqn
for all $t_0 \in (0,T)$ and $p \in (1,\infty)$. 
Since now $f$ satisfies the parabolic equation 
$$
\partial_t f - \Delta f = -\nabla u \cdot \nabla f - (\Delta u) f =: Z
$$
with $Z \in L^2(t_0,T; L^q(\R^2))$ for all $t_0 \in (0,T)$ and all $q \in [1,2)$ from \eqref{eq:uniqvortexregL2nabla} and 
\eqref{eq:uLpMaximal}, the same maximal regularity of the heat equation in $L^q$-spaces (with the choice $s_1 = s_2 = (4/3)^-$) implies 
$$
\nabla f  \in L^p(t_0,T; L^p( \R^2)), \; \forall \, p \in [2,4),
$$
 and then $Z \in  L^p(t_0,T; L^p( \R^2))$,  $\forall \, p \in [2,4)$. By a bootstrap argument of  the regularity property of the heat equation, we easily get 
\beqn\label{eq:fLpMaximal}
f  \in L^\infty(t_0,T; L^p( \R^2)),
\quad  \nabla f  \in L^\infty(t_0,T; L^p( \R^2)),
\quad \partial_t f, D^2 f  \in L^p((t_0,T) \times \R^2), 
\eeqn
for all $t_0 \in (0,T)$ and $p \in (1,\infty)$. The Morrey inequality implies then $f,\nabla f, u, \nabla u \in C^{0,\alpha}((t_0,T) \times \R^2)$
for any $0< \alpha < 1$, and any $t_0 >0$.  Finally the classical Holderian regularity result for the heat equation (see Theorem X.13 stated in \cite{BrezisBook} and the quoted references) implies first $u \in C^{2,\alpha}((t_0,T) \times \R^2)$ and next $f \in C^{2,\alpha}((t_0,T) \times \R^2)$, which concludes the proof.
\qed

\medskip

We prove now the free energy-dissipation of the free energy identity \eqref{eq:FreeEnergy} in Theorem~\ref{theo:uniq}.

\begin{proof}[Proof of the free energy identity in Theorem~\ref{theo:uniq}] 
We split the proof into four steps.

\medskip
\noindent
{\it Step 1.} 
We claim that the free energy functional $\FF$
is lower semi-continuous (lsc) in the sense that for any sequences $(f_n)$ and $(u_n)$ of nonnegative functions such that $(f_n)$ is bounded in $L^1 \cap L^1(\log \la x \ra^2) $ with same mass $M < 8\pi$, $(u_n)$ is bounded in $H^1$ if $\alpha>0$ or in $L^1 \cap \dot H^1 $ if $\alpha=0$, $(f_n \, u_n)$ is bounded in $L^1$,  $(\FF(f_n,u_n))$ is  bounded
and $(f_n , u_n) \wto (f,u)$ in $\DD'(\R^2) \times \DD'(\R^2)$,  there holds 
\beqn\label{eq:Fsci}
0 \le f \in L^1 \cap L^1(\log \la x \ra^2)  
\quad\hbox{and}\quad
\FF(f,u) \le \liminf_{n\to\infty} \FF(f_{n}, u_n).
\eeqn 
Indeed, because of \eqref{eq:H+<H} and \eqref{eq:F<H}, we have  $\HH^+(f_n) \le C$ and we may apply the Dunford-Pettis lemma
which implies that $f_n \wto f$ weakly in $L^1(\R^2)$. 
We then rewrite the free energy functional as
$$
\FF(f_n,u_n) = \HH(f_n) + F_\alpha(f_n,u_n),
$$
with 
$$
\HH(f_n) := \int f_n \log f_n
\quad\text{and}\quad 
F_\alpha (f_n,u_n) := {1 \over 2} \int |\nabla u_n|^2   + {\alpha \over 2}   \int u_n^2  - \int f_n \, u_n ,
$$
and we consider separately the case $\alpha >0$ and $\alpha=0$.

\medskip
\noindent
{\it Step 2: Case $\alpha>0$.} We denote $\bar u_n = \kappa_\alpha * f_n$ where $\kappa_\alpha(z) = \frac{1}{4\pi} \int_{0}^{\infty}\frac1t e^{-\frac{|z|^2}{4t} - \alpha t} \, dt$ is the Bessel kernel. Since $f_n \ge0 $, $f_n \in L^1 \cap L \log L$ and $u_n \in H^1$, \cite[Lemma 2.2]{MR2433703} implies that $\bar u_n \in H^1$ and also that the functional $F_\alpha(f_n,u_n)$ is finite and satisfies
$$
F_\alpha(f_n,u_n) - F_\alpha(f_n, \bar u_n) = \frac12 \| \nabla(u_n-\bar u_n) \|_{L^2}^2 + \frac{\alpha}{2} \| u_n - \bar u_n \|_{L^2}^2.
$$
Hence, we can write
$$
\bal
\FF(f_n,u_n) 
&= \HH(f_n) -\frac12 \iint f_n(x) f_n(y) \kappa_\alpha (x-y)
+ \frac12 \| \nabla(u_n-\bar u_n) \|_{L^2}^2 + \frac{\alpha}{2} \| u_n - \bar u_n \|_{L^2}^2\\
&=: \HH(f_n) + \VV(f_n) + \UU_1(u_n - \bar u_n) + \UU_2(u_n -\bar u_n),
\eal
$$
where the functionals $\UU_1$ and  $\UU_2$ are defined through the third and fourth term respectively. 
We clearly have that $\UU_1 + \UU_2$ is lsc for the weak $H^1$ convergence and $\HH$ is lsc for the weak $L^1$ convergence, so we investigate the functional $\VV$. For any $\epsilon \in (0,1)$ we split $ \VV = \VV_{\epsilon} + \RR_{\epsilon}$ as
$$
\bal
\VV_{\epsilon} (g) &:= -\frac12 \iint g(x) g(y) \, \kappa_\alpha (x-y) \, {\mathbf 1}_{|x-y|>\epsilon} \\
\RR_{\epsilon} (g) &:= -\frac12 \iint g(x) g(y) \, \kappa_\alpha (x-y) \, {\mathbf 1}_{|x-y| \le\epsilon} .
\eal
$$ 
The Bessel kernel $\kappa_\alpha$ is a positive radial decreasing function with a singularity at the origin: $\kappa_\alpha (z) = -\frac{1}{2\pi} \log |z| + O(1)$ when $|z| \to 0$. Hence $\VV_{\epsilon}$ is continuous for the weak $L^1$ convergence and for the rest term we obtain, for any $\epsilon \in (0,1)$ and $\lambda >1$,
$$
\bal
|\RR_{\eps}(g)| 
&\le C \iint g(x)  g(y)  {\mathbf 1}_{|x-y| \le\epsilon}
 +C \iint g(x)  g(y)  (\log |x-y|)_{-}  {\mathbf 1}_{|x-y| \le\epsilon} \\
&\le  C \iint g(x) {\mathbf 1}_{g(x) \le \lambda} g(y) {\mathbf 1}_{|x-y| \le\epsilon}
+ C \iint g(x) {\mathbf 1}_{g(x) > \lambda} g(y) {\mathbf 1}_{|x-y| \le\epsilon} \\
&\quad
+C \iint g(x) {\mathbf 1}_{g(x) \ge \lambda} g(y) (\log |x-y|)_{-} {\mathbf 1}_{|x-y| \le\epsilon} 
+C \iint g(x) {\mathbf 1}_{g(x) > \lambda} g(y) \log (|x-y|^{-1}) {\mathbf 1}_{|x-y| \le\epsilon} \\
&\le  
 C \lambda \int_y g(y) \left\{ \int_{|z|\le \epsilon} dz \right\}
+C \int_y g(y) \left\{ \frac{1}{\log \lambda}\int_x g(x) \log g(x) \right\} \\
&\quad
+C \lambda \int_y g(y) \left\{ \int_{|z|\le \epsilon} (\log |z|)_- \, dz \right\}
+C \int_x g(x) {\mathbf 1}_{g(x) > \lambda} \int_y  \{ g(y) \log g(y) + |x-y|^{-1}  \} {\mathbf 1}_{|x-y| \le\epsilon} \\
&\le C \, M \lambda \epsilon^{2} + C \, M \frac{\HH(g)}{\log \lambda} +  C \, M \lambda \epsilon^{3/2}
+  C \, \frac{\HH(g)}{\log \lambda} \{   \HH(g) + \epsilon \},
\eal
$$ 
where we have used the convexity inequality $u v \le u \log u + e^v$ for all $u > 0$, $v \in \R$ and the elementary inequality $u\log u \ge - u^{1/2}$ for all $u \in (0,1)$. Hence $\sup_n |\RR_\eps (f_n)| \to 0$ as $\epsilon \to 0$ and we deduce that $\FF$ is lsc.

\medskip
\noindent
{\it Step 3: Case $\alpha = 0$.} We define $ \bar u_n = \kappa_0 (f_n - M H)$ where $H(x) = \la x \ra^{-4}/\pi$ and $\kappa_0(z) = -\frac1{2\pi} \log |z|$ is the Laplace kernel. Since $0 \le f_n \in L^1 \cap L^1(\log \la x \ra^2)$, $\HH(f_n)$ is finite, $\int (f-MH) = 0$ and $u_n \in \dot H^1$, \cite[Lemma 2.2]{MR2433703} implies that $\bar u_n \in \dot H^1$ and also that the functional $F_0(f_n-M H,u_n)$ is finite and verifies
$$
F_0(f_n - M H,u_n) - F_0(f_n - M H, \bar u_n) = \frac12 \| \nabla(u_n-\bar u_n) \|_{L^2}^2.
$$
Now we argue as in the case $\alpha >0$. First we write 
$$
\bal
\FF(f_n,u_n) 
&= \HH(f_n) -\frac12 \iint f_n(x) f_n(y) \kappa_0 (x-y)
+M \iint f_n(x) H(y) \kappa_0 (x-y)\\
&\quad
+ \frac12 \| \nabla(u_n-\bar u_n) \|_{L^2}^2 
-M \int H u_n 
- \frac{M^2}{2} \iint H(x) H(y) \kappa_0(x-y)\\
&=: \HH(f_n) + \VV(f_n) + \WW(f_n) +\UU_1(u_n - \bar u_n) + \UU_0(u_n)+ \ZZ(H) .
\eal
$$
The functional $\UU_1$ is lsc for the weak $\dot H^1$ convergence and $\HH$ is lsc for the weak $L^1$ convergence. For $\VV$ we just argue as in the preceding case $\alpha > 0$. 
In the same (even simpler) way we obtain that $\WW$ is lsc for the weak $L^1$ convergence.
Finally we conclude that $\FF$ is lsc.

\medskip
\noindent
{\it Step 4.} Now, we easily deduce that the free energy identity \eqref{eq:FreeEnergy} holds. Indeed, since $(f,u)$ is smooth for positive time, for any fixed $t > 0$ and any given sequence $(t_n)$ of positive real numbers which decreases to $0$, we clearly 
have 
$$
\FF(f(t_n), u(t_n)) = \FF(t) + \int_{t_n}^t \DD_\FF(f(s), u(s)) \, ds. 
$$  
Then, thanks to the Lebesgue convergence theorem, the lsc property of $\FF$ and the fact that $f(t_n) \wto f_0$ and $u(t_n) \wto u_0$ weakly in $\DD'(\R^2)$, we deduce 
from the above free energy identity  for positive time that 
$$
\bal
\FF(f_0,u_0) 
&\le \liminf_{n\to\infty} \FF(f(t_n), u(t_n)) \\
&\le \lim_{n\to\infty} \bigl\{\FF(t) + \int_{t_n}^t \DD_\FF(f(s),u(s)) \, ds \bigr\}= \FF(t) + \int_{0}^t \DD_\FF(f(s),u(s)) \, ds.
\eal
$$ 
Together with the reverse inequality \eqref{eq:bddFD} we conclude to \eqref{eq:FreeEnergy}. 
\end{proof}

\section{Uniqueness - Proof of Theorem~\ref{theo:uniq}} 
\label{sec:uniqueness}

In this section we prove the uniqueness part of Theorem~\ref{theo:uniq}. In order to do so,
 we first prove some estimates in Lemmas \ref{lem:LptC} and \ref{lem:L43to0}.

\begin{lem}\label{lem:LptC} Any weak solution $(f,u)$ to the Keller-Segel equation satisfies that for any $p \in (1,\infty)$, $T \in (0,\infty)$,
 there exists a constant $K = K(f_0, p, T)$ such that  
\beqn\label{eq:LptC}
t^{p-1} \| f(t) \|_{L^p}^p \leq K \qquad \forall\, t\in(0,T).
\eeqn
\end{lem}

\noindent
{\sl Proof of Lemma~\ref{lem:LptC}.} Recall that we already know that $\| f \|_{L^p} \in C^1(0,T)$ for any $p>1$ and $\|f \|_{L^p} \in L^\infty(t_0,T)$ for any $0 < t_0 < T$ and any $p \in [1,\infty]$. For $p>1$, we have
\bean
\frac{d}{dt}\, \int f^p 
&=& - 4(1-1/p)\int |\nabla (f^{p/2}) |^2 +(p-1) \int f^{p+1} - (p-1)\int  (\partial_t u + \alpha u) f^p \\
&=:& T_1 + T_2 + T_3.
\eean
Using the splitting $f = \min(f,A) + (f-A)_+$, for some $A>0$, and denoting $h(u) := \partial_t u + \alpha u \in L^2(0,T; L^2(\R^2))$, we have
$$
\bal
|T_3| &\leq C \int |h(u)| \min(f,A)^p + C  \int |h(u)| (f-A)_+^p =: T_{31}+ T_{32}.
\eal
$$
For the term $T_{31}$, we have
$$
|T_{31}| \leq C A^{p-1/2} \int |h(u)| f^{1/2} \leq C A^{p-1/2} ( \| h(u)\|_{L^2}^2 + M ).
$$
For $T_{32}$, using Gagliardo-Niremberg-Sobolev inequality $\| g \|_{L^4(\R^2)}^4 \leq C \| g\|_{L^2(\R^2)}^2 \| \nabla g \|_{L^2(\R^2)}^2$ with $g = (f-A)_+^{p/2}$, we have
$$
\bal
|T_{32}| 
&\leq C \int |h(u)| (f-A)_+^p \\
&\leq C \| h(c) \|_{L^2} \left(\int (f-A)_+^{2p}\right)^{1/2} \\
&\leq C \| h(c) \|_{L^2} \left(\int (f-A)_+^{p}\right)^{1/2} \left(\int | \nabla(f-A)_+^{p/2}|^2 \right)^{1/2} \\
&\leq C_\delta \| h(u)\|_{L^2}^2 \int (f-A)_+^{p} + \delta \int | \nabla (f^{p/2})|^2 \, \indiq_{f\geq A},
\eal
$$
for any $\delta >0$.

For the second term $T_2$, we have
$$
|T_2| \leq C \int \min(f,A)^{p+1} + C \int (f-A)_+^{p+1}.
$$
At the right-hand side, the first term is easily bounded by $C A^p M$, and using \eqref{eq:GNp+1}, the second one is bounded as follows
$$
\int (f-A)_+^{p+1} \leq C \left(\int (f-A)_+\right) \left( \int|\nabla (f-A)_+^{p/2}|^2 \right) \leq C \,\frac{\HH^+(f)}{\log A} \, \left( \int |\nabla (f^{p/2})|^2 \, \indiq_{f\geq A} \right). 
$$
Gathering all the previous estimates, choosing $\delta >0$ small enough and $A$ large enough, we get
\beqn\label{eq:dtf^p}
\frac{d}{dt}\, \int f^p \leq - C_0 \int |\nabla (f^{p/2})|^2 + C_1 \| h(u) \|_{L^2}^2 \int f^p + C_1(M + \|h(u)\|_{L^2}^2).
\eeqn
Thanks to the following H\"older and Gagliardo-Niremberg-Sobolev inequalities
$$
\left( \| f\|_{L^p}^p \right)^{p/(p-1)}\leq C  \| f\|_{L^1}^{1/(p-1)} \| f \|_{L^{p+1}}^{p+1}
\quad\text{and}\quad
\| f \|_{L^{p+1}}^{p+1} \leq C \|f\|_{L^1} \| \nabla (f^{p/2})\|_{L^2}^2,
$$
we obtain from \eqref{eq:dtf^p} the differential inequality
$$
\frac{d}{dt}\, X(t) \leq - C_0 X(t)^{\frac{p}{p-1}} + C_1 H(t) X(t) + C_2(1 + H(t)),
\qquad
t\in (0,T),
$$
where we denote $X(t) := \| f(t) \|_{L^p}^p$ and $H(t) := \| h(u)\|_{L^2}^2 (t) \in L^1(0,T)$.
By standard arguments (see e.g.\ \cite[Proof of Theorem 5.1]{MR2433703}) we conclude to \eqref{eq:LptC}. 
\qed

\medskip

We (crucially) improve the preceding estimate by showing


\begin{lem}\label{lem:L43to0}  For any $q \in (1,\infty)$, any weak solution $(f,u)$ to the Keller-Segel equation satisfies:  
\beqn\label{eq:L43to0}
t^{1-\frac1q} \|f (t,.) \|_{L^{q}} \to 0 \quad\hbox{as}\ t \to 0.
\eeqn
\end{lem}

\noindent
{\sl Proof of Lemma~\ref{lem:L43to0}.} We prove \eqref{eq:L43to0} from \eqref{eq:LptC} and an interpolation argument.  
On the one hand, denoting $\widetilde \log_+ f := 2 + \log_+ f$, we use the Cauchy-Schwarz inequality in order to get 
\bean
\int f^{q} 
&\le& \Bigl( \int f  \, \widetilde \log_+ f \Bigr)^{1/2} \, \Bigl( \int f^{2q-1} \, (\widetilde \log_+ f)^{-1} \Bigr)^{1/2},
\eean
or in other words
\beqn\label{eq:L43bdd}
\|f \|_{L^{q}} \le C (M,\HH_+(f)) \, \Bigl( \int f^{2q-1} \, (\widetilde \log_+ f)^{-1} \Bigr)^{1/(2q)}.
\eeqn
On the other hand, we observe that 
\bear\nonumber
t^{2q-2} \int f^{2q-1} \, (\widetilde \log_+ f)^{-1} 
&\le&
t^{2q-2} \,  {R^{2q-1}  \over \widetilde \log_+ R}
\int_{f \le R} f + {t^{2q-2}  \over \widetilde \log_+ R } \int_{f \ge R} f^{2q-1} \quad \forall \, R >0
\\ \nonumber
&\le&
t^{2q-2}\,  {M  R^{2q-1}  \over \widetilde \log_+ R}
 + { K \over \widetilde \log_+ R } ,
\eear
for any $R > 0$, where we have used that $s \mapsto s^{2q-1} (\widetilde \log_+ s)^{-1}$ is an increasing function in the first line, the mass conservation of the solution of the Keller-Segel equation and the estimate \eqref{eq:LptC} in the second  line. Choosing  $R := t^{-1}$, we deduce
\beqn \label{eq:L2*to0}
t^{2q-2} \int f^{2q-1} \, (\widetilde \log_+ f)^{-1} 
\le{M +K  \over \widetilde \log_+ t}, \quad  \forall \, t \le 1.
\eeqn
We conclude to \eqref{eq:L43to0} by gathering 
\eqref{eq:L43bdd} and \eqref{eq:L2*to0}.  
\qed

\medskip

We are now able to prove the uniqueness of solutions.

\medskip\noindent
{\sl Proof of the uniqueness part in Theorem~\ref{theo:uniq}.} 
We consider two weak solutions $(f_1,u_1)$ and $(f_2,u_2)$ to the Keller-Segel equation \eqref{eq:KS} that we write in the mild form
$$
\bal
f_i (t) 
&= e^{t \Delta} f_i(0) - \int_0^t \nabla e^{(t-s)\Delta} (f_i(s) \nabla u_i(s))\, ds
\eal
$$
and
$$
u_i(s) = e^{-\frac{\alpha}{\eps}s} e^{\frac{s}{\eps}\Delta} u_i(0) 
+ \frac{1}{\eps}\int_0^{s} e^{-\frac{\alpha}{\eps}(s-\sigma)} e^{\frac{(s-\sigma)}{\eps}\Delta} f_i(\sigma)\, d\sigma,
$$
from which we also obtain
$$
\nabla u_i(s) = e^{-\frac{\alpha}{\eps}s} e^{\frac{s}{\eps}\Delta} (\nabla u_i(0)) 
+ \frac{1}{\eps}\int_0^{s} e^{-\frac{\alpha}{\eps}(s-\sigma)} (\nabla e^{\frac{(s-\sigma)}{\eps}\Delta}) f_i(\sigma)\, d\sigma.
$$
When we assume $f_1(0) = f_2(0)$ and $u_1(0)=u_2(0)=u_0$, the difference $F := f_2 - f_1$ satisfies
\beqn\label{eq:F(t)}
\bal
F(t) &= 
- \int_0^t \nabla e^{(t-s)\Delta} \left\{ F(s) \left[e^{-\frac{\alpha}{\eps}s} e^{\frac{s}{\eps}\Delta} (\nabla u_{0})\right] \right\}\, ds \\
&\quad
- \int_0^t \nabla e^{(t-s)\Delta} \left\{ F(s) \left[\frac{1}{\eps}\int_0^s e^{-\frac{\alpha}{\eps}(s-\sigma)} \nabla e^{\frac{(s-\sigma)}{\eps}\Delta}  f_2(\sigma)\, d\sigma\right] \right\}\, ds \\
&\quad
- \int_0^t \nabla e^{(t-s)\Delta} \left\{ f_1(s) \left[\frac{1}{\eps}\int_0^s e^{-\frac{\alpha}{\eps}(s-\sigma)} \nabla e^{\frac{(s-\sigma)}{\eps}\Delta}  F(\sigma)\, d\sigma\right] \right\}\, ds \\
&=: I_1(t) + I_2(t) + I_3(t).
\eal
\eeqn
For any $t > 0$, we define  
$$
Z_p^i(t) := \sup_{0 < s \le t} s^{\frac12 - \frac1{2p}} \, \|  f_i(s) \|_{L^{\frac{2p}{p+1}}}, 
\quad
 \delta (t) := \sup_{0 < s \le t} s^{\frac14} \, \|  F(s) \|_{L^{4/3}}.
$$
We recall the explicit formula for the heat semigroup
$$
e^{t\Delta} g = \gamma(t,\cdot) *_{x} g,
\qquad
\gamma(t,x) :=  \frac{1}{4\pi t} \exp\left(- \frac{|x|^2}{4t} \right),
$$
and the following well-known inequalities that will be useful in the sequel
\beqn\label{eq:young}
\| K * g \|_{L^r} \leq \| K \|_{L^q} \|g \|_{L^p}, \quad \frac{1}{p}+\frac{1}{q}=\frac{1}{r} + 1, \quad 1\leq p,q,r\leq \infty,
\eeqn
and
$$
\| \gamma(t,\cdot) \|_{L^q(\R^2)} \leq C_q \, t^{\frac1q - 1}, 
\quad 
\| \nabla \gamma(t,\cdot) \|_{L^q(\R^2)} \leq C_q \, t^{\frac1q - \frac32}.
$$
We fix $p>2$ and we compute the quantity $t^{\frac14}\| \cdot \|_{L^{4/3}}$ for each term of \eqref{eq:F(t)}.

For the second term, we compute 
\beqn\label{eq:I2-0}
\bal
t^{\frac14} \| I_2(t) \|_{L^{4/3}} 
&\leq C(\alpha,\eps) \, t^{\frac14}\int_0^t \left\| \nabla e^{(t-s)\Delta} \left\{ F(s) \int_0^s  \nabla e^{\frac{(s-\sigma)}{\eps}\Delta}  f_2(\sigma)\, d\sigma \right\} \right\|_{L^{4/3}} \, ds\\
&\leq C \,t^{\frac14} \int_0^t  \|\nabla \gamma (t-s)\|_{L^{4/3}} \, \left\|F(s) \int_0^s  \nabla e^{\frac{(s-\sigma)}{\eps}\Delta}  f_2(\sigma)\, d\sigma \right\|_{L^{1}} \, ds  \\
&\leq C \, t^{\frac14} \int_0^t (t-s)^{-\frac34} \, \| F(s)\|_{L^{4/3}}  \, \int_0^s  \|\nabla e^{\frac{(s-\sigma)}{\eps}\Delta}  f_2(\sigma)\|_{L^{4}}\, d\sigma \, ds,
\eal
\eeqn
where we have used Young's inequality for convolution \eqref{eq:young} in the second line and H\"older's inequality in the third line.  
Now we can estimate the integral over $d\sigma$ using again Young's inequality \eqref{eq:young} with $1/4 + 1 = 1/a + (p+1)/(2p)$, i.e.\ $ 1/a = 3/4 - 1/(2p)$, by
$$
\bal
\int_0^s \left\| \nabla e^{\frac{(s-\sigma)}{\eps}\Delta}  f_2(\sigma) \right\|_{L^{4}}  \, d\sigma
&\leq \int_0^s  \| \nabla \gamma (\tfrac{t-s}{\eps}) \|_{L^a}   \, 
\| f_2(\sigma)\|_{L^{\frac{2p}{p+1}}} \, d\sigma \\
&\leq C \int_0^s (s-\sigma)^{\frac34-\frac{1}{2p} - \frac32} \, 
\| f_2(\sigma)\|_{L^{\frac{2p}{p+1}}} \, d\sigma \\
&\leq C \, Z^2_p(s) \int_0^s (s-\sigma)^{-\frac34-\frac{1}{2p}} \, \sigma^{-\frac12+\frac{1}{2p}} \, d\sigma \\
&\leq C \, Z^2_p(s) \, s^{-\frac14},
\eal
$$
since the last integral is bounded thanks to $-\frac34 - \frac{1}{2p} > -1$ from $p>2$.

Gathering that last estimate with \eqref{eq:I2-0}, it follows 
\beqn\label{eq:I2}
\bal
t^{\frac14} \| I_2(t) \|_{L^{4/3}}  
&\leq  C \, Z^2_p(t) \,  \delta (t)  \int_0^t   (t-s)^{-\frac34} \,t^{\frac14}\, s^{-\frac12} \, ds \\
&\leq C \, Z^2_p(t) \,  \delta (t) .
\eal
\eeqn

For the term $I_3$, we have
\beqn\label{eq:I3-0}
\bal
t^{\frac14} \| I_3(t) \|_{L^{4/3}} 
&\leq C \, t^{\frac14}\int_0^t \left\| \nabla e^{(t-s)\Delta} \left\{ f_1(s) \int_0^s  \nabla e^{\frac{(s-\sigma)}{\eps}\Delta}  F(\sigma)\, d\sigma \right\} \right\|_{L^{4/3}} \, ds\\
&\leq C \,t^{\frac14} \int_0^t  \| \nabla \gamma (t-s) \|_{L^{4/3}} \, \left\|f_1(s) \int_0^s  \nabla e^{\frac{(s-\sigma)}{\eps}\Delta}  F(\sigma)\, d\sigma \right\|_{L^{1}} \, ds\\
&\leq C \, t^{\frac14} \int_0^t (t-s)^{-\frac34} \, \| f_1(s)\|_{L^{\frac{2p}{p+1}}}  \, \int_0^s  \|\nabla e^{\frac{(s-\sigma)}{\eps}\Delta}  F(\sigma)\|_{L^{\frac{2p}{p-1}}}\, d\sigma \, ds. 
\eal
\eeqn
We compute the integral over $d\sigma$ in the following way
$$
\bal
\int_0^s  \|\nabla e^{\frac{(s-\sigma)}{\eps}\Delta}  F(\sigma)\|_{L^{\frac{2p}{p-1}}}\, d\sigma 
&\leq  \int_0^s  \|\nabla \gamma ( \tfrac{s-\sigma}{\eps})\|_{L^{\frac{4p}{3p-2}}}  \|F(\sigma)\|_{L^{4/3}} \, d\sigma \\
&\leq C \,  \delta (s)  \int_0^s  (s-\sigma)^{-\frac34 - \frac{1}{2p}} \, \sigma^{-\frac14} \, d\sigma  \\
&\leq C \,  \delta (s)  \, s^{-\frac{1}{2p}},
\eal
$$
since the last integral is bounded because $p>2$. 
Putting together this estimate with \eqref{eq:I3-0}, we obtain
\beqn\label{eq:I3}
\bal
t^{\frac14} \| I_3(t) \|_{L^{4/3}} 
&\leq C \, Z^1_p(t) \,  \delta (t)  \int_0^t (t-s)^{-\frac34}\, t^{\frac14}\, s^{-\frac12}\\
&\leq C \, Z^1_p(t) \,  \delta (t) .
\eal
\eeqn

For the term $I_1$, we compute
\beqn\label{eq:I1-0}
\bal
t^{\frac14} \| I_1(t) \|_{L^{4/3}} 
&\leq C \, t^{\frac14} \int_0^t \left\| \nabla e^{(t-s)\Delta} \left\{ F(s)  e^{\frac{s}{\eps}\Delta}  \nabla u_{0} \right\} \right\|_{L^{4/3}} \, ds\\
&\leq C \,t^{\frac14} \int_0^t  \| \nabla \gamma (t-s) \|_{L^{4/3}} \, \left\|F(s) e^{\frac{s}{\eps}\Delta}  \nabla u_{0} \right\|_{L^{1}} \, ds\\
&\leq C \, t^{\frac14} \int_0^t (t-s)^{-\frac34} \, \| F(s)\|_{L^{4/3}}  \, \left\|e^{\frac{s}{\eps}\Delta}  \nabla u_{0} \right\|_{L^{4}} \, ds, 
\eal
\eeqn
where we have used Young's and H\"older's inequalities. Let $K>0$ to be chosen later, we estimate
\beqn\label{eq:nablau0}
\|e^{\frac{s}{\eps}\Delta}  \nabla u_{0} \|_{L^{4}}
\leq \|e^{\frac{s}{\eps}\Delta}  \nabla u_{0} \,\indiq_{\{|\nabla u_0| \leq K\}} \|_{L^{4}}
+\|e^{\frac{s}{\eps}\Delta}  \nabla u_{0} \,\indiq_{\{|\nabla u_0| \geq K\}} \|_{L^{4}}. 
\eeqn
Using Young's inequality, we have
$$
\|e^{\frac{s}{\eps}\Delta}  \nabla u_{0} \,\indiq_{\{|\nabla u_0| \leq K\}} \|_{L^{4}}
\leq  \| \gamma ( \tfrac{s}{\eps}) \|_{L^{1}}  \|\nabla u_{0} \,\indiq_{\{|\nabla u_0| \leq K\}} \|_{L^{4}} 
\leq C K^{\frac12}\, \| \nabla u_0\|_{L^2}^{1/2}.
$$
Using Young's inequality again for the second term in \eqref{eq:nablau0}, we have 
$$
\|e^{\frac{s}{\eps}\Delta}  \nabla u_{0} \,\indiq_{\{|\nabla u_0| \geq K\}} \|_{L^{4}}
\leq  \| \gamma ( \tfrac{s}{\eps}) \|_{L^{4/3}}   \|\nabla u_{0} \,\indiq_{\{|\nabla u_0| \geq K\}} \|_{L^{2}}
\leq C s^{-\frac14} \, \varphi(K),
$$
where
$$
\varphi(K) := \|\nabla u_{0} \,\indiq_{\{|\nabla u_0| \geq K\}} \|_{L^{2}} \to 0 
\quad\text{as}\quad
K \to +\infty,
$$
by the dominated convergence theorem. Putting together that last estimates in \eqref{eq:nablau0} and choosing $K = s^{-\frac14}$, it follows
$$
\|e^{\frac{s}{\eps}\Delta}  \nabla u_{0} \|_{L^{4}} \leq C\, s^{-\frac14}\, \alpha(s)
\quad\text{with}\quad
\alpha(s) := s^{\frac18} + \varphi(s^{-\frac14}) \xrightarrow[s\to 0]{} 0.
$$
Coming back to \eqref{eq:I1-0}, we obtain
\beqn\label{eq:I1}
\bal
t^{\frac14} \| I_1(t) \|_{L^{4/3}} 
&\leq C \, t^{\frac14} \int_0^t (t-s)^{-\frac34} \, \| F(s)\|_{L^{4/3}}  \, s^{-\frac14} \, \alpha(s) ds \\
&\leq C \, \left( \sup_{0<s\leq t} \alpha(s)  \right)  \delta (t)  \int_0^t t^{\frac14} (t-s)^{-\frac34} \, s^{-\frac12}\, ds \\
&\leq C \, \left( \sup_{0<s\leq t} \alpha(s)  \right)  \delta (t) .
\eal
\eeqn

Gathering \eqref{eq:I2}, \eqref{eq:I3} and \eqref{eq:I1} and using Lemma~\ref{lem:L43to0}, we conclude to
$$
 \delta (t)  \leq C \left[\sup_{0<s\leq t} \alpha(s) + Z_p^1(t) + Z_p^2(t)\right]  \delta (t)  \leq \frac12\,  \delta (t) ,
$$
for $t\in (0,T)$, $T>0$ small enough. That in turn implies $ \delta (t)  \equiv 0$ on $[0,T)$. We may then repeat the argument for later times and conclude to the uniqueness of the solution. 
\qed

\medskip

\section{Self-similar solutions and linear stability} 
\label{sec:autosimilaire}

\subsection{Convergence of the stationary solutions}

First, for a given mass $M \in (0,8\pi)$ and a given parameter $\eps \in (0,1/2)$, we consider the self-similar profile $(G_\eps,V_\eps)$ which is the unique solution of the system of elliptic equations \eqref{eq:KSprofileMass-intro}-\eqref{eq:KSprofile-intro}. We also consider the unique positive solution $(G,V)$ to 
the system of equations corresponding to the limit case $\eps = 0$  
\bear \label{eq:KSprofile0}
&& \Delta G - \nabla (G \, \nabla V- {1 \over 2} \, x \, G) = 0 \quad \hbox{in} \quad   \R^2,  \quad \int_{\R^2} G \, dx = M, 
\\ \nonumber
&& \Delta V + G = 0   \quad \hbox{in} \quad  \R^2.
\eear
It is worth emphasizing that $(G,V)$ is the unique self-similar profile associated to the parabolic-elliptic Keller-Segel equation, see \cite{CamposDolbeault2012,EM}. 

\begin{lem}\label{lem:LinStabBd} There exists a  constant $C$ such that for any $\eps \in (0,1/4]$
\beqn\label{eq:Gepsbound}
 0 \le G_\eps(x) \le C \, e^{-|x|^2/4},
\eeqn
\beqn\label{eq:borneVeps}
 \sup_{x \in \R^2} (\frac{1}{|x|} + \langle x \rangle) \, |\nabla V_\eps(x)| \le C,
\eeqn
and
\beqn\label{eq:borneDeltaVeps}
 \sup_{x \in \R^2}  |\Delta V_\eps(x)| \le C.
\eeqn

\end{lem}

\noindent
{\sl Proof of Lemma~\ref{lem:LinStabBd}.} We split the proof into three steps.

\medskip
\noindent
{\it Step 1.} The estimate \eqref{eq:Gepsbound} has been proved in \cite{BCD}. More precisely it is a consequence of equations (26) and (49) in \cite{BCD}, and 
$$
G(0) = b, \quad 0 \le M(\eps,b) \le 4 \pi \, \min(2,b).
$$
Here the parametrization of $G$ is made in function of $\eps$ and $b = G(0)$ instead of $\eps$ and $M$ because 
this dependence is more tractable. Observe that the above estimate guarantees that the mass is subcritical, i.e.\ $M(\eps,b) \le 8\pi$. 

\medskip
\noindent
{\it Step 2.} Since $V_\eps$ and $G_\eps$ are radially symmetric functions, the equation on $V_\eps$ writes
\beqn\label{eq:VGrad}
V_\eps'' + \bigl( {1 \over r} + \frac12 \eps \, r \bigr) \, V_\eps' + G_\eps = 0 \quad \forall \, r > 0, 
\eeqn
where we abuse notation in writing $V_\eps(r) = V_\eps(x)$, $G_\eps(r) = G_\eps(|x|)$, $r = |x|$. The function $V_\eps$ is smooth and the equation is 
complemented with the boundary conditions $V_\eps'(0) = V'_\eps(\infty) = 0$. Defining $w := (r V_\eps')^2$, we find
\bean 
{w' \over 2} &=& - \frac12 \eps \, r \, w - G_\eps \, V_\eps' \, r^2 \le C \, \sqrt{w} { r \over \langle r \rangle^3}, 
\quad C := \sup_{r > 0} G_\eps \, \langle r \rangle^3.
\eean 
As a consequence
$$
{d \over dr} \sqrt{w} \le C \,   { r \over \langle r \rangle^3},
$$
and then 
$$
\sqrt{w} \le C \, (1 \wedge r)^2,
$$
from which the inequality $\sup_x [\la x \ra \, |\nabla V_\eps (x)| ]\le C$ of \eqref{eq:borneVeps} follows.

\medskip
\noindent
{\it Step 3.} We rewrite \eqref{eq:VGrad} as 
$$
{1 \over r}   (V_\eps' r)' = w := - G_\eps - \frac12 \eps r V'_\eps \in L^\infty,
$$
which implies 
$$
|V'_\eps(r) r |   = \left| \int_0^r s w(s) \, ds \right| \le C \, r^2.
$$
This completes the estimate \eqref{eq:borneVeps}. Coming back to \eqref{eq:VGrad}, we also obtain
$$
|V''_\eps(r)|   \le C,
$$
which gives \eqref{eq:borneDeltaVeps} and ends the proof.
\qed

\begin{cor}\label{cor:cvgceGVeps} 
As $\eps \to 0$, there  hold
$$
G_\eps \to G \,\, \hbox{ in } \,\, W^{2,p} \quad  \forall\, p \in (1,\infty),
$$
and 
$$
\nabla V_\eps \to \nabla V  \,\, \hbox{ in } \,\,  L^{\infty}_1 , \quad
 (\Delta V_\eps)_{\eps > 0} \,\,\hbox{uniformly bounded in } L^\infty \, \, \hbox{and } \Delta V_\eps \to \Delta V  \,\, \hbox{ a.e}.   . 
$$
\end{cor}

\noindent{\sl Proof of Corollary~\ref{cor:cvgceGVeps}.} 
Coming back to \eqref{eq:KSprofile-intro} and using Lemma~\ref{lem:LinStabBd}, for any $p \in (1,\infty)$, we have
$$
L G_\eps = \nabla G_\eps \cdot \nabla V_\eps + G_\eps \Delta V_\eps \in L^p,
$$
where $L$ denotes the operator $L G_\eps := \Delta G_\eps - \nabla \cdot ( \frac12 x G_\eps)$.
By elliptic regularity we obtain that $G_\eps$ is uniformly bounded (with respect to $\eps \in (0,1/2)$) in $W^{2,p}$.
Thanks to previous estimates and Lemma~\ref{lem:LinStabBd}, there exists $(\bar G,\bar V)$ and
a subsequence (still denoted as $(G_{\eps}, V_{\eps}))$ such that $G_\eps \to \bar G$, $V_\eps \to \bar V$.  We may pass to the limit (in the weak sense) in the system of equations, and we find
$$
\bar V'' +  {1 \over r} \bar V'  + \bar G = 0, \quad \bar V'(0) = \bar V'(\infty) = 0.
$$
We conclude that $(\bar G,\bar V)$ is a solution to the stationary equation \eqref{eq:KSprofile0}, so that $(\bar G , \bar V) = (G,V)$. 
\qed

\medskip
 
\subsection{Splitting structure for the linearized operator}

The evolution equation in self-similar variables writes (see \eqref{eq:KSresc} and \eqref{eq:KSresc2eme})
\beqn \label{eq:KSinSSV}
\left\{
\bal
\partial_t g  &=  \Delta g + \nabla ({1 \over 2} \, x \, g - g  \, \nabla v)  ,\\ 
\partial_t v  &=    {1 \over \eps} ( \Delta v + g) + {1 \over 2} \, x \cdot \nabla v ,
\eal
\right.
\eeqn
and the associated linearized equation around the self-similar profile $(G_\eps, V_\eps)$ is given by
\beqn \label{eq:OpeLambdaeps}
\left\{
\bal
\partial_t f 
&= \Lambda_{1,\eps} (f,u) := \Delta f + \nabla ({1 \over 2} \, x \, f - f  \, \nabla V_\eps - G_\eps \, \nabla u)  , \\
\partial_t u 
&= \Lambda_{2,\eps} (f,u) :=  {1 \over \eps} ( \Delta u+f ) + {1 \over 2} \, x \cdot \nabla u ,
\eal
\right.
\eeqn
which we also denote $\partial_t (f,u) = \Lambda_\eps (f,u) = \big(\Lambda_{1,\eps} (f,u), \Lambda_{2,\eps} (f,u)   \big)$. From now on, we restrict ourselves to a radially symmetric setting.

We introduce some classical notation of operator theory. 
Let $X$ be a Banach space and consider a linear operator $\Lambda:X \to X$. We denote by $S_\Lambda(t) = e^{t \Lambda}$ the semigroup of operators generated by $\Lambda$, by $\Sigma(\Lambda)$ its spectrum and by $\Sigma_d(\Lambda)$ its discrete spectrum. Moreover, for two Banach spaces $X,Y$ we denote by $\BBB(X,Y)$ the space of bounded linear operators from $X$ to $Y$ and by $\| \cdot \|_{\BBB(X,Y)}$ its norm, with the usual shorthand $\BBB(X) = \BBB(X,X)$. We also define the subset $ \CC_a \subset \C$ for any $a \in \R$ by
\beqn\label{eq:Delta_a}
\CC_a := \{ z \in \C \; \mid \;  \Re e z > a  \} .
\eeqn

\smallskip

Let us denote by $L^2_{rad}$ the $L^2$ space of radially symmetric functions and by $L^2_{k,j}$, $j < k$, the following space
$$
L^2_{k,j} := \left\{ g \in L^2_k  \; \mid \;  \int x^{\alpha} g = 0, \; \forall\, \alpha \in \N^2, \, |\alpha| \le j  \right\}. 
$$ 
We fix $k>7$ and we introduce the Hilbert space 
\beqn\label{eq:X}
X := X_1 \times X_2  , \quad X_1 :=  L^2_{rad}  \cap L^2_{k,0} \subset L^2_{k,1} ,
\quad X_2 = L^2_{rad},
\eeqn
associated to the norm 
\bear\label{eq:normeX}
\| (f,u)  \|_X^2 :=  \| f \|_{L^2_{k}}^2  + \| u  \|_{L^2}^2.
\eear

We now state a property of the spectrum of $\Lambda_\eps$ in $X$ that is the main result of this subsection.
\begin{prop}\label{prop:spectre}
Fix some $a^* >-1/2$.
There exist $\eps^*, r^* > 0$ such that in $X$
$$
\forall \, \eps \in (0,\eps^*) \qquad\Sigma(\Lambda_\eps) \cap \CC_{a^*} \subset \Sigma_d(\Lambda_\eps) \cap B(0,r^*).
$$

\end{prop}

We define the bounded operator $\AA = (\AA_1,\AA_2) : X \to X$ by 
\beqn\label{eq:def-A}
\AA_1(f,u) := N { \chi_R [f]} :=  N  { (\chi_R f - \chi_1 \langle \chi_R f \rangle)}, \quad \AA_2(f,u) := 0, 
\eeqn
for some constants $N,R >0$ to be chosen later and a smooth non-negative radially symmetric cut-off function $\chi_R(x) := \chi(x/R)$ with $\chi \equiv 1$ on $B_{1/2}$, $\Supp \chi \subset B_2$ and $\langle \chi_1 \rangle = 1$.
We can split the operator $\Lambda_\eps = \AA + \BB_\eps$ and we shall investigate some properties of $\AA$ and $\BB_\eps$ in the next lemmas before proving 
Proposition~\ref{prop:spectre}.

\begin{lem}\label{lem:A&B} In the above splitting, we may choose $N_*$ and $R_*$ large enough in such a way that 
for any $N \ge N_*$, $R \ge R_*$, the operator $\BB_\eps$ is $a$-hypo-dissipative in $X$ for any $a \in (-1/2,0)$, in the sense that 
$$
\| S_{\BB_\eps} (t) \|_{\BBB(X)} \le C_a \, e^{at}, \quad \forall \, t \ge 0,
$$
for some constant $C_a > 0$. 
\end{lem}

\begin{proof}[Proof of Lemma \ref{lem:A&B}]
First of all, thanks to Lemma \ref{lem:kappa*f} and using the notation of Appendix~\ref{sec:A}, we see that in $X$ the norm of $L^2_k \times L^2$ is equivalent to the norm defined by 
\beqn\label{eq:normeX*}
\| (f,u) \|_{X_*}^2 := \| f \|_{L^2_k}^2 +  \eta \,  \|  u - \kappa_f \|_{L^2}^2,
\eeqn
for any fixed $\eta > 0$. 
We also observe that, thanks to Lemma~\ref{lem:K*f}, we have
$$
\| \nabla \kappa_f \|_{L^2} = \| \KK*f \|_{L^2} \le C \| f \|_{L^2_\ell}, \quad \forall\, \ell >2,
$$
and
$$
\| \nabla \kappa_f \|_{L^2_1} = \| \KK*f \|_{L^2_1} \le C \| f \|_{L^2_\ell}, \quad \forall\, \ell >3,
$$
thus we fix some $\ell \in (3,k)$ from now on.


We consider the equation 
\beqn\label{eq:dfu=B}
\partial_t (f,u) = \BB_\eps (f,u) = \Lambda_\eps(f,u) - \AA(f,u)
\eeqn
and split the proof into three steps.

\medskip
\noindent
{\it Step 1.  }
We write the equation satisfied by $f$ as 
\bear \nonumber
&& \partial_t f = \Delta f + \nabla ({1 \over 2} \, x \, f - f  \, \nabla V_\eps - G_\eps \, \nabla u )  
 - N { \chi_R [f]}.
\eear
Using that $\langle f \rangle = 0$ and the notation $\chi_R^c = 1 - \chi_R$, we compute
\beqn\label{eq:dissip-f}
\bal
\frac12 \frac{d}{dt} \int f^2 \la x \ra^{2k}
& = \int \Delta f \, f \la x \ra^{2k} + \frac12 \int \nabla \cdot(xf) \, f \la x \ra^{2k}
- \int \nabla\cdot (f \nabla V_\eps) \, f \la x \ra^{2k}  \\
&\quad - \int \nabla\cdot (G_\eps \nabla u ) \, f \la x \ra^{2k}
 -  \int N{ \chi_R [f]} f \la x \ra^{2k}\\
&= - \int |\nabla f|^2 \la x \ra^{2k} 
+ \int \{ \varphi(x) - N \chi_R(x) \} \, f^2 \la x \ra^{2k} 
\\
&\quad- \int \nabla\cdot (G_\eps \nabla u) \, f \la x \ra^{2k} -  N  \int  f \, \chi_1 \, \langle x \rangle^{2k} dx \, \langle \chi_R^c f \rangle  ,
 \eal
\eeqn
using that $\langle f \rangle = 0$ in order to replace $\chi_R$ with $\chi_R^c$ in last line, and where
$$
\bal
\varphi(x) 
&= \left(\frac12 \Delta \la x \ra^{2k} - \frac12 x \cdot \nabla \la x \ra^{2k} 
+\frac14 \nabla \cdot (x \la x \ra^{2k}) - \frac12 \nabla \cdot (\nabla V_\eps \la x \ra^{2k})
 + \nabla V_\eps \cdot \nabla \la x \ra^{2k} \right) \la x \ra^{-2k}\\
&= -\frac12 (k-1) + k(2k+1/2) \la x \ra^{-2} - k(2k-2)\la x \ra^{-4} -\frac12 \Delta V_\eps + k (\nabla V_\eps \cdot x) \la x \ra^{-2}.
\eal
$$
We observe that, thanks to Lemma \ref{lem:LinStabBd}, we have $(\nabla V_\eps \cdot x) \la x \ra^{-2} \to 0$ as $|x| \to \infty$. From \eqref{eq:KSprofile-intro}, we also have that
$$
-\frac12 \Delta V_\eps = \frac12 G_\eps + \frac\eps4 \, x \cdot \nabla V_\eps
$$
with $G_\eps \to 0$ as $|x|\to \infty$ from \eqref{eq:Gepsbound} and $|x\cdot \nabla V_\eps| \le C_{V_\eps}$ from \eqref{eq:borneVeps}. All together, it follows 
\beqn\label{eq:assymptotique}
\limsup_{|x| \to \infty }\varphi(x) \le  -\frac12 (k-1 - \eps \, C ),
\eeqn
where $C>0$ is the constant exhibited in \eqref{eq:borneVeps}.

For the third term in \eqref{eq:dissip-f}, for any $\delta>0$, thanks to H\"older's inequality and using that $G_\eps(x)  \le C \la x \ra^{-\alpha}$ from \eqref{eq:Gepsbound}, we get
$$
\bal 
 -  \int \nabla\cdot(  G_\eps  \nabla u )  \, f \, \la x \ra^{2k} 
&=   \int \nabla f  \cdot  \nabla u  \,  G_\eps  \, \la x \ra^{2k} 
+ \int G_\eps \nabla u  \cdot \nabla(\la x \ra^{2k}) \, f
\\
&\le  \delta \| \nabla f \|_{L^2}^2  
+ C(\delta) \| \nabla u \|_{L^2}^2 
+ C \| f \|_{L^2}^2 \\
&\le  \delta \| \nabla f \|_{L^2}^2  
+ C(\delta) \| \nabla (u - \kappa_f) \|_{L^2}^2 
+ C(\delta) \| \nabla \kappa_f \|_{L^2}^2
+ C \| f \|_{L^2}^2 \\
&\le  \delta \| \nabla f \|_{L^2}^2  
+ C(\delta) \| \nabla (u - \kappa_f) \|_{L^2}^2 
+ C(\delta) \| f \|_{L^2_{\ell}}^2
+ C \| f \|_{L^2}^2,
\eal
$$
where we recall that we have fixed some $\ell \in (3,k)$. For the fourth term in \eqref{eq:dissip-f}, we have 
$$
\bal 
 - N  \int  f \, \chi_1 \, \langle x \rangle^{2k} dx \, \langle \chi_R^c f \rangle  
&\le N \, \|\chi_1 \|_{L^2_{2k-\ell}} \, \|\chi_R^c \, \langle x \rangle^{-\ell} \|_{L^2} \, \|f \|_{L^2_\ell}^2  
\\
&\le N \, R^{ 1 - \ell} \, C_\ell \, \| f \|_{L^2_\ell}^2 .  
\eal
$$
We conclude this step by gathering the previous estimates to obtain
\beqn\label{eq:step1}
\bal
\frac12 \frac{d}{dt} \| f \|_{L^2_k}^2 
&\le
- (1 - \delta)\| \nabla f \|_{L^2_k}^2 + C(\delta) \| \nabla( u -\kappa_{ f}) \|_{L^2}^2
\\
&\quad + \int \left\{ \varphi(x) + (C(\delta)  + C \, N \, R^{ 1 - \ell}  ) \la x \ra^{2(\ell - k)} - N   \chi_R(x) \right\} f^2 \la x \ra^{2k} .  
\eal
\eeqn

\medskip
\noindent
{\it Step 2.}
From the second equation in \eqref{eq:OpeLambdaeps}, we have
$$
\bal
\frac12{d \over dt} \int (u-\kappa_{ f})^2 
= \int  (u-\kappa_{ f}) \, \left\{ {1 \over \eps} (\Delta u +  f) + \frac12 x \cdot \nabla (u-\kappa_{ f}) + \frac12 x \cdot \nabla \kappa_{ f} - \partial_t \kappa_{ f} \right\}
\\
= - {1 \over \eps} \int  |\nabla (u-\kappa_{ f}) |^2 
-\frac12 \int  (u-\kappa_{ f})^2 
+  \int  (u-\kappa_{ f})  \left\{ \frac12 x \cdot \nabla \kappa_{ f} - \partial_t \kappa_{ f} \right\} ,
\eal
$$
and we shall estimate the last integral. 
Since $\partial_t \kappa_{f} = \kappa * \partial_t f$,
we may write
$$
\bal
&\int  (u-\kappa_{f})  \left\{ \frac12 x \cdot \nabla \kappa_{ f} - \partial_t \kappa_{ f} \right\}
 = \frac12 \int  (u-\kappa_{ f})  \bigl\{ x \cdot \nabla \kappa_{ f} \}\\
&\qquad
- \int  (u-\kappa_{ f})  \kappa * \left\{ \Delta f + \frac12 \nabla(x  f) - \nabla( f \nabla V_\eps) - \nabla(G_\eps \nabla u) - N { \chi_R [f]} \right\} ,
\eal
$$
and we estimate each of these terms separately. First, the first and third terms together gives 
$$
I_1 := \frac12 \int  (u-\kappa_{ f})  \bigl\{ x \cdot \nabla \kappa_{ f} -   \kappa * \nabla(x  f) \} = 0,
$$
because
$$
\bal
x \cdot \nabla \kappa_{ f} -   \kappa * \nabla(x  f) 
&=  x \cdot \KK*f  -   \KK * \nabla(x  f)
\\
&=-{1 \over 2\pi}\int {(x-y) \over |x-y|^2} \bigl\{ x \, f(y) - y \, f(y) \bigr\} \, dy = -{1 \over 2\pi}\int   f(y)   \, dy = 0. 
\eal
$$
%
Next, we have
$$
\bal
I_2 
&:= -\int  (u-\kappa_{ f})  \kappa * \{ \Delta f \} 
 = \int (u-\kappa_{ f}) f \\
&\le \delta \,\| (u-\kappa_{ f}) \|_{L^2}^2 + C(\delta) \, \| f \|_{L^2}^2.
\eal
$$
Furthermore, since $f \nabla V_\eps \in L^2_{k,0}$, we can apply Lemma~\ref{lem:K*f} and use that $\nabla V_\eps \in L^\infty$ to obtain
$$ 
\bal
I_3 &:=  \int  (u-\kappa_{ f})  \kappa *  \nabla( f \nabla V_\eps) 
=  \int  (u-\kappa_{ f}) \, \KK_i * ( f \partial_i V_\eps)  \\
&\leq 
\delta \| (u-\kappa_{ f}) \|_{L^2}^2
+ C(\delta) \|   f \nabla V_\eps \|_{L^2_{\ell}}^2   \\
&\leq 
\delta \| (u-\kappa_{ f}) \|_{L^2}^2
+ C(\delta) \|  f \|_{L^2_{\ell}}^2.
\eal
$$
For the next term, using Lemma~\ref{lem:K*f}, since $G_\eps \nabla u \in L^2_{k,0}$, and the bound $G_\eps \in L^\infty_{k}$, we have
$$
\bal
I_4 &:=  \int  (u-\kappa_{ f})  \kappa *   \nabla(G_\eps \nabla u)
=\int    (u-\kappa_{ f}) \, \KK_i *  (  G_\eps \, \partial_i u ) \\
&\leq 
\delta \| (u-\kappa_{ f}) \|_{L^2}^2 + C(\delta) \| \KK *(G_\eps \nabla u) \|_{L^2}^2\\
&\leq 
\delta \| (u-\kappa_{ f}) \|_{L^2}^2 + C(\delta) \| G_\eps \nabla u \|_{L^2_k}^2\\
&\leq 
\delta \| (u-\kappa_{ f}) \|_{L^2}^2 + C(\delta) \|  \nabla u \|_{L^2}^2\\
&\leq 
\delta \| (u-\kappa_{ f}) \|_{L^2}^2  + C(\delta) \|  \nabla (u-\kappa_f) \|_{L^2}^2 + C(\delta) \|  \nabla \kappa_f \|_{L^2}^2\\
&\leq 
\delta \| (u-\kappa_{ f}) \|_{L^2}^2  + C(\delta) \|  \nabla (u-\kappa_f) \|_{L^2}^2 + C(\delta) \|  f \|_{L^2_{\ell}}^2.
\eal
$$
For the last term and thanks to Lemma~\ref{lem:kappa*f}, we finally have
$$
\bal
I_5 &:= N \int (u - \kappa_f) \kappa * \{ \chi_R [f]  \} \le C \,  \| u - \kappa_f \|_{L^2} \,  N \, \|  \chi_R [f] \|_{L^2_{\ell,1}} 
\\
&\le  \delta  \| u - \kappa_f \|_{L^2}^2 + C(\delta) N^2 \| f \|_{L^2_\ell}^2.
\eal
$$

Putting together all the estimates of this step, we deduce
\beqn\label{eq:step2}
\bal
\frac12\frac{d}{dt} \| u - \kappa_{f} \|_{L^2}^2
& \le - \Bigl(\frac1\eps - C(\delta) \Bigr) \| \nabla(u-\kappa_{ f} ) \|_{L^2}^2
-  \Bigl( \frac12 - \delta \Bigr) \| u - \kappa_{f} \|_{L^2}^2 
+C(\delta) N^2 \| f \|_{L^2_{\ell}}^2.
\eal
\eeqn

\medskip
\noindent
{\it Step 3. Conclusion.}
Gathering \eqref{eq:step1} and \eqref{eq:step2}, we obtain
\beqn\label{eq:Norm0*}
\bal
\frac12 \frac{d}{dt} \| (f,u) \|_{X_*}^2
&\le \int \left\{ \varphi(x) +  \Bigl[  (1+ \eta N^2) C(\delta)  +  {CN \over R^{\ell-1}} \Bigr] \la x \ra^{2(\ell-k)}  - N { \chi_R}(x) \right\} |f|^2  \la x \ra^{2k}    \\
&
- \eta \Bigl( \frac12 - \delta \Bigr) \| u - \kappa_f \|_{L^2}^2 
- (1-\delta) \| \nabla f \|_{L^2_k}^2 - \eta  \Bigl( \frac{1}{\eps}  - C(\delta) \Bigr) \| \nabla (u-\kappa_f)\|_{L^2}^2 .
\eal   
\eeqn
Taking then first $\delta \in (0,1)$ small enough and next $\eps \in (0,1) $ small enough, it follows that for $\eta = N^{-3}$ and $R = N$ 
$$
\bal
\frac12 \frac{d}{dt} \| (f,u) \|_{X_*}^2
&\le \int \left\{ \bar \varphi_N(x) - N {\chi_R}(x) \right\} |f|^2  \la x \ra^{2k} 
+ a \| \nabla f \|_{L^2}^2\\
&\quad
+  a \,  \eta \,  \| u - \kappa_{ f} \|_{L^2}^2 
+  a \,   \eta \,  \| \nabla(u - \kappa_{ f}) \|_{L^2}^2,       
\eal   
$$
for any $a > -1/2$, where $\bar \varphi_N(x) = \varphi(x) + C  (1+ N^{-1} +N^{2-\ell} )  \la x \ra^{2(\ell - k)}$ has the same asymptotic behaviour as $\varphi(x)$ when $|x| \to \infty$ and $\bar\varphi_N$ decreases as $N$ increases. 
We can choose $N$ large enough such that
$$
\bar \varphi_N(x) - N {\chi_R}(x) \le a, \quad \forall\, x\in \R^2,
$$
which yields that $\BB_\eps$ is $a$-hypo-dissipative for any $a > -1/2$.
%
%
\end{proof}

We introduce the space
\beqn\label{eq:Y}
Y := Y_1 \times Y_2, \quad Y_1 := H^1_{k} \cap L^2_{k,0} \cap L^2_{rad}, \quad Y_2 := H^1 \cap L^2_{rad},
\eeqn
endowed with the norm
\beqn\label{eq:normeY}
 \| (f,u) \|_Y^2 := \| (f,u) \|_X^2 + \| \nabla f \|_{L^2_k}^2 + \| \nabla u \|_{L^2}^2.
\eeqn

For the operator $\AA$ defined in \eqref{eq:def-A}, the following result holds true.

\begin{lem}\label{lem:Abounded}
There hold $\AA \in \BBB(X)$ and $\AA \in \BBB(Y)$.
\end{lem}

\begin{proof}[Proof of Lemma \ref{lem:Abounded}]
The proof is straightforward so we omit it.
\end{proof}

\begin{lem}\label{lem:T} 
We can choose $N$ and $R$ large enough such that $\BB_\eps$ is $a$-hypo-dissipative in $Y$ for any $a\in(-1/2,0)$, i.e.
$$
\| S_{\BB_\eps}(t) \|_{\BBB(Y)} \le C  e^{at}, \quad \forall \, t \ge 0.
$$
Moreover, we also have
$$
\| S_{\BB_\eps}(t) \|_{\BBB(X,Y)} \le C \, t^{-1/2} \, e^{at} , \quad \forall \, t \ge 0.
$$

\end{lem}

\begin{proof}[Proof of Lemma \ref{lem:T}] 
We introduce the following norm
\beqn\label{eq:normeY*}
 \| (f,u) \|_{Y_*}^2 := \| (f,u) \|_{X_*}^2 + \eta_1 \| \nabla f \|_{L^2_k}^2 + \eta_1 \| \nabla (u - \kappa_f) \|_{L^2}^2,
\eeqn
which is equivalent to \eqref{eq:normeY} for any $\eta_1 >0$ thanks to Lemma \ref{lem:K*f}.
We recall that, as in Lemma~\ref{lem:A&B}, we have
$$
\| \nabla u \|_{L^2} \le \| \nabla (u-\kappa_f) \|_{L^2} + \| \nabla \kappa_f \|_{L^2} \le\| \nabla (u-\kappa_f) \|_{L^2} + C \| f \|_{L^2_\ell},   
$$
for the same fixed $\ell \in(3, k)$, and moreover we observe that
$$
\| \nabla^2 u \|_{L^2} \le \| \nabla^2 (u-\kappa_f) \|_{L^2} +  \| \nabla^2 \kappa_f \|_{L^2}
\le \| \nabla^2 (u-\kappa_f) \|_{L^2} +  \|  f \|_{L^2}.
$$
We consider now the equation \eqref{eq:dfu=B} and we split the proof of the announced results into three steps.

\medskip
\noindent
{\it Step 1. $L^2$ differential inequality.}
For $i=1,2$, $\partial_i u$ verifies 
$$
\partial_t (\partial_i u) = \frac1\eps \Delta (\partial_i u) + \frac1\eps \partial_i f + \frac12 x \cdot \nabla(\partial_i u) + \frac12 \partial_i u  .
$$
We have then
$$
\bal
\frac12 \frac{d}{dt} \| \partial_i (u-\kappa_f) \|_{L^2}^2 
&= \int \partial_i (u-\kappa_f) \partial_t \{ \partial_i u - \partial_i \kappa_f \}\\
&= \frac1\eps \int \partial_i (u-\kappa_f) \Delta\{ \partial_i (u- \kappa_f) \} 
+ \frac12 \int \partial_i (u-\kappa_f)\partial_i u \\
&\quad
+\frac12 \int \partial_i (u-\kappa_f) x \cdot \nabla \{ \partial_i(u-\kappa_f) \} 
+ \frac12\int \partial_i (u-\kappa_f) x \cdot \nabla \{ \partial_i \kappa_f \} \\
&\quad
 - \int \partial_i (u-\kappa_f) \KK_i * (\partial_t f)\\
&=: T_1 + T_2 + T_3 + T_4 + T_5. 
\eal
$$
For the first term, we easily have
$$
T_1 = - \frac1\eps \| \nabla\{ \partial_i (u - \kappa_f)\} \|_{L^2}^2.
$$
For the second term, we have
$$
T_2 \le C \| \nabla (u - \kappa_f) \|_{L^2}^2 + C \| \nabla u \|_{L^2}^2
\le C \| \nabla (u - \kappa_f) \|_{L^2}^2 + C \| f \|_{L^2_{\ell}}^2.
$$
We also easily see that
$$
T_3 = - \frac12 \| \partial_i (u - \kappa_f) \|_{L^2}^2 \le 0,
$$
and, for the fourth term, that
$$
T_4 \le C \| \nabla (u - \kappa_f) \|_{L^2}^2 + C \| D^2 \kappa_f \|_{L^2_1}^2
 \le C \| \nabla (u - \kappa_f) \|_{L^2}^2 + C \| f \|_{L^2_1}^2.
$$
For the last term $T_5$, we use the equation satisfied by $f$ and we get
$$
\bal
T_5 &= -  \int \partial_i (u-\kappa_f) \KK_i * \left\{ \Delta f + \frac12 \nabla (xf) - \nabla(f \nabla V_\eps) - \nabla (G_\eps \nabla u) -N \chi_R[f]   \right\} \\
&=: T_{51} + T_{52} + T_{53} + T_{54} + T_{55}.
\eal
$$
We estimate each term separately. We have
$$
T_{51} \le C \| \nabla (u-\kappa_f) \|_{L^2}^2 + C \| \nabla f \|_{L^2}^2
$$  
and
$$
T_{52} \le C \| \nabla (u-\kappa_f) \|_{L^2}^2 + C \| \nabla^2 \kappa * (xf) \|_{L^2}^2
\le C \| \nabla (u-\kappa_f) \|_{L^2}^2 + C \| f \|_{L^2_1}^2.
$$
Using that $\nabla V_\eps \in L^\infty$, we have
$$
\bal
T_{53} 
&\le C \| \nabla (u-\kappa_f) \|_{L^2}^2 + C\| \nabla^2 \kappa * (f \nabla V_\eps) \|_{L^2}^2 \\
&\le C \| \nabla (u-\kappa_f) \|_{L^2}^2 + C\| f \nabla V_\eps \|_{L^2}^2\\
&\le C \| \nabla (u-\kappa_f) \|_{L^2}^2 + C\| f  \|_{L^2}^2,
\eal
$$
and arguing as above with $G_\eps \in L^\infty$, we also obtain
$$
\bal
T_{54} 
&\le C \| \nabla (u-\kappa_f) \|_{L^2}^2 + C\| \nabla^2 \kappa * ( G_\eps \nabla u) \|_{L^2}^2 \\
&\le C \| \nabla (u-\kappa_f) \|_{L^2}^2 + C\| G_\eps \nabla u \|_{L^2}^2\\
&\le C \| \nabla (u-\kappa_f) \|_{L^2}^2 + C\| \nabla u  \|_{L^2}^2 \\
&\le C \| \nabla (u-\kappa_f) \|_{L^2}^2 + C\|  f  \|_{L^2_{\ell}}^2.
\eal
$$
For the last term, using Lemma \ref{lem:K*f}, we have
$$
\bal
T_{55} 
&\le C \| \nabla (u-\kappa_f) \|_{L^2}^2+ C N^2 \| \KK* \chi_R[f] \|_{L^2}^2 \\
&\le C \| \nabla (u-\kappa_f) \|_{L^2}^2+ C N^2 \| f \|_{L^2_\ell}^2.
\eal
$$
Gathering these previous estimates, we finally obtain
\beqn\label{eq:Ystep1}
\bal
\frac12 \frac{d}{dt}  \| \nabla (u - \kappa_{ f}) \|_{L^2}^2 
& \le  - \frac{1}{\eps}  \| \nabla^2 (u - \kappa_{f}) \|_{L^2}^2 
+ C(1+N^2) \|  f \|_{L^2_{\ell}}^2 \\
&\quad
+ C \| \nabla f \|_{L^2}^2
+ C \| \nabla (u - \kappa_f) \|_{L^2}^2.
\eal
\eeqn

\medskip
\noindent
{\it Step 2. $H^1$ differential inequality.}
We write the equation satisfied by $\partial_i f$ which is nothing but
$$
\bal
\partial_t (\partial_i f) 
&= \BB_{\eps,1} (\partial_i f, \partial_i u) - \frac12 \partial_i f 
- \nabla(f \nabla(\partial_i V_\eps))
- \nabla(\partial_j G_\eps \nabla u)
- N (\partial_i { \chi_R})f
 + N \la \chi_R f \ra  \partial_i \chi_1,
\eal
$$
and then we can write
$$
\bal
\frac12 \frac{d}{dt} \| \partial_i f \|_{L^2_k}^2 
&= \int \BB_{\eps,1} (\partial_i f, \partial_i u) \,  \partial_i f \, \la x \ra^{2k}
- \frac12 \| \partial_i f \|_{L^2_k}^2
-\int \nabla(f \nabla(\partial_i V_\eps)) \,  \partial_i f \, \la x \ra^{2k}\\
&\quad
-\int \nabla(\partial_j G_\eps \nabla u) \,  \partial_i f \, \la x \ra^{2k} 
- N \int (\partial_i { \chi_R}) f \partial_i f \la x \ra^{2k}
+  N \la \chi_R f \ra  \int  (\partial_i \chi_1)  \, \partial_i f \, \la x \ra^{2k}\\
&=: A_1 + A_2 + A_3 + A_4 + A_5 + A_6.
\eal
$$
Arguing as in the first step of Lemma~\ref{lem:A&B}, for any $\delta>0$, we have
$$
\bal
A_1 
&\le
- (1 - \delta)\| \nabla (\partial_i f) \|_{L^2_k}^2 + C(\delta) \| \nabla^2( u -\kappa_{ f}) \|_{L^2}^2
\\
&\quad + \int \left\{ \varphi(x) + [ C(\delta)  + C \, N \, R^{ 1 - \ell}   ] \la x \ra^{2(\ell - k)} - N   \chi_R(x) \right\} |\partial_i f|^2 \la x \ra^{2k} .  
\eal
$$
We next compute
$$
\bal
A_3 
&:=  \int  f  \nabla( \partial_i V_\eps) \cdot \nabla (\partial_{i} f ) \la x \ra^{2k}
+ \int f \nabla( \partial_i V_\eps) \cdot \nabla \la x \ra^{2k} \, \partial_i f  \\
&\quad
\le \eps C(\delta) \| f \|_{L^2_{k}}^2 + \delta \| \nabla f \|_{L^2_{k}}^2
+ \delta \| \nabla^2 f \|_{L^2_{k}}^2,
\eal
$$
using that $\Delta V_\eps = -G_\eps  - (\eps/2) x \cdot \nabla V_\eps$ and Lemma~\ref{lem:LinStabBd}. We also have
$$
\bal
A_4
&=  \int  \partial_i G_\eps  \nabla u \cdot \nabla(\partial_i f) \la x \ra^{2k}
+ \int  \partial_i G_\eps  \nabla u \cdot \nabla \la x \ra^{2k} \, \partial_j f  \\
&\le  C(\delta) \| \nabla u \|_{L^2}^2 + \delta \| \nabla f \|_{L^2_{k-1/2}}^2
+ \delta \| \nabla^2 f \|_{L^2_k}^2 \\
&\le  C(\delta) \| \nabla (u-\kappa_f) \|_{L^2}^2 + C(\delta) \| f \|_{L^2_{\ell}}^2 + \delta \| \nabla f \|_{L^2_{k-1/2}}^2
+ \delta \| \nabla^2 f \|_{L^2_k}^2,
\eal
$$
and we easily get
$$
A_5 \le N\frac{C}{R} \int {\mathbf 1}_{R/2 \le |x| \le 2R} \, f^2 \la x \ra^{2k} 
+ N\frac{C}{R} \int {\mathbf 1}_{R/2 \le |x| \le 2R} \, |\partial_i f|^2 \la x \ra^{2k}.
$$
For the last term, we have
$$
\bal
A_6 
&\le N \la \chi_R f \ra \int (\partial_i \chi_1) \, \partial_i f \, \la x \ra^{2k} \\
&\le C N \, \| f \|_{L^2} \, \| \partial_i f \|_{L^2_k}
\leq C(\delta) N^2 \, \| f \|_{L^2}^2 + \delta \| \partial_i f \|_{L^2_k}^2.
\eal
$$

Finally, putting together all the above estimates, we obtain
\beqn\label{eq:Ystep2}
\bal
\frac12 \frac{d}{dt} \| \nabla f \|_{L^2_k}^2 
&\le - (1 - \delta )\| \nabla^2 f \|_{L^2_k}^2 
+ C(\delta) \| \nabla(u - \kappa_f) \|_{L^2}^2
+ C(\delta) \| \nabla^2(u - \kappa_f) \|_{L^2}^2 \\
&
+ \int \psi^0(x) \, |f|^2 \la x \ra^{2k}
+ \int \psi^1(x) \, |\nabla f|^2 \la x \ra^{2k} ,
\eal
\eeqn
where
$$
\psi^0(x) := \eps C(\delta) + N\frac{C}{R} {\mathbf 1}_{R/2 \le |x| \le 2R} 
+ C(\delta)\la x \ra^{2(\ell-k)}
+ C(\delta)N^2 \la x \ra^{-2k}
$$
and
$$
\psi^1(x) := \varphi(x)- \frac12 + \delta + N\frac{C}{R} {\mathbf 1}_{R/2 \le |x| \le 2R}  + [C(\delta) + CNR^{1-\ell}]\la x \ra^{2(\ell-k)} -  N { \chi_R}.
$$

\medskip
\noindent
{\it Step 3. Conclusion.}
We gather estimates \eqref{eq:Ystep1}, \eqref{eq:Ystep2} and \eqref{eq:Norm0*}, and we get
\beqn\label{eq:Norm1*}
\bal
&\frac12\frac{d}{dt} \| (f,u) \|^2_{Y_*} 
 \leq  
\int \Big\{ \varphi(x) + \eta_1 \psi^0(x)  - N { \chi_R}\\
&\qquad
+ [\eta_1 C(\delta) + \eta_1 C(1+N^2) + C(\delta)(1+\eta N^2) + C\frac{N}{R^{\ell-1}} ]\la x \ra^{2(\ell - k)}      \Big\} f^2 \la x \ra^{2k} \\
&\qquad
+\eta_1 \int \psi^1(x) \, | \nabla f |^2 \la x \ra^{2k} 
-(1-\delta)\| \nabla f \|_{L^2_k}^2    \\
&\qquad
- \eta\left( \frac12 - \delta   \right) \| u - \kappa_f \|_{L^2}^2 
-\eta_1(1-\delta) \| \nabla^2 f \|_{L^2_k}^2 \\
&\qquad
- \eta\left( \frac1\eps - C(\delta) - \frac{\eta_1}{\eta} C(\delta)   \right) \| \nabla (u - \kappa_{f}) \|_{L^2}^2
- \eta_1\left( \frac1\eps - C(\delta)   \right) \| \nabla^2 (u - \kappa_{f}) \|_{L^2}^2.
\eal
\eeqn
Now we conclude as in Step 3 of the proof of Lemma~\ref{lem:A&B}. We choose first $\delta \in(0,1)$ small enough, next $\eps \in (0,1)$ small enough and then   $\eta_1 = \eta = N^{-3}$ and $R = N$ in the above inequality. For any $a>-1/2$, we obtain
\beqn\label{eq:Norm1*bis}
\bal
\frac12\frac{d}{dt} \| (f,u) \|^2_{Y_*} 
& \leq  
\int \{ \varphi^1_N(x) - N { \chi_R} \} f^2 \la x \ra^{2k} 
+\eta_1 \int \left\{ \varphi^2_N(x) - N { \chi_R}    \right\} | \nabla f |^2 \la x \ra^{2k} \\
&\quad
+a\| \nabla f \|_{L^2_k}^2    
+ a \eta \| u - \kappa_f \|_{L^2}^2 
+ a \eta_1 \| \nabla^2 f \|_{L^2_k}^2 \\
&\quad
+a \eta_1 \| \nabla (u - \kappa_{f}) \|_{L^2}^2
+a \eta_1 \| \nabla^2 (u - \kappa_{f}) \|_{L^2}^2,
\eal
\eeqn
where 
$$
\varphi^1_N(x) := \varphi(x) + CN^{-3} + CN^{-1} {\mathbf 1}_{R/2 \le |x| \le 2R}  + (C + CN^{-1} + CN^{-3} + C N^{2-\ell})\la x \ra^{2(\ell-k)} + C N^{-1} \la x \ra^{-2k}
$$ 
and 
$$
\varphi^2_N(x) := \varphi(x) + C {\mathbf 1}_{R/2 \le |x| \le 2R}  + (C + C N^{2-\ell})\la x \ra^{2(\ell-k)} + C \la x \ra^{-2k}
$$
have the same asymptotic behaviour as $\varphi (x)$ when $|x|\to \infty$ and are decreasing as a function of $N$. Picking $N$ large enough such that
$$
\varphi^i_N(x) - N \chi_R(x) \le a, \quad \forall\, x \in \R^2,
$$
we deduce from \eqref{eq:Norm1*bis} that, for some constant $K>0$,
\beqn\label{eq:dissipY}
\frac12\frac{d}{dt} \| (f,u) \|^2_{Y_*} \le a \| (f,u) \|^2_{Y_*} - K \big( \| \nabla^2 f \|_{L^2_k}^2 + \| \nabla^2 (u-\kappa_f) \|_{L^2}^2  \big).
\eeqn
We first conclude that $\BB_\eps$ is $a$-hypo-dissipative in $Y$.
Moreover, using the interpolation inequality
$$
\| g \|_{H^1_k}^2 \le C  \| g \|_{H^2_k} \| g \|_{L^2_k},
$$
it follows from \eqref{eq:dissipY} that
$$
\frac12\frac{d}{dt} \| (f,u) \|^2_{Y_*} \le a \| (f,u) \|^2_{Y_*} - K \| (f,u) \|_{Y_*}^4 \, \| (f,u) \|_{X_*}^{-2}.
$$
By standard arguments, we get the estimate
$$
\| S_{\BB_\eps}(t) (f,u) \|_{Y} \le C\, t^{-1/2} \, e^{at} \| (f,u) \|_{X},
$$
concluding the proof.
\end{proof}

%
%

\medskip
\noindent
{\sl Proof of Proposition~\ref{prop:spectre}.} 
The domain $\dom(\Lambda_\eps)$ of the operator $\Lambda_\eps : \dom(\Lambda_\eps) \subset X \to X$ is given by 
$$
\dom(\Lambda_\eps) = H^2_k \cap L^2_{k,0} \cap L^2_{rad} \times H^2 \cap L^2_{rad}
$$ 
and we recall that $X = L^2_{k,0} \cap L^2_{rad} \times L^2_{rad}$ and 
$Y = H^1_k \cap L^2_{k,0} \cap L^2_{rad} \times H^1 \cap L^2_{rad}$ (see equations \eqref{eq:X} and \eqref{eq:Y}). We define a family of interpolation spaces
$$
X^\eta := H^{2\eta}_k \cap L^2_{k,0} \cap L^2_{rad} \times H^{2\eta} \cap L^2_{rad}, \quad \eta \in [0,1],
$$
so that $X^0 = X$, $X^1 = \dom(\Lambda_\eps)$ and $X^{1/2} = Y$. Thanks to classical interpolation results we have $Y = X^{1/2} \subset \dom( \Lambda_\eps^\eta)$ for any $\eta \in [0,1/2)$, see \cite{komatsu,LP,MS*}. Now we fix some $\eta \in (0,1/2)$ and we have $Y \subset \dom(\Lambda_\eps^\eta) \subset X$.

Recalling the results from Lemma \ref{lem:A&B}, Lemma \ref{lem:Abounded}, Lemma \ref{lem:T} and \eqref{eq:def-A} we have, for any $a > -1/2$,
$$
\bal
& S_{\BB_\eps}(t) : X \to X, \quad\text{with}\quad \| S_{\BB_\eps}(t) \|_{\BBB(X)} \le C \, e^{at},\\
& S_{\BB_\eps}(t) : Y \to Y, \quad\text{with}\quad \| S_{\BB_\eps}(t) \|_{\BBB(Y)} \le C \, e^{at},\\
& S_{\BB_\eps}(t) : X \to Y, \quad\text{with}\quad \| S_{\BB_\eps}(t) \|_{\BBB(X,Y)} \le C  \, t^{-1/2} \, e^{at},
\eal
$$
moreover $\AA \in \BBB(X) \cap \BBB(Y)$ and
$$
\bal
\AA S_{\BB_\eps}(t) : X \to  Y, \quad\text{with}\quad \| \AA S_{\BB_\eps}(t) \|_{\BBB(X, Y)} \le C  \, t^{-1/2} \, e^{at}.
\eal
$$
First,  from the previous estimates, we immediately obtain, for any $a > -1/2$, 
\beqn\label{eq:h1}
\forall\, \ell\ge 0, \quad
\| S_{\BB_\eps} * (\AA S_{\BB_\eps})^{(*\ell)} (t) \|_{\BBB(X)} \le C \, e^{a t}. 
\eeqn
Moreover, from \cite[Lemma 2.17]{GMM} there exists $n \in \N$ such that
$$
\| (\AA S_{\BB_\eps})^{*(n)} (t) \|_{\BBB(X,Y)} \le C \, e^{a t},
$$
which  together with the fact $S_{\BB_\eps}(t) : \dom( \Lambda_\eps^\eta) \to \dom( \Lambda_\eps^\eta)$ with $\| S_{\BB_\eps}(t) \|_{\BBB(\dom( \Lambda_\eps^\eta))} \le C \, e^{at}$ (by interpolation of the same results in $X$ and $Y$), yield
\beqn\label{eq:h2}
\| S_{\BB_\eps} * (\AA S_{\BB_\eps})^{(*n)} (t) \|_{\BBB(X, \dom(\Lambda_\eps^\eta))} \le C \, e^{a t}.
\eeqn
Recall that we have fixed some $a^*>-1/2$.
Gathering that last estimate  with \eqref{eq:h1}, we can apply \cite[Theorem 2.1]{MS*} which yields, for some $r^*>0$,
$$
\Sigma(\Lambda_\eps) \cap \CC_{a^*} \subset B(0,r^*) \quad\text{on}\quad X.
$$
From the previous estimates together with the fact that $\AA \in \BBB(X,L^2_{k+1} \times L^2_{k+1})$, we also obtain
\beqn\label{eq:h3}
\int_0^\infty \| (\AA S_{\BB_\eps})^{*(n+1)} (t) \|_{\BBB(X, \overline Y)} \, e^{-a t}\, dt \le C,
\eeqn
where $\overline Y :=  Y \, \cap \, (L^2_{k+1} \times L^2_{k+1}) \subset X$ with compact embedding. Hence, thanks to 
\eqref{eq:h1}-\eqref{eq:h2}-\eqref{eq:h3}, we are able to apply \cite[Theorem 3.1]{MS*} that implies
$$
\Sigma(\Lambda_\eps) \cap \CC_{a^*} \subset \Sigma_d(\Lambda_\eps) \quad\text{on}\quad X ,
$$
and that concludes the proof. \qed

\medskip 
 
\subsection{Localization of the spectrum for the linearized operator in a radially symmetric setting} 
\label{subsec:localizationSigma}
We recall that we consider a radially symmetric setting and we have already defined the space $X$ in \eqref{eq:X}. 
We establish in this subsection the following localization of the spectrum of $\Lambda_\eps$.

\begin{theo}\label{theo:S_Lambda}
Fix some $a^*>-1/2$.
There exists $\eps^*> 0$ such that in $X$ there holds 
$$
\Sigma(\Lambda_\eps) \cap \CC_{a^*} = \emptyset \quad\hbox{for any}\quad \eps \in (0,\eps^*).
$$
As a consequence, for any $ a > a^*$, there exists $C_a \ge 1$ such that 
$$
\| S_{\Lambda_\eps} (t) \|_{\BBB(X)}  \le C_a \, e^{at} \quad \forall \, t \ge 0, \,\, \forall \, \eps \in (0,\eps^*).
$$

\end{theo}

The difficulty is that $\Lambda_\eps$ is not a perturbation of some fixed operator $\Lambda$ and we cannot
apply directly the perturbation theory developed in \cite{MMcmp,Tristani}. However, we are able to identify the limit of $\RR_{\Lambda_\eps}$ as $\eps \to 0$ 
which is enough to conclude.

\smallskip
 
We introduce the notations 
\bean
&&A f := \Delta f + \nabla ({1 \over 2} \, x \, f - f  \, \nabla V_\eps),
\quad B u := - \nabla ( G_\eps \, \nabla u)
\\
&& C  u := \Delta u, \quad  D u :=  {1 \over 2} \, x \cdot \nabla u, 
\eean
so that the linearized equation writes 
\bear \label{eq:ode:eps}
&& \partial_t f  = A f + B u  ,  \qquad  \partial_t u    =   {1 \over \eps} ( C u + f) + Du.
\eear
The important point is that at a very formal level, the limit system (as $\eps \to 0$) is the 
linearized parabolic-elliptic system
\beqn \label{eq:ode:infty}
  \partial_t f  = A_0 f + B_0 u  ,  \qquad  Cu = - f,
\eeqn
where
$$
A_0 f := \Delta f + \nabla ({1 \over 2} \, x \, f - f  \, \nabla V), \quad
B_0 u := - \nabla ( G \, \nabla u),
$$
with $G$ and $V$ defined in \eqref{eq:KSprofile0}, which simplifies into a single equation 
\beqn \label{eq:1ode}
  \partial_t f  = (A_0  +  B_0(-C)^{-1}) f =: \Omega f .
\eeqn

Observe that the last equation is nothing but the linearized equation associated to parabolic-elliptic Keller-Segel equation 
which has been studied in \cite{MR2996772,CamposDolbeault2012,EM} and for which it has been proved therein that the associated
semigroup is exponentially stable in several weighted Lebesgue spaces. In the sequel we explain why the linearized parabolic-parabolic
system inherits that exponential stability  at least for $\eps > 0$ small enough.

We recall the following result which is an immediate consequence of \cite[Section 6.1]{CamposDolbeault2012} 
and \cite[Theorem 4.3]{EM}.

\begin{theo}\label{Theorem:theoSOmega}
There exists a constant $C$ such that  
$$
\forall\, t\ge 0, \;
\forall \, h \in L^2_{k,0}, \quad \| e^{\Omega t} h \|_{L^2_k} \le C \, e^{-t/2} \, \|  h \|_{L^2_k}. 
$$
As a consequence, there holds
$$
\RR_\Omega \in \HH(\CC_{-1/2};\BBB(X_1)) \quad \hbox{and then}\quad \Sigma(\Omega) \cap \CC_{-1/2} = \emptyset \ \hbox{in} \ X_1,
$$
where $X_1 = L^2_{rad} \cap L^2_{k,0}$ is defined in \eqref{eq:X}.
\end{theo}

\begin{rem*}
The optimal decay rate $-1$ stated in \cite{CamposDolbeault2012,EM} becomes $-1/2$ here, because of the different normalization choice in the definition of the rescaled functions $g$ and $v$ in \eqref{eq:KSresc} and \eqref{eq:KSresc2eme}.
\end{rem*}

%
%
%
%
\smallskip
In order to formalize the link between the linearized parabolic-parabolic equation and the linearized parabolic-elliptic equation, we write the linearized parabolic-parabolic system \eqref{eq:ode:eps} into the matrix form
$$
{d \over dt} \left( \begin{array}{c} f  \\  u \end{array} \right) = \Lambda_\eps \left( \begin{array}{c} f  \\  u \end{array} \right),
\quad 
\Lambda_\eps := \left( \begin{array}{cc} A & B  \\  \eps^{-1}I  & \eps^{-1} C + D\end{array} \right).
$$
For the analysis of the spectrum of $\Lambda_\eps$, for any $z \in \C$, we have
$$
\Lambda_\eps - z = \left( \begin{array}{cc} a & b  \\  c  & d \end{array} \right),
$$
with 
$$
a :=A(z) = A-z, \quad b := B, \quad  c:= \eps^{-1}I, \quad d := \eps^{-1}C + D(z), \quad D(z) = D-z.
$$
One can readily verify that for $z \in \C$ such that $D-z$ and its Schur's complement 
$$
s_\eps = s_\eps(z) := a -bd^{-1}c =  A(z) - B(C+\eps D(z))^{-1} 
$$
are invertible, the resolvent of $\Lambda_\eps$ is given by  
\bean 
\RR_{\Lambda_\eps}(z) 
=  (\Lambda_\eps - z)^{-1} 
= \left( \begin{array}{cc} s_\eps^{-1} & -s_\eps^{-1} b d^{-1}  \\   -d^{-1} c s_\eps^{-1}  &  d^{-1} + d^{-1} c s_\eps^{-1} b d^{-1}  \end{array} \right)
=: \left( \begin{array}{cc} \RR^{\Lambda_\eps}_{11}(z) & \RR^{\Lambda_\eps}_{12}(z)  \\   \RR^{\Lambda_\eps}_{21}(z)  &  \RR^{\Lambda_\eps}_{22}(z)  \end{array} \right). 
\eean
Then, at least formally, we see that  
\beqn\label{eq:RLepstoU0}
\RR_{\Lambda_\eps}(z) 
\mathop{\longrightarrow}_{\eps\to0} \left( \begin{array}{cc} \RR_\Omega(z) & 0  \\  -C^{-1} \RR_\Omega(z)  & 0  \end{array} \right).
\eeqn
Indeed, on the one hand, we have  
  $$
\bal
\RR^{\Lambda_\eps}_{11} &= s_\eps^{-1} 
= \{ A(z) - B(\eps^{-1} C + D(z))^{-1} \eps^{-1} I \}^{-1}\\
&= \{ A- z - B( C + \eps D-\eps z)^{-1} \}^{-1} \xrightarrow[\eps\to0]{} (A_0-B_0 C^{-1}-z)^{-1}
= \RR_{\Omega}(z).
\eal
$$
and 
$$
\bal
\RR^{\Lambda_\eps}_{21}(z) &= -d^{-1} c s_\eps^{-1}= -(\eps^{-1} C + D(z))^{-1} \eps^{-1} I\{ A(z) - B( C + \eps D(z))^{-1} \}^{-1} \\
&= -( C + \eps D-\eps z)^{-1}\{ A- z - B( C + \eps(D-z))^{-1} \}^{-1}\\
& \xrightarrow[\eps\to0]{} -C^{-1}(A_0-B_0 C^{-1}-z)^{-1}
= -C^{-1}\RR_{\Omega}(z).
\eal
$$
In the same way, we have 
$$
\bal
\RR^{\Lambda_\eps}_{12}(z)
&= -s_\eps^{-1} b d^{-1} =- \eps \{ A- z - B( C + \eps D-\eps z)^{-1} \}^{-1} B (C+\eps D - \eps z)^{-1} 
\eal
$$
as well as 
$$
\bal
&\RR^{\Lambda_\eps}_{22} (z)
= d^{-1} + d^{-1} c s_\eps^{-1} b d^{-1} 
&= \eps(C+\eps D(z))^{-1} + 
\eps(C+\eps D(z))^{-1}\{ A(z) - B( C + \eps D(z))^{-1} \}^{-1} B (C+\eps D(z))^{-1},
\eal
$$
and then both last terms vanish in the limit $\eps \to 0$.

\medskip

In fact, we will not try to prove that convergence  \eqref{eq:RLepstoU0} rigorously holds, but we will just prove the following result. 
We define, for a given $\rho >0$ and some fixed $a^* >-1/2$, 
$$
\OO_{\rho} := \CC_{a^*} \cap B(0,\rho).
$$

\begin{prop}\label{lem:RepsInHH} For any $\rho > 0$, there exists $\eps^*_{\rho}> 0$ such that in $X$ there holds 
$$	
\RR_{\Lambda_\eps} \in \HH(\OO_{\rho};\BBB(X)) \quad\hbox{for any}\quad \eps \in (0,\eps^*_{\rho}).
$$

\end{prop}

 Before proving Proposition~\ref{lem:RepsInHH}, we establish some estimates on the terms involved in $\RR_{\Lambda_\eps}$. 
 
\begin{lem}\label{lem:tildes} Define 
$$
\tilde s(z) :=   B (C + \eps D(z))^{-1} (- D(z)) \, C^{-1}.
$$ 
For any $\rho > 0$, there exists  $\eps^*_\rho >0$ such that
$$
\sup_{z \in \OO_{\rho}} \sup_{\eps \in (0,\eps^*_\rho)} \| \tilde s (z) \|_{\BBB(X_1)} \le C.
$$

\end{lem}

\noindent
{\sl Proof of Lemma~\ref{lem:tildes}.} On the one hand, from Lemma~\ref{lem:kappa*f} and Lemma~\ref{lem:K*f} we have 
$$
- D(z) \, C^{-1} : L^2_{k,0} \cap  L^2_{rad} \to L^2_{rad}
$$
is bounded uniformly for $z \in \OO_\rho$. More precisely, for $f \in L^2_{k,0}$ we write 
$$
f = f_0 + f_1 + f_2; \quad f_i = \lambda_i F_i  , \; i=1,2,
$$
where we define the coefficients $\lambda_i \in \R$ by
$$
\lambda_1 =   \int_0^\infty f(r) \, r^2 \, r \, dr, \quad
\lambda_2   =   \int_0^\infty f(r) \, r^4 \, r \, dr,
$$
and the functions $F_i$ by
$$
F_1(r) = \left( -\frac18 r^4 + \frac54 r^2 - \frac32   \right) e^{-r^2/2}, \quad
F_2(r) = \left( \frac{1}{64}r^4 - \frac18 r^2 + \frac18  \right) e^{-r^2/2},
$$
so that it holds
$$
\int_0^\infty F_1 (r) \, (1, r^2, r^4) \, r \, dr = (0,1,0), \quad
\int_0^\infty F_2 (r) \, (1, r^2, r^4) \, r \, dr = (0,0,1),
$$
and hence $f_0 \in L^2_{k,5}$. 

We may then solve the equation 
$$
Cu = \Delta u = f, \quad u \,\, \hbox{radially symmetric},\quad  u'(0) =   u'(\infty) = 0 ,
$$
by writing 
$$
u = u_0 +  u_1 +   u_2, \quad \Delta u_i = f_i, \quad    u_i'(0) =  u_i'(\infty) = 0,
$$
where $u_i$, $i = 1,2$, is defined by the relation 
$$
r u'_i(r) := \int_0^r \sigma f_i(\sigma) \, d\sigma, 
$$ 
so that $\langle r \rangle |u'_i| \le C$, $|u| \le C \log (1 + \langle r \rangle)$, and 
$u_0 \in L^2_{rad} \cap L^2_4$, $\nabla u_0 \in L^2_5$ is the unique solution to the above Poisson equation as given by Lemma \ref{lem:kappa*f} and Lemma \ref{lem:K*f}. 
As a consequence, $g_0 := D(z) u_0  \in L^2_{4,1}$ and $g_i :=  r u'_i + z u_i $ satisfy the estimates $g_i \, e^{-r/2} \in L^\infty$ for $i=1,2$.
Thanks to Lemmas \ref{lem:du=f} and \ref{lem:du=fbis} and using the notation of Appendix~\ref{sec:B},  we have $v_i := L^{-1}_\eps g_i$ which satisfy $\| v_0 \|_{\dot H^1 \cap \dot H^2} \le C \, \| g_0 \|_{L^2_{4,1}}$, while for $i=1,2$, $v_i$ satisfy the estimates 
$$
\| v_i \, e^{-(1+\eps|z|)r} \|_{L^\infty} +  \| v'_i \, e^{-(1+\eps|z|)r} \|_{L^\infty} \le C \| g_i \, e^{-r/2} \|_{L^\infty}.
$$
Finally, we solve the equation $w_i := B v_i$, which means
$$
w_i = G_\eps \Delta v_i + \nabla G_\eps \cdot \nabla v_i 
= G_\eps (v''_i + \frac1r \, v'_i) + r \, G'_\eps v'_i,
$$
and the previous estimates together with the bound \eqref{eq:Gepsbound} yield $w = w_0 + w_1 + w_2 \in L^2_{k,0} \cap L^2_{rad} = X_1$.
%
\qed

 \begin{lem}\label{lem:d-1} With the above notation, for any  $\rho > 0$, there exists $C_{\rho}$ such that  
\beqn\label{eq:Leps1}
\sup_{z \in  \OO_\rho} \| d^{-1}(z) \|_{\BBB(X_2)} \le C_{\rho}  .
\eeqn
\end{lem}

\noindent{\sl Proof of Lemma~\ref{lem:d-1}.} 
Consider the equation 
\beqn\label{eq:d(z)u=f}
 d(z) v = \eps^{-1} \Delta v  + {1 \over 2} x \cdot \nabla v - z \, v = u,
\eeqn
for $z \in  \OO_\rho$. 
 Multiplying the equation by $\bar v$ and the conjugated equation by $v$, we find 
\bear\label{eq:d(z)u=f1}
 {1 \over \eps} \int |\nabla v|^2 + \bigl( {1 \over 2} + \Re e z \bigr) \int |v|^2 
 &=&  -  {1 \over 2} \int v \bar u    -  {1 \over 2} \int \bar v u  
 \\ \nonumber
 &\le&    \| u \|_{L^2} \, \| v \|_{L^2} ,
\eear
and then 
$$
\bigl( {1 \over 2} + \Re e z \bigr) \, \| v\|_{L^2} \le \| u\|_{L^2},
$$
which is nothing but \eqref{eq:Leps1}.
\qed


\begin{lem}\label{lem:bd-1to0} With the above notations, for any $\rho > 0$, there exists $C_{k,a,\rho}>0$ such that  
 \beqn\label{eq:bd-1} 
\sup_{z \in  \OO_\rho} \| b d^{-1} (z) \|_{\BBB(X_2,X_1)} \le C_{k,a,\rho} \, \sqrt{\eps} \to 0. 
\eeqn
\end{lem}

\noindent{\sl Proof of Lemma~\ref{lem:bd-1to0}.} 
Consider the equation \eqref{eq:d(z)u=f} again. Coming back to \eqref{eq:d(z)u=f1}, we have  
\beqn\label{eq:thetaTO0}
 \| \nabla v \|_{L^2} \le \sqrt{ { \eps \over {1 \over 2} + \Re e z } } \| u\|_{L^2}.
\eeqn
Next, multiplying the equation \eqref{eq:d(z)u=f} by $x \cdot \nabla \bar v$ and the conjugated equation by $x \cdot \nabla v$, we find 
$$
\int |x \cdot \nabla v|^2 = \int (u-z v) (x \cdot \nabla \bar v) + (\bar u-\bar z \bar v) (x \cdot \nabla   v),
$$
which in turn implies 
$$
\| x \cdot \nabla v \|_{L^2}  \le C_{a,r}.
 $$
Coming back to \eqref{eq:d(z)u=f} and together with \eqref{eq:thetaTO0}, we have proved 
\beqn\label{eq:vTO0}
\| \Delta v \|_{L^2} +  \| \nabla v \|_{L^2} \le C_{a,r}  \sqrt{  \eps } \| u\|_{L^2}.
\eeqn
We then immediately conclude to \eqref{eq:bd-1}. \qed

 \begin{lem}\label{lem:cd-1} With the above notation, for any $\rho > 0$, there exists $C_{k,\rho}>0$ such that  
 \beqn\label{eq:cd-1} 
\sup_{z \in  \OO_\rho} \| c d^{-1} \|_{\BBB(X_1,X_2)} \le C_{k,\rho}. 
\eeqn
\end{lem}

\medskip
\noindent{\sl Proof of Lemma~\ref{lem:cd-1}. } We just have to use appendix~\ref{sec:B}. \qed

\medskip
\noindent{\sl Proof of Proposition~\ref{lem:RepsInHH}.} 
We split the proof into four steps.

\smallskip\noindent
{\it Step 1.} We prove that 
 \beqn\label{eq:R11}
\forall \, \eps \in (0,\eps^*_\rho) \quad \RR^{\Lambda_\eps}_{11} = s_\eps^{-1} \in \HH(\OO_\rho;\BBB(X_1)).
\eeqn
We write 
$$
s_\eps(z) =  A(z) - B C^{-1} - [ B (C + \eps D(z))^{-1} - B C^{-1}] =: s_0(z) - \eps \tilde s(z) ,
$$
and then
$$
s_\eps(z) - \Omega(z) = [s_0(z) - \Omega(z)] - \eps \tilde s(z),
$$
with 
$$
s_0(z) =A(z) - B C^{-1}, \quad \Omega(z) = \Omega - z = A_0(z) - B_0 C^{-1}, \quad \tilde s(z) :=   B (C + \eps D(z))^{-1} (- D(z)) \, C^{-1}.
$$
We remark that 
$$
 s_0(z) - \Omega(z)= s_0 - \Omega = -\nabla( \cdot \nabla (V_\eps - V)) - \nabla ((G_\eps-G)\nabla (\Delta^{-1} \cdot))
$$
does not depend on $z$, and thanks to Corollary~\ref{cor:cvgceGVeps}, we easily get
$$
\| s_0 - \Omega \|_{\BBB(Y_1,  X_1)} \le \eta(\eps) \quad\text{with}\quad \eta(\eps) \to 0 \; \text{as} \; \eps \to 0,
$$
where we recall that $Y_1 = H^1_k \cap L^2_{k,0} \cap L^2_{rad}$ is defined in \eqref{eq:Y}.
Moreover, using Lemma~\ref{lem:tildes}, we get 
$$
\sup_{z \in \OO_\rho} \sup_{\eps \in (0,\eps^*_\rho)} \| \tilde s \|_{\BBB(X_1)} \le C,
$$
from which we deduce, for any $z \in \OO_\rho$ and $\eps \in (0,\eps^*_\rho)$, the bound
$$
\| s_\eps(z) - \Omega(z) \|_{\BBB(Y_1,  X_1)} \le \eta(\eps) + C \eps.
$$
%
%
%
Then, arguing as is \cite[Lemma 2.16]{Tristani}, the operator
$$
\TT_\eps(z) := (-1)^n (s_\eps - \Omega) \RR_{\Omega}(z) \, (\AA_{11} \RR^{\BB_\eps}_{11}(z) )^n
$$
satisfies
$$
\| \TT_\eps (z) \|_{\BBB(X_1)} \le \eta'(\eps) \quad \forall\, z \in \OO_\rho, 
\quad
\eta'(\eps) \to 0 \;\text{ as }\; \eps \to 0, 
$$
where
$$
\RR_{\BB_\eps} =: \left( \begin{array}{cc} \RR^{\BB_\eps}_{11} & \RR^{\BB_\eps}_{12}  \\  \RR^{\BB_\eps}_{21}  &  \RR^{\BB_\eps}_{22}  \end{array} \right),
\quad
\AA =: \left( \begin{array}{cc} \AA_{11} & 0  \\   0  &  0 \end{array} \right), 
$$
and the integer $n$ is defined in the proof of Proposition~\ref{prop:spectre}. As a consequence, 
the operators $I + \TT_\eps(z)$ and $s_\eps(z)$ are invertible for any $z \in \OO_\rho$, and furthermore
$$
\RR_{11}^{\Lambda_\eps} (z) = s_\eps(z)^{-1} = \UU_\eps(z) (1 + \TT_\eps(z))^{-1},
$$
where
$$
\UU_\eps (z) = \sum_{j=0}^{n-1}  (-1)^{j} \RR^{\BB_\eps}_{11}(z) (\AA_{11} \RR^{\BB_\eps}_{11}(z))^{j} +  (-1)^{n} \RR_{\Omega}(z) (\AA_{11} \RR^{\BB_\eps}_{11}(z))^{n}.
$$
We immediately conclude to  \eqref{eq:R11}, because $\Sigma(\Omega) \cap \CC_{a} = \emptyset$ on $X_1$, and $\| R_\Omega (z) \|_{\BBB(X_1)} \le C$ for any $ z \in \OO_{\rho}$ from Theorem \ref{Theorem:theoSOmega}.

\smallskip\noindent
{\it Step 2.} We have 
 \beqn\label{eq:Leps2}
\forall \, \eps \in (0,\eps^*_\rho) \quad \RR^{\Lambda_\eps}_{12}:= - s_\eps^{-1} b d^{-1}   \in \HH(\OO_\rho;\BBB(X_2,X_1)),
\eeqn
as an immediate consequence of Lemmas~\ref{lem:tildes} and \ref{lem:bd-1to0}. 

\smallskip\noindent
{\it Step 3.} We also have  
 \beqn\label{eq:Leps3}
\forall \, \eps \in (0,\eps^*_\rho) \quad \RR^{\Lambda_\eps}_{21} := - d^{-1} c s_\eps^{-1} \in \HH(\OO_\rho;\BBB(X_1,X_2)),
\eeqn
as a  consequence of Lemmas~\ref{lem:tildes} and \ref{lem:cd-1}. 

\smallskip\noindent
{\it Step 4. } We finally have 
 \beqn\label{eq:Leps4}
\forall \, \eps \in (0,\eps^*_\rho) \quad \RR^{\Lambda_\eps}_{22}:= d^{-1} + d^{-1} c s_\eps^{-1} b d^{-1}  \in  \HH(\OO_\rho;\BBB(X_2)),
\eeqn
as an immediate consequence of Step 1 together with Lemmas \ref{lem:bd-1to0} and \ref{lem:cd-1}. \qed

\medskip
\noindent{\sl Proof of Theorem~\ref{theo:S_Lambda}.} 
The proof is a consequence of Proposition~\ref{prop:spectre}, Proposition~\ref{lem:RepsInHH} and Theorem~\ref{Theorem:theoSOmega} together with \cite[Theorem 2.1]{MS*}.
\qed

\section{Nonlinear exponential stability of self-similar solutions}
\label{sec:nonlinear}

\subsection{Linear stability in higher-order norms}

Define the space
\beqn\label{eq:Z}
 Z := Z_1 \times  Z_2, \quad Z_1 := H^1_k \cap L^2_{k,0} \cap L^2_{rad} ,
\quad  Z_2 := H^2 \cap L^2_{rad},
\eeqn
associated to the norm 
\beqn\label{eq:normeZ}
 \| (f,u) \|_{Z}^2 := \| (f,u) \|_{Y}^2+ \| \nabla^2 u \|_{L^2}^2.
\eeqn

We shall prove that the same linear stability estimate in $X$ established in Theorem~\ref{theo:S_Lambda} also holds in $Z$, as stated in the following result.
\begin{prop}\label{prop:stabZ}
Let $a^* > -1/2$ be fixed.
There exists $\eps^* >0$ such that there holds in $Z$
$$
\Sigma(\Lambda_\eps) \cap \CC_{a^*} = \emptyset, \quad \forall\, \eps \in (0,\eps^*).
$$
As a consequence, we have
$$
\| S_{\Lambda_\eps}(t) \|_{\BBB(Z)} \le C e^{at }, \quad \forall\, t\ge0,\; \forall\, \eps \in (0,\eps^*), \,\,  \forall \, a > a^*.
$$

\end{prop}

Before proving Proposition \ref{prop:stabZ}, we state and prove an auxiliary technical result.

\begin{lem}\label{lem:BBhypoZ}
(1) $\AA \in \BBB(Z)$.

(2) There exist $N,R$ large enough such that the operator $\BB_\eps$ is $a$-hypo-dissipative in $Z$ for any $a > -1/2$, i.e.
$$
\| S_{\BB_\eps}(t) \|_{\BBB(Z)} \le C e^{at} , \quad \forall\, t \ge 0.
$$
Moreover, we have the following estimate
$$
\| S_{\BB_\eps}(t) \|_{\BBB(X,Z)} \le C\, t^{-1} \, e^{at} , \quad \forall\, t \ge 0, \,\, \forall \, a > -1/2.
$$

\end{lem}

\begin{proof}
Point (1) is straightforward from \eqref{eq:def-A} and we omit the proof. For proving point (2), we consider a solution $(f,u)$ to the equation $\partial_t (f,u) = \BB_\eps (f,u)$. First of all, observe that the norm $\| \cdot \|_Z$ is equivalent to
\beqn\label{eq:normeZ*}
 \| (f,u) \|_{Z_*}^2 := \| (f,u) \|_{Y_*}^2+ \eta_2 \| \nabla^2 (u - \kappa_f) \|_{L^2}^2,
\eeqn
for any $\eta_2 >0$.
We write
$$
\bal
\partial_t (\partial_{ij} u) 
&= \frac1\eps \Delta(\partial_{ij} (u - \kappa_f ))   
+ \partial_{ij}u
+ \frac12 x \cdot \nabla( \partial_{ij}u ) ,
\eal
$$
and then we compute 
$$
\bal
\frac12\frac{d}{dt} \int |\partial_{ij} (u - \kappa_{f})|^2
&= \int \partial_{ij} (u - \kappa_{ f})  \partial_t \{ \partial_{ij} (u - \kappa_f)\} \\
&= \int \partial_{ij} (u - \kappa_{ f}) \frac1\eps \Delta(\partial_{ij} (u - \kappa_f ))  
+\int \partial_{ij} (u - \kappa_{ f}) \partial_{ij}u  \\
&\quad
+\int \partial_{ij} (u - \kappa_{ f}) \frac12 x \cdot \nabla( \partial_{ij}(u- \kappa_f ))  
+\int \partial_{ij} (u - \kappa_{ f}) \frac12 x \cdot \nabla( \partial_{ij} \kappa_f )  \\
&\quad
-\int \partial_{ij} (u - \kappa_{ f}) \partial_{ij} \kappa * (\partial_t f) \\
&=: B_1 + \cdots + B_5. 
\eal
$$
We estimate each term separately.
We easily obtain
$$
B_1 = - \frac1\eps \| \nabla \{ \partial_{ij} (u-\kappa_f) \} \|_{L^2}^2,
$$
moreover
$$
B_2 \le C \| \nabla^2 (u - \kappa_f) \|_{L^2}^2 + C\| \nabla^2 u \|_{L^2}^2
\le C \| \nabla^2 (u - \kappa_f) \|_{L^2}^2 + C\| f \|_{L^2}^2,
$$
by integration by parts $B_3 \le 0$ and also
$$
B_4 \le C \| \nabla^2 (u - \kappa_f) \|_{L^2}^2 + C \| \nabla^3 \kappa * f \|_{L^2_1}^2
 \le C \| \nabla^2 (u - \kappa_f) \|_{L^2}^2 + C \| \nabla f \|_{L^2_1}^2.
$$
For the last term, we get
$$
\bal
B_5 
&= - \int \partial_{ij} (u - \kappa_{ f}) \partial_{ij} \kappa *\left\{ \Delta f + \frac12 \nabla (xf) - \nabla(f \nabla V_\eps) - \nabla (G_\eps \nabla u)   - N \chi_R[f] \right\} \\
&=: B_{51} + \cdots  + B_{55}.
\eal
$$
We have
$$
B_{51} \le C(\delta) \| \nabla^2 (u-\kappa_f) \|_{L^2}^2 + \delta \| \nabla^2 f \|_{L^2}^2
$$ 
and
$$
B_{52} \le C \| \nabla^2 (u-\kappa_f) \|_{L^2}^2 + C \| \nabla^2 \kappa * \nabla(xf) \|_{L^2}^2
\le C \| \nabla^2 (u-\kappa_f) \|_{L^2}^2 + C \| \nabla f \|_{L^2_1}^2 + C \| f \|_{L^2}^2.
$$
Moreover, using that $\nabla V_\eps, \Delta V_\eps \in L^\infty$, we get
$$
\bal
B_{53} 
&\le C \| \nabla^2 (u-\kappa_f) \|_{L^2}^2 + C\| \nabla^2 \kappa * \nabla(f \nabla V_\eps) \|_{L^2}^2 \\
&\le C \| \nabla^2 (u-\kappa_f) \|_{L^2}^2 + C\| \nabla(f \nabla V_\eps) \|_{L^2}^2\\
&\le C \| \nabla^2 (u-\kappa_f) \|_{L^2}^2 + C\| \nabla f  \|_{L^2}^2 +  C \| f \|_{L^2}^2,
\eal
$$
and arguing as above with $G_\eps, \nabla G_\eps \in L^\infty$, we also obtain
$$
\bal
B_{54} 
&\le C \| \nabla^2 (u-\kappa_f) \|_{L^2}^2 + C\| \nabla^2 \kappa * \nabla( G_\eps \nabla u) \|_{L^2}^2 \\
&\le C \| \nabla^2 (u-\kappa_f) \|_{L^2}^2 + C\| \nabla (G_\eps \nabla u) \|_{L^2}^2\\
&\le C \| \nabla^2 (u-\kappa_f) \|_{L^2}^2 + C\| \nabla^2 u  \|_{L^2}^2+ C\| \nabla u  \|_{L^2}^2 \\
&\le C \| \nabla^2 (u-\kappa_f) \|_{L^2}^2 + C\|  f  \|_{L^2_{\ell}}^2,
\eal
$$
where we recall that $\ell \in (3,k)$ is fixed in Lemma \ref{lem:A&B}. For the last term, we easily obtain
$$
\bal
B_{55} \le C \| \nabla^2 (u-\kappa_f) \|_{L^2}^2 + C\| N \chi_R[f] \|_{L^2}^2
 \le C \| \nabla^2 (u-\kappa_f) \|_{L^2}^2 + CN^2 \| f \|_{L^2}^2.
\eal
$$

All the above estimates yield
\beqn\label{eq:step6}
\bal
\frac12\frac{d}{dt} \| \nabla^2 (u - \kappa_{f}) \|_{L^2}^2
&\le - \frac1\eps  \| \nabla^3 (u - \kappa_{\tilde f}) \|_{L^2}^2
+ C(\delta) \| \nabla^2 (u - \kappa_{\tilde f}) \|_{L^2}^2 \\
&\quad
+ C \| f \|_{L^2_{\ell}}^2
+C N^2 \| f \|_{L^2}^2
+ C \| \nabla f \|_{L^2_1}^2
+ \delta \| \nabla^2 f \|_{L^2}^2.
\eal
\eeqn

Putting together that last equation with \eqref{eq:Norm1*}, we get
$$
\bal
&\frac12\frac{d}{dt} \| (f,u) \|^2_{Z_*} \\
& \leq  
\int \Big\{ \varphi(x) + \eta_1 \eps C(\delta) + \eta_1  C\frac{N}{R} {\mathbf 1}_{R/2 \le |x| \le 2R} + \eta_1 C(\delta) N^2 \la x \ra^{-2k}  + \eta_2 C N^2 \la x \ra^{-2k}\\ 
&\qquad\quad
+ [\eta_1 C(\delta) + \eta_1 C(1+N^2) + C(\delta)(1+\eta N^2) + C\frac{N}{R^{\ell-1}}  + \eta_2 C]\la x \ra^{2(\ell - k)}  - N { \chi_R}    \Big\} f^2 \la x \ra^{2k} \\
&\quad
+\eta_1 \int \Big\{ \varphi(x) - \frac12 + \delta  + C\frac{N}{R} {\mathbf 1}_{R/2 \le |x| \le 2R}+ [C(\delta) +C\frac{N}{R^{\ell-1}} ]\la x \ra^{2(\ell-k)} \\
&\qquad\qquad\quad
 + C \la x \ra^{-2k} + \frac{\eta_2}{\eta_1} C \la x \ra^{2(1-k)} - N { \chi_R}    \Big\} | \nabla f |^2 \la x \ra^{2k} \\
&\quad
-(1-\delta)\| \nabla f \|_{L^2_k}^2    
- \eta\left( \frac12 - \delta   \right) \| u - \kappa_f \|_{L^2}^2 
-\eta_1 \left( 1-\delta - \frac{\eta_2}{\eta_1} \delta \right) \| \nabla^2 f \|_{L^2_k}^2 \\
&\quad
- \eta\left( \frac1\eps - C(\delta) - \frac{\eta_1}{\eta} C(\delta)   \right) \| \nabla (u - \kappa_{f}) \|_{L^2}^2
- \eta_1\left( \frac1\eps - C(\delta)  - \frac{\eta_2}{\eta_1}C(\delta)  \right) \| \nabla^2 (u - \kappa_{f}) \|_{L^2}^2 \\
&\quad
-\frac{1}{\eps} \| \nabla^3 (u - \kappa_{\tilde f}) \|_{L^2}^2.
\eal
$$
We can now conclude exactly as in the proof of Lemma \ref{lem:A&B} and Lemma \ref{lem:T}, and we obtain that for any $a>-1/2$,
\beqn\label{eq:Z*}
\frac12\frac{d}{dt} \| (f,u) \|^2_{Z_*} \le a \| (f,u) \|^2_{Z_*} - K ( \| \nabla^2 f \|_{L^2_k}^2 + \| \nabla^3 (u-\kappa_f) \|_{L^2}^2),
\eeqn
for some constant $K>0$, from which $\BB_\eps$ is $a$-hypo-dissipative in $Z$.

From \eqref{eq:Z*} and the interpolation inequalities
$$
\| f \|_{H^1_k}^2 \le C \| f \|_{H^2_k} \, \| f \|_{L^2_k}, \quad
\| u \|_{H^2}^{2} \le C \| u \|_{H^3}^{4/3} \, \| u \|_{L^2}^{2/3}
$$
it follows that
$$
\frac12\frac{d}{dt} \| (f,u) \|^2_{Z_*} 
\le a \| (f,u) \|^2_{Z_*} - K  \| (f,u) \|_{X_*}^{-1} \, \| (f,u) \|_{Z_*}^{3} .
$$
We obtain by standard arguments
$$
\| S_{\BB_\eps}(t) (f,u) \|_{Z} \le C \, t^{-1} \, e^{at} \, \| (f,u) \|_{X},
$$
which concludes the proof.
\end{proof}

\begin{proof}[Proof of Proposition \ref{prop:stabZ}]

From Lemma \ref{lem:A&B}, Lemma \ref{lem:Abounded}, Lemma \ref{lem:BBhypoZ} and \cite[Lemma 2.17]{GMM}, it follows that there is $n \in \N$ such that
$$
\| (\AA S_{\BB_\eps})^{*n} (t) \|_{\BBB(X,Z)} \le C \, e^{at}.
$$
Then the proof of the linear stability result in $Z$ is a consequence of last estimate, Lemma~\ref{lem:BBhypoZ}, Theorem~\ref{theo:S_Lambda} and the ``extension theorem'' \cite[Theorem 1.1]{MM*}.
\end{proof}

\subsection{Dissipative norm}

We define the new norm
$$
\Nt (f,u) \Nt_{Z}^2 := \eta \| (f,u) \|_Z^2 + \int_0^\infty \| S_{\Lambda_\eps}(\tau) (f,u) \|_Z^2 \, d\tau
$$
for some $\eta >0$. 
Thanks to Proposition \ref{prop:stabZ}, the norm $\Nt \cdot \Nt_Z$ is equivalent to $\| \cdot \|_Z$ for any $\eta >0$.
Moreover, considering a solution $(f,u)$ to the linearized equation $\partial_t (f,u) = \Lambda_\eps (f,u)$, we obtain from Proposition~\ref{prop:stabZ}, Lemma~\ref{lem:BBhypoZ} and arguing as in \cite{GMM,EM}, that
\beqn\label{eq:Zdissipative}
\bal
\frac{d}{dt} \Nt (f,u) \Nt_{Z}^2 \le - K \Nt (f,u) \Nt_{Z}^2 - K \{ \| \nabla^2 f\|_{L^2_k}^2 + \|\nabla^3 u  \|_{L^2}^2 \}
=: - K \| (f,u) \|_{\widetilde Z}^2,
\eal
\eeqn
for $\eta>0$ small enough and some constant $K>0$.

\subsection{The nonlinear problem : Proof of theorem \ref{theo:stab}}
We focus now on the nonlinear parabolic-parabolic Keller-Segel system \eqref{eq:KSresc}-\eqref{eq:KSresc2eme} in self-similar variables and we prove Theorem~\ref{theo:stab}. 
Consider a solution $(g,v)$ to \eqref{eq:KSresc}-\eqref{eq:KSresc2eme} and define $f := g-G_\eps$ and $u := v - V_\eps$, which satisfy
\beqn\label{eq:nl}
\bal
\partial_t f &= \Lambda_{\eps,1}(f,u) - \nabla \cdot (f \nabla u)
\\
\partial_t u &= \Lambda_{\eps,2} (f,u),
\eal
\eeqn
together with the initial condition $(f,u)_{|t=0} = (f_0,u_0) := (g_0,v_0) - (G_\eps,V_\eps) \in Z$.

\medskip

We split the proof into three parts.

\subsubsection{A priori estimate}

\begin{lem}\label{lem:stabilite}
The solution $(f_t,u_t)$ to \eqref{eq:nl} satisfies, at least formally, the following differential inequality
\beqn\label{eq:dZdt}
\frac{d}{dt} \Nt (f,u) \Nt_{Z}^2 
\le ( -K + C \Nt (f,u) \Nt_{Z}  ) \, \| (f,u) \|_{ \widetilde Z}^2,
\eeqn
where $ \| \cdot \|_{\widetilde Z} $ is defined in \eqref{eq:Zdissipative}. 

\end{lem}

\begin{proof}[Proof of Lemma \ref{lem:stabilite}]
Because 
$$
\| (f,u) \|_{Z}^2 = \| f \|_{H^1_k}^2 + \| u  \|_{H^2}^2,
$$
and denoting $Q(f,u) = ( - \nabla\cdot(f \nabla u), 0 ) $, we obtain from \eqref{eq:nl} that,
$$
\bal
\frac12 \frac{d}{dt} \Nt (f , u ) \Nt_{Z}^2
&=   \eta \la (f, u) , \Lambda_\eps (f,u) \ra
+  \int_0^\infty \la  S(\tau) (f , u), S(\tau)\Lambda_\eps (f,u) \ra  \, d\tau\\
&+
\eta \la (f, u) , Q(f,u) \ra
+  \int_0^\infty \la  S(\tau) (f , u), S(\tau)Q(f,u) \ra  \, d\tau =: I_1 + I_2.
\eal
$$ 
For the first (linear) term, we have already obtained in \eqref{eq:Zdissipative} that
$$
\bal
I_1 &:= \eta \la (f, u) , \Lambda_\eps(f,u) \ra
+  \int_0^\infty \la  S(\tau) (f , u), S(\tau)\Lambda_\eps(f,u) \ra  \, d\tau \\
&\le   - K \left\{ \Nt (f,u) \Nt_{Z}^2 +  \| \nabla^2 f\|_{L^2_k}^2 + \|\nabla^3  u  \|_{L^2}^2    \right\} =: - K \| (f,u) \|_{\widetilde Z}^2.
\eal
$$
For the second (nonlinear) term, we use the linear stability in $Z$ from Proposition~\ref{prop:stabZ} to obtain
$$
\bal
I_2 
&:= \eta \la (f, u) , Q(f,u) \ra
+  \int_0^\infty \la  S(\tau) (f , u), S(\tau)Q(f,u) \ra  \, d\tau \\
&\le \eta \| (f,u) \|_{Z} \, \|  Q(f,u) \|_{Z} 
+ \int_0^\infty \| S(\tau) (f , u)\|_{Z} \, \| S(\tau) Q(f,u) \|_{Z}   \, d\tau \\
&\le \eta \| (f,u) \|_{Z} \, \|  Q(f,u) \|_{Z} 
+ C\int_0^\infty   e^{at} \|(f , u)\|_{Z} \, e^{at} \| Q(f,u) \|_{Z}   \, d\tau \\
&\le C  \| (f,u) \|_{Z} \, \|  Q(f,u) \|_{Z} .
\eal
$$
Now, we have to compute
$$
\|  Q(f,u) \|_{Z}^2 = \| ( - \nabla\cdot(f \nabla u), 0 )  \|_{Z}^2
= \| \nabla\cdot(f \nabla u) \|_{L^2_k}^2 + \| \nabla\cdot(f \nabla u) \|_{\dot H^1_k}^2.
$$
We split $\nabla\cdot(f \nabla u) = f \Delta u + \nabla f \cdot \nabla u$ and compute
$$
\bal
\| f \Delta u \|_{L^2_k}^2 
&= \int f^2 |\Delta u|^2 \la x \ra^{2k}  
\le C  \| \nabla^2 u \|_{L^2}^2 \, \| \la x \ra ^k  f \|_{L^\infty}^2 \\
&\le C  \| \nabla^2 u \|_{L^2}^2 \, \| f \|_{H^2_k}^2  
\le C  \| (f,u) \|_{Z}^2 \, \| (f,u) \|_{\widetilde Z}^2, 
\eal
$$
where we have used the embedding $H^2 (\R^2) \hookrightarrow  L^\infty(\R^2)$. Moreover, we have
$$
\bal
\| \nabla f \cdot \nabla u \|_{L^2_k}^2 
&= \int |\nabla f|^2 |\nabla u|^2 \la x \ra^{2k} \le C  \| \nabla u \|_{L^\infty}^2 \, \| \nabla f \|_{L^2_k}^2 \\
&\le C  \| u \|_{H^3}^2 \, \| \nabla f \|_{L^2_k}^2
\le C  \| (f,u) \|_{\widetilde Z}^2 \, \| (f,u) \|_{Z}^2.
\eal
$$
Furthermore, we have
$$
\bal
\| \nabla ( f \Delta u) \|_{L^2_k}^2 
&\le C 
\int f^2 |\nabla^3 u|^2 \la x \ra^{2k} + C\int |\nabla f|^2 |\nabla^2 u|^2 \la x \ra^{2k} \\
&\le C 
\| \nabla^3 u \|_{L^2}^2 \, \| \la x \ra ^k  f \|_{L^\infty}^2  
+ C\| \la x \ra^k  \nabla f \|_{L^4}^2 \, \| \nabla^2 u \|_{L^4}^2  \\
&\le C 
\| \nabla^3 u \|_{L^2}^2 \, \| f \|_{H^2_k}^2  
+ C\| f \|_{L^2_k} \, \| f \|_{H^2_k} \, \| \nabla u \|_{L^2} \, \| \nabla^3 u \|_{L^2}^2  \\
&\le C \| (f,u) \|_{\widetilde Z}^4,
\eal
$$
where we have used the Gagliardo-Niremberg-Sobolev inequality $\| h \|_{L^4(\R^2)}^2 \le C  \| h \|_{L^2(\R^2)} \, \| \nabla h \|_{L^2(\R^2)} $. For the last term, we have
$$
\bal
\| \nabla ( \nabla f \cdot \nabla u) \|_{L^2_k}^2 
&\le C 
\int | \nabla f|^2 |\nabla^2 u|^2 \la x \ra^{2k} + C \int |\nabla^2 f|^2 |\nabla u|^2 \la x \ra^{2k} \\
&\le C 
\| \la x \ra^k  \nabla f \|_{L^4}^2 \, \| \nabla^2 u \|_{L^4}^2
+ C \| \nabla u \|_{L^\infty}^2 \, \| \nabla^2 f \|_{L^2_k}^2   \\
&\le C 
\| f \|_{L^2_k} \, \| f \|_{H^2_k} \, \| \nabla u \|_{L^2} \, \| \nabla^3 u \|_{L^2}^2
+ C \|  u \|_{H^3}^2 \, \| \nabla^2 f \|_{H^2_k}^2  \\
&\le C \| (f,u) \|_{\widetilde Z}^4.
\eal
$$
Finally, gathering the above estimates, the solution $(f,u)$ of \eqref{eq:nl} verifies
$$
\frac{d}{dt} \Nt (f,u) \Nt_{Z}^2 
\le - K \| (f,u) \|_{\widetilde Z}^2 + C \| (f,u) \|_{Z} \, \| (f,u) \|_{\widetilde Z}^2
\le ( -K + C \Nt (f,u) \Nt_{Z}  ) \, \| (f,u) \|_{\widetilde Z}^2,
$$
which concludes the proof.
\end{proof}

\subsubsection{Regularity} We prove that starting close enough to the self-similar profile, the solution to the Keller-Segel equation \eqref{eq:KS}
satisfies some strong and uniform in time estimates. 

%
%

\begin{prop}\label{prop:existence}
There is $\delta>0$ such that, if $\Nt (f_0,u_0) \Nt_{Z} \le \delta$,  there exists a solution $(f,u)  \in C([0,\infty); Z)$ to \eqref{eq:nl} (and thus to the Keller-Segel equation \eqref{eq:KS}) that verifies
\beqn\label{eq:fu<delta}
\forall \, t \ge 0, \qquad \Nt (f,u)(t) \Nt_{Z}^2 + \frac{K}2 \int_0^t \| (f,u) (\tau) \|_{\widetilde Z}^2 \, d\tau \le 2 \delta^2.
\eeqn

\end{prop}

\begin{proof}[Proof of Proposition \ref{prop:existence}]
At least formally, taking $\delta:= K/(2C)$ in \eqref{eq:dZdt}, we see that $t \mapsto \Nt (f,u)(t) \Nt_{Z}$ is decreasing if $\Nt (f_0,u_0) \Nt_{Z} \le \delta$, 
and then the a priori estimate \eqref{eq:fu<delta} immediately follows from  \eqref{eq:dZdt}. We refer to  \cite[Theorem 5.3]{GMM} for a completely rigorous proof based on an iterative scheme. 
\end{proof}

\subsubsection{Sharp exponential convergence to the equilibrium}
We now complete the proof of Theorem~\ref{theo:stab} following \cite{GMM}. Applying Lemma~\ref{lem:stabilite} to the solution $(f,u)$ constructed above and using the estimate \eqref{eq:fu<delta}, we obtain
$$
\bal
\frac{d}{dt} \Nt (f,u) \Nt_{Z}^2 
&\le (-K + C \Nt (f,u) \Nt_{Z} ) \, \| (f,u) \|_{\widetilde Z}^2\\
&\le (-K + C' \delta^2 ) \, \| (f,u) \|_{\widetilde Z}^2
\le (-K + C' \delta^2 ) \, C'' \, \Nt (f,u) \Nt_{Z}^2.
\eal
$$
If $\delta>0$ is small enough so that $-K + C' \delta^2 \le - K/2$, this differential inequality implies the exponential decay
$$
\Nt (f,u)(t) \Nt_{Z} \le e^{-\frac{KC''}{4} t} \, \Nt (f_0,u_0) \Nt_{Z}.
$$
Finally, we can recover the optimal decay rate $O(e^{at})$ of the linearized semigroup in Proposition~\ref{prop:stabZ} by performing a bootstrap argument as in \cite[Proof of Theorem 5.3]{GMM}, and that concludes the proof.

\appendix

\section{Orlicz space, interpolation and a convex function}

We describe here some classical and technical results on Orlicz spaces and interpolation spaces. We then apply these results to the function $\xi \mapsto \xi^2 (\widetilde \log \, \xi)^2$ used during the proof of Lemma~\ref{lem:BdLp}.

\subsection{Basic notions}\label{app:orlicz}
We recall some basic notions of the theory of Orlicz spaces, and we refer 
to \cite{MR0126722,MR1113700} for more details. 
We say that a function $\Phi : \R^+ \to \R^+$ is a $N$-function 
 if it is continuous, convex, $\Phi(t) >0$ for any $t>0$ and satisfies
$$
\lim_{t\to 0} \frac{\Phi(t)}{t} = 0, \quad
\lim_{t\to \infty} \frac{\Phi(t)}{t} = \infty.  
$$
For such a function $\Phi$, we define the  \emph{Orlicz space} $L_\Phi$ by
$$
f \in L_\Phi \quad\Leftrightarrow \quad \exists \lambda>0 \text{ such that } \int \Phi \left( \frac{f}{\lambda} \right) < \infty,
$$
and we endow $L_\phi$ with the Luxembourg norm
$$
\| f \|_{L_\Phi} := \inf \{ \lambda >0 \;:\; \int \Phi \left( \frac{f}{\lambda} \right) \le 1  \}.
$$
Moreover, when $\Phi$ satisfies the following $\Delta_2$-condition, 
$$
\exists \, c, s > 0, \quad  \forall \,  t \ge s, \quad   \Phi(2 t) \le c \, \Phi(t), 
$$
we have
$$
\| f \|_{L_\Phi} < \infty \quad\Leftrightarrow\quad  \forall\, \lambda>0, \; \int \Phi \left( \frac{f}{\lambda} \right) < \infty. 
$$

%
%
%
%
%
%
\subsection{Interpolation}\label{app:inter}
We state an interpolation result on Orlicz spaces from \cite{MR0199696} (see also \cite{MR1838797} and the references therein) which is used in Lemma \ref{lem:BdLp}.

\begin{theo}
Consider an $N$-function $\Phi$ such that
$$
1 \le p < \inf_{s >0} \frac{s \Phi'(s)}{\Phi(s)} \le \sup_{s >0} \frac{s \Phi'(s)}{\Phi(s)} < q < \infty. 
$$
Then the Orlicz space $L_\Phi$ is an interpolation space between $L^p$ and $L^q$. In other words, if $U : L^r \to L^r$ is bounded for $r=p,q$, then 
$U : L_\Phi \to L_\Phi$ is bounded. 
\end{theo}

\subsection{The function $s \mapsto s^2 (\widetilde \log \, s)^2$.}\label{app:example}
With the notation of the proof of Lemma~\ref{lem:BdLp}, we define  $\Phi(s) := s^2 (\widetilde \log \, s)^2$,  as well as 
$\Phi^*$ the conjugate function of $\Phi$ by 
$$
\Phi^*(t) = \sup_{s} \{ ts - \Phi(s)\} \quad \forall \, t \ge 0.
$$ 
We show below that there exists a constant $C>0$ such that 
\beqn\label{Phi*}
\Phi^*(t) \le C \, t^2 (\widetilde  \log \, t)^{-2}.
\eeqn
In order to prove such a claim, we first observe that, for any $t > 0$, the above supremum is reached at some unique point $s_t \in (0,\infty)$, which is implicitly given by the equation 
$t = \Phi'(s_t)$. In particular, $s_t \to \infty$ when $t\to\infty$ (because $\Phi'$ is increasing and bijective from $\R^*_+$ onto $\R^*_+$) so that there exists 
$t^* > 0$ such that $s_t > e$ for any $t > t^*$. For $t > t^*$, we have 
$$
4 s_t (\log s_t)^2 \ge t = \Phi'(s_t) = 2 s_t (\log s_t)^2 + 2 s_t^2 (\log s_t)^2 \ge 2 s_t (\log s_t)^2.
$$
We then compute
\bean
{t \over (\log t)^2} \ge {2 s_t (\log s_t)^2  \over (\log [4 s_t (\log s_t)^2])^2}
\ge {s_t \over 2}
\eean
for  $t > t^{**}$, $t^{**}$ large enough. We conclude that 
$$
\Phi^*(t) = t s_t - \Phi(s_t) \le 2 {t^2 \over (\log t)^2}
$$
for any $t > t^{**}$. We then easily conclude to \eqref{Phi*} by observing that $\Phi^*(t) = 4t^2$ for $t$ small enough. 

%
%


\section{Estimates on the solutions to the Poisson equation}
\label{sec:A}

In this section, we give some technical (and we believe standard) estimates on the solutions to the Poisson equation in $\R^2$ that are useful in the paper. 
We recall the notation 
$$
 \kappa(z) = - {1 \over 2\pi} \log |z|, \quad \KK(z) := \nabla \kappa(z) =  - {1 \over 2\pi} \, {z \over  |z|^2}
$$
and  
$$
\kappa_f := \kappa * f, \quad \KK_f := \KK * f,
$$
so that there holds 
$$
\kappa_f \in C^2(\R^2), \quad |\kappa_f| \le C \, (1 + \log \langle x \rangle), \quad
- \Delta \kappa_f = f. 
$$

\begin{lem}\label{lem:kappa*f} For any integer $j \ge 1$ and real number $k  > j + 2$, there exists $C_{k,j}$ such that 
\beqn\label{eq:INEGkappa*fLk1}
\| \kappa * f \|_{ L^2_{ j-1}} \le C_{k,j}  \| f \|_{L^2_{k}} \quad \forall \, f \in L^2_{k,j}. 
\eeqn
 
\end{lem}

\begin{proof}[Proof of Lemma \ref{lem:kappa*f}]
Consider $f \in L^2_{k,j}$,   so that $\hat f \in C^{j+1}_b$ and  $\partial_\xi^\alpha \hat f(0) = 0$ for any $|\alpha| \le j$ thanks to the moments condition. 
For a multi-index $\alpha \in \N^2$, $|\alpha| \le j-1$, we may thus write the Taylor expansion
$$
\partial_\xi^\alpha \hat f(\xi) =   \int_0^1(1-s) D^2_\xi \partial_\xi^\alpha \hat f (s \, \xi) : \xi^{\otimes 2}  \, ds. 
$$
In Fourier variables the Laplace equation writes $|\xi|^2 \hat \kappa_f (\xi) = \hat f(\xi)$, and then for a multi-index $\alpha \in \N^2$ 
$$
\bal
\int |x^\alpha|^2 \, | \kappa_f|^2  = \int  |\partial^\alpha_\xi \hat \kappa_f|^2     = \int \frac{|\partial^\alpha_\xi \hat f|^2}{|\xi|^4}.
\eal
$$
All together, we have 
$$
\bal
  \int | x |^{2(j-1)} \, | \kappa_f|^2
&\le   \int_{B^c_1} \frac{ |D^{j-1} \hat \kappa_f|^2}{|\xi|^4} \, d\xi + \int_{B_1} \sup_{\eta \in B_1} | D^{j+1} \hat f ( \eta )|^2 \, d\xi \\
&\le  \| f \|_{L^2_{j-1}}^2 + C \|  f  \|_{L^1_{j+1}}^2
\le C \| f \|_{L^2_k}^2,
\eal
$$
where we have used $\| g \|_{L^1} \le C  \| g \|_{L^2_r}$ for $r>1$ with $g = \langle x \rangle^{j+1} f$ in the last line, which gives $k > j +2$.
\end{proof}

\begin{lem}\label{lem:K*f} For any  integer $j \ge 0$ and real number $k > j+2$, there exists $C_{k,j}$ such that 
\beqn\label{eq:INEGLkj}
\| \KK * f \|_{L^2_j} \le C_{k,j}  \| f \|_{L^2_{k}} \quad \forall \, f \in L^2_{k,j}. 
\eeqn

\end{lem}

\begin{proof}[Proof of Lemma \ref{lem:K*f}]	
The proof is similar to Lemma \ref{lem:kappa*f}. 
In the case $j=0$ for instance, we write in Fourier variables $|\xi|^2 \hat \kappa_f (\xi) = \hat f(\xi)$, we observe that
$$
\bal
\int |\nabla \kappa_f|^2 = \int |\xi|^2 |\hat \kappa_f|^2 = \int \frac{|\hat f|^2}{|\xi|^2},
\eal
$$
and we use the moments conditions to conclude. 
\end{proof}

\section{Estimates on the $c^{-1} d$ operator}
\label{sec:B}

With the notations of Section~\ref{subsec:localizationSigma}, we consider the equation 
\beqn\label{eq:Appendix:edu}
L_\eps u := c^{-1} d u  =   \Delta u + {\eps \over 2} x \cdot \nabla u - \eps z u  =f
\eeqn
with $z \in \CC_{-1/2}$.

\begin{lem}\label{lem:du=f}    For any $f \in L^2_{k,0}$, $k > 2$, $z \in  \CC_{-1/2} \backslash \{ 0 \}$, there exists a solution $u \in H^2_{loc}$ to the 
equation $L_\eps u = f$ and there exists a constant $C$ (which does not depend on $\eps > 0$ and $z \in  \CC_{-1/2} \backslash \{ 0 \}$) such that 
$$
\| \nabla u \|_{L^2} + \| \Delta u \|_{L^2} \le C \, \| f \|_{L^2_k}. 
$$
 
\end{lem}

\begin{proof}[Proof of Lemma \ref{lem:du=f}]
Multiplying equation \eqref{eq:Appendix:edu} by $\bar u$ and $\Delta \bar u$ and its conjugate by $u$ and $\Delta u$, we find
\beqn\label{eq:Appendix:edu1}
\int |\nabla u|^2 + \eps \bigl( {1 \over 2} + \Re e  z \bigr) \int |u|^2 = -   {1 \over 2} \int (f \bar u + \bar f u)
\eeqn
and
\beqn\label{eq:Appendix:edu2}
\int |\Delta u |^2 + (\eps \, \Re e z) \int |\nabla u|^2 = {1 \over 2}\int f \Delta \bar u + {1 \over 2}\int \bar f  \Delta   u.
\eeqn
Writing 
$$
u(x) = u(y) + \int_0^1 \nabla u (z_s) \, (x-y) \, ds, \quad z_s := y + s \, (y-x), 
$$
we have 
$$
u(x) = \langle u \rangle_1 + \int_{B(0,1)} \int_0^1 \nabla u (z_s) \, (x-y) \, ds dy, \quad \langle u \rangle_1 := \int_{B(0,1)} u(y) \, dy.
$$ 
From the above equation \eqref{eq:Appendix:edu1} and the moment condition, we obtain
\bean
\int |\nabla u|^2
&\le& - \frac12 \int_{\R^2}  \int_{B(0,1)} \int_0^1 \{ \bar f(x)  \nabla u (z_s) + f(x)  \nabla \bar u (z_s) \} \cdot  (x-y) \, ds dy dx 
\\
&\le& {1 \over 2\alpha}    \int_{\R^2}   \int_{B(0,1)} \int_0^1 |f(x)|^2 \, \langle x \rangle^{2\ell} \, |x-y|^2  \, ds dy dx 
\\
&&+ {\alpha \over 2}    \int_{\R^2}   \int_0^{1/2} \Bigl\{  \int_{B(0,1)} |\nabla u(s x + (1-s) y )|^2 \, dy \Bigr\}  \,    ds   {dx \over \langle x \rangle^{2\ell}}
\\
&&+ {\alpha \over 2}    \int_{B(0,1)}  \int_{1/2}^{1} \Bigl\{  \int_{\R^2}   |\nabla u(s x + (1-s) y )|^2 \, dx \Bigr\}  \,    ds   dy 
\\
&\le& {1 \over \alpha} \, \Bigl\{  \int_{B(0,1)}  dy     \Bigr\}  \int_{\R^2}  |f(x)|^2 \, \langle x \rangle^{2(\ell+1)} \,  dx 
\\
&&+ {\alpha \over 2}    \int_{\R^2}  {dx \over \langle x \rangle^{2\ell}} \,  \int_0^{1/2} {ds \over (1-s)^2}    \int_{\R^2}  |\nabla u(z)|^2 \, dz   
\\
&&+ {\alpha \over 2}  \Bigl\{  \int_{B(0,1)}  dy     \Bigr\}    \int_{1/2}^{1} {ds \over s^2} \,    \int_{\R^2}  |\nabla u(z)|^2 \, dz  , 
\eean
and we deduce that 
$$
\int |\nabla u|^2 \le C \int_{\R^2}  |f(x)|^2 \, \langle x \rangle^{2k} \,  dx ,
$$
by choosing $\ell = k-1$ and $\alpha > 0$ small enough. 
From \eqref{eq:Appendix:edu2}, we have
$$
\| \Delta u \|_{L^2}^2 \le \frac{\eps}{2} \| \nabla u \|_{L^2}^2 + \| f \|_{L^2} \| \Delta u \|_{L^2},
$$
and we conclude the proof thanks to the above estimate on $\| \nabla u \|_{L^2}$ and Young's inequality. 
\end{proof}

\begin{lem}\label{lem:du=fbis}  There exists a constant $C$ such that for any $\eps  \in (0,1)$, any $z \in \CC_{-1/2}$ and any radially symmetric function $f $  
the equation 
$$
L_\eps u = f, \quad u \,\,\hbox{radially symmetric}, \quad u(0) = u'(0) = 0, 
$$
has a unique solution which furthermore satisfies
\beqn\label{depsrad}
\| u \, e^{-(1+\eps|z|) r} \|_{L^\infty} +
\| u' \, e^{-(1+\eps|z|) r} \|_{L^\infty} \le C \, \| f \, e^{-r/2} \|_{L^\infty}.
\eeqn
\end{lem}

\begin{proof}[Proof of Lemma \ref{lem:du=fbis}] 
We may write the equation as 
\beqn\label{eq:edoKummer}
u'' + \bigl( {1 \over r} + \eps \, r \bigr) \, u' +\eps     z u = f, \quad \forall \, r > 0,
\eeqn
with the additional boundary conditions $u(0) = u'(0) = 0$. Defining $U := |u|^2 + |u'|^2$, we have 
\bean
U' 
&=& u \bar u' + u' \bar u + u' \bar u'' + u'' \bar u'  
\\
&\le& 2 (1 + \eps |z|) \, |u| \, |u'| -   \bigl( {1 \over r} + \eps \, r \bigr) \, |u'|^2 + 2 \, |u'| \, |f|
\\
&\le&   (2 + \eps |z|) \, U +  |f|^2,
\eean
from which we immediately get \eqref{depsrad} thanks to Gronwall's lemma. 
\end{proof}


\begin{thebibliography}{10}

\bibitem{MR1230930}
{\sc Beckner, W.}
\newblock Sharp {S}obolev inequalities on the sphere and the
  {M}oser-{T}rudinger inequality.
\newblock {\em Ann. of Math. (2) 138}, 1 (1993), 213--242.

\bibitem{MR1308857}
{\sc Ben-Artzi, M.}
\newblock Global solutions of two-dimensional {N}avier-{S}tokes and {E}uler
  equations.
\newblock {\em Arch. Rational Mech. Anal. 128}, 4 (1994), 329--358.

\bibitem{Bil98}
{\sc Biler, P.}
\newblock Local and global solvability of some parabolic systems modelling chemotaxis.  
 \newblock {\em Adv. Math. Sci. Appl. 8}, 2 (1998), 715--743.

\bibitem{BCD}
{\sc Biler, P., Corrias, L., and Dolbeault, J.}
\newblock Large mass self-similar solutions of the parabolic-parabolic
  {K}eller-{S}egel model of chemotaxis.
\newblock {\em J. Math. Biol. 63}, 1 (2011), 1--32.

\bibitem{BGK}
{\sc Biler, P., Guerra, I., and Karch, G.}
\newblock Large global-in-time solutions of the parabolic-parabolic
  {K}eller-{S}egel system on the plane.
\newblock {\em Commun. Pure Appl. Anal. 14 }, 6(2015),  2117--2126.

 
\bibitem{BCM}
{\sc Blanchet, A., Carrillo, J., and Masmoudi, N.}
\newblock Infinite time aggregation for the critical {P}atlak-{K}eller-{S}egel model in $\R^2$.
\newblock {\em Comm. Pure Appl. Math. 61} no. 10 (2008), 1449--1481.
 

\bibitem{BDP}
{\sc Blanchet, A., Dolbeault, J., and Perthame, B.}
\newblock Two-dimensional {K}eller-{S}egel model: optimal critical mass and
  qualitative properties of the solutions.
\newblock {\em Electron. J. Differential Equations\/} (2006), No. 44, 32 pp.
  (electronic).


\bibitem{BrezisBook}
{\sc Brezis, H.}
\newblock {\em Analyse fonctionnelle}.
\newblock Collection Math\'ematiques Appliqu\'ees pour la Ma\^\i trise.
  [Collection of Applied Mathematics for the Master's Degree]. Masson, Paris,
  1983.
\newblock Th{\'e}orie et applications. [Theory and applications].

\bibitem{MR1308858}
{\sc Brezis, H.}
\newblock Remarks on the preceding paper by {M}. {B}en-{A}rtzi: ``{G}lobal
  solutions of two-dimensional {N}avier-{S}tokes and {E}uler equations''
  [{A}rch.\ {R}ational {M}ech.\ {A}nal.\ {\bf 128} (1994), no.\ 4, 329--358;
  {MR}1308857 (96h:35148)].
\newblock {\em Arch. Rational Mech. Anal. 128}, 4 (1994), 359--360.

\bibitem{MR2433703}
{\sc Calvez, V., and Corrias, L.}
\newblock The parabolic-parabolic {K}eller-{S}egel model in {$\Bbb R^2$}.
\newblock {\em Commun. Math. Sci. 6}, 2 (2008), 417--447.

\bibitem{MR2996772}
{\sc Campos, J., and Dolbeault, J.}
\newblock A functional framework for the {K}eller-{S}egel system: logarithmic
  {H}ardy-{L}ittlewood-{S}obolev and related spectral gap inequalities.
\newblock {\em C. R. Math. Acad. Sci. Paris 350}, 21-22 (2012), 949--954.

\bibitem{CamposDolbeault2012}
{\sc Campos, J.~F., and Dolbeault, J.}
\newblock Asymptotic estimates for the parabolic-elliptic {K}eller-{S}egel
  model in the plane.
\newblock {\em Comm. Partial Differential Equations 39}, 5 (2014), 806--841.

\bibitem{MR1143664}
{\sc Carlen, E., and Loss, M.}
\newblock Competing symmetries, the logarithmic {HLS} inequality and {O}nofri's
  inequality on {$S^n$}.
\newblock {\em Geom. Funct. Anal. 2}, 1 (1992), 90--104.

\bibitem{CLM}
{\sc Carrillo, J.~A., Lisini, S., and Mainini, E.}
\newblock Uniqueness for {K}eller-{S}egel-type chemotaxis models.
\newblock {\em Discrete Contin. Dyn. Syst. 34}, 4 (2014), 1319--1338.

\bibitem{CMPS}
{\sc Chalub, F., Markowich, P., Perthame, B. and Schmeiser, C.}
\newblock Kinetic models for chemotaxis and their drift-diffusion
              limits.
\newblock {\em Monatsh. Math. 142}, 1-2 (2004), 123--141


\bibitem{CEM}
{\sc Corrias, L., Escobedo, M., and Matos, J.}
\newblock Existence, uniqueness and asymptotic behavior of the solutions to the
  fully parabolic {K}eller-{S}egel system in the plane.
\newblock 
{\em J. Differential Equations 257}, 6 (2014),  1840--1878


\bibitem{dPL}
{\sc DiPerna, R.~J., and Lions, P.-L.}
\newblock Ordinary differential equations, transport theory and {S}obolev
  spaces.
\newblock {\em Invent. Math. 98}, 3 (1989), 511--547.

\bibitem{EM}
{\sc Ega\~na, G., and Mischler, S.}
\newblock Uniqueness and long time asymptotic for the {K}eller-{S}egel equation: the parabolic-elliptic case.
\newblock {\em Arch. Ration. Mech. Anal. 220}, 3 (2016), 1159--1194.



\bibitem{ErbanOthmer4}
{\sc Erban, R., and Othmer, H.~G.}
\newblock From individual to collective behavior in bacterial chemotaxis.
\newblock {\em SIAM J. Appl. Math. 65}, 2 (2004/05), 361--391.


\bibitem{ErbanOthmer3}
{\sc Erban, R., and Othmer, H.~G.}
\newblock From signal transduction to spatial pattern formation in {\it {E}. coli}: a paradigm for multiscale modeling in biology.
\newblock {\em Multiscale Model. Simul. 3}, 2 (2005), 62--394.


\bibitem{ErbanOthmer2}
{\sc Erban, R., and Othmer, H.~G.}
\newblock Taxis equations for amoeboid cells.
\newblock {\em J. Math. Biol. 54}, 6 (2007), 847--885.



\bibitem{MR2785976}
{\sc Ferreira, L. C.~F., and Precioso, J.~C.}
\newblock Existence and asymptotic behaviour for the parabolic-parabolic
  {K}eller-{S}egel system with singular data.
\newblock {\em Nonlinearity 24}, 5 (2011), 1433--1449.

\bibitem{FHM}
{\sc Fournier, N., Hauray, M., and Mischler, S.}
\newblock Propagation of chaos for the 2d viscous vortex model.
\newblock {\em J. Eur. Math. Soc. (JEMS) 16}, 7 (2014), 1423--1466.

\bibitem{FournierJourdain}
{\sc Fournier, N., and Jourdain, B.}
\newblock Stochastic particle approximation of the Keller-Segel equation and two-dimensional generalization of Bessel processes.
\newblock arXiv:1507.01087.

\bibitem{MR1654677}
{\sc Gajewski, H., and Zacharias, K.}
\newblock Global behaviour of a reaction-diffusion system modelling chemotaxis.
\newblock {\em Math. Nachr. 195\/} (1998), 77--114.

\bibitem{GodinhoQuininao}
{\sc Godinho, D. and Quininao, C.}
\newblock Propagation of chaos for a sub-critical Keller-Segel model.
\newblock {\em Ann. Inst. Henri Poincar\'e Probab. Stat. 51}, 3 (2015), 965--992.
 
\bibitem{GMM}
{\sc Gualdani, M.~P., Mischler, S., and Mouhot, C.}
\newblock Factorization of non-symmetric operators and exponential
  ${H}$-{T}heorem.
\newblock hal-00495786.


\bibitem{HaskovecSchmeiser1}
{\sc Ha{\v{s}}kovec, J., and Schmeiser, C.}
\newblock Stochastic particle approximation for measure valued solutions of the 2{D} {K}eller-{S}egel system.
\newblock {\em J. Stat. Phys. 135}, 1 (2009), 133--151.


\bibitem{HaskovecSchmeiser2}
{\sc Ha{\v{s}}kovec, J., and Schmeiser, C.}
\newblock Convergence of a stochastic particle approximation for measure solutions of the 2{D} {K}eller-{S}egel system.
\newblock {\em Comm. Partial Differential Equations 36}, 6 (2011), 940--960.



\bibitem{MR1627338}
{\sc Herrero, M.~A., and Vel{\'a}zquez, J.~L.~L.}
\newblock A blow-up mechanism for a chemotaxis model.
\newblock {\em Ann. Scuola Norm. Sup. Pisa Cl. Sci. (4) 24}, 4 (1997), 633--683 (1998).
 


\bibitem{HillenOthmer1}
{\sc Hillen, T., and Othmer, H.~G.}
\newblock The diffusion limit of transport equations derived from velocity-jump processes.
\newblock {\em SIAM J. Appl. Math. 61}, 3 (2000), 751--775.

  
\bibitem{HillenPainter}
{\sc Hillen, T., and Painter, K. J.}
\newblock A user's guide to {PDE} models for chemotaxis.
\newblock {\em J. Math. Biol. 58}, 1-2 (2009), 183--217. 

\bibitem{MR1838797}
{\sc Karlovich, A.~Y., and Maligranda, L.}
\newblock On the interpolation constant for {O}rlicz spaces.
\newblock {\em Proc. Amer. Math. Soc. 129}, 9 (2001), 2727--2739.

\bibitem{MR760047}
{\sc Kato, T.}
\newblock Strong {$L^{p}$}-solutions of the {N}avier-{S}tokes equation in
  {${\bf R}^{m}$}, with applications to weak solutions.
\newblock {\em Math. Z. 187}, 4 (1984), 471--480.

\bibitem{KS}
{\sc Keller, E.~F., and Segel, L.~A.}
\newblock Initiation of slime mold aggregation viewed as an instability.
\newblock {\em J. Theor. Biol. 26\/} (1970), 399--415.

\bibitem{komatsu}
{\sc Komatsu, H.}
\newblock Fractional powers of operators. {II}. {I}nterpolation spaces.
\newblock {\em Pacific Journal of Mathematics 21}, 1 (1967), 89--111.

\bibitem{MR0126722}
{\sc Krasnosel'ski{\u\i}, M.~A., and Ruticki{\u\i}, J.~B.}
\newblock {\em Convex functions and {O}rlicz spaces}.
\newblock Translated from the first Russian edition by Leo F. Boron. P.
  Noordhoff Ltd., Groningen, 1961.

\bibitem{LP}
{\sc Lions, J.-L., and Peetre, J.}
\newblock Sur une classe d'espaces d'interpolation.
\newblock {\em Publ. Math. Inst. Hautes \'Etudes Sci. 19\/} (1964), 5--68.

\bibitem{MM*}
{\sc Mischler, S., and Mouhot, C.}
\newblock Exponential stability of slowly decaying solutions to the
              kinetic-{F}okker-{P}lanck equation.
\newblock {\em Arch. Ration. Mech. Anal. 221}, 2 (2016), 677--723.



\bibitem{MMcmp}
{\sc Mischler, S., and Mouhot, C.}
\newblock Stability, convergence to self-similarity and elastic limit for the
  {B}oltzmann equation for inelastic hard spheres.
\newblock {\em Comm. Math. Phys. 288}, 2 (2009), 431--502.

\bibitem{MS*}
{\sc Mischler, S., and Scher, J.}
\newblock Spectral analysis of semigroups and growth-fragmentation
              equations.
\newblock {\em Ann. Inst. H. Poincar\'e Anal. Non Lin\'eaire 33}, 3 (2016), 849--898.

\bibitem{MR3116019}
{\sc Mizoguchi, N.}
\newblock Global existence for the {C}auchy problem of the parabolic-parabolic
  {K}eller-{S}egel system on the plane.
\newblock {\em Calc. Var. Partial Differential Equations 48}, 3-4 (2013),
  491--505.

\bibitem{MR1970697}
{\sc Nagai, T.}
\newblock Global existence and blowup of solutions to a chemotaxis system.
\newblock In {\em Proceedings of the {T}hird {W}orld {C}ongress of {N}onlinear
  {A}nalysts, {P}art 2 ({C}atania, 2000)\/} (2001), vol.~47, pp.~777--787.

\bibitem{MR1799300}
{\sc Nagai, T., Senba, T., and Suzuki, T.}
\newblock Chemotactic collapse in a parabolic system of mathematical biology.
\newblock {\em Hiroshima Math. J. 30}, 3 (2000), 463--497.

\bibitem{NSY}
{\sc Naito, Y., Suzuki, T., and Yoshida, K.}
\newblock Self-similar solutions to a parabolic system modeling chemotaxis.
\newblock {\em J. Differential Equations 184}, 2 (2002), 386--421.


\bibitem{OthmerHillen2}
{\sc Othmer, H.~G. and Hillen, T.}
\newblock The diffusion limit of transport equations. {II}. {C}hemotaxis equations.
\newblock {\em SIAM J. Appl. Math. 62}, 4 (2002), 1222--1250.

\bibitem{P}
{\sc Patlak, C.~S.}
\newblock Random walk with persistence and external bias.
\newblock {\em Bull. Math. Biophys. 15\/} (1953), 311--338.

\bibitem{MR1113700}
{\sc Rao, M.~M., and Ren, Z.~D.}
\newblock {\em Theory of {O}rlicz spaces}, vol.~146 of {\em Monographs and
  Textbooks in Pure and Applied Mathematics}.
\newblock Marcel Dekker, Inc., New York, 1991.

\bibitem{RS}
{\sc Rapha\"el, P. and Schweyer, R.}
\newblock On the stability of critical chemotactic aggregation.
\newblock {\em Math. Ann. 359}, 1-2 (2014),  267--377.

\bibitem{SaragostiCBBSP}
{\sc Saragosti, J., Calvez, V., Bournaveas, N., Buguin, A., Silberzan, P., and Perthame, B.}
\newblock Mathematical description of bacterial traveling pulses.
\newblock {\em PLoS Comput. Biol. 6}, 8 (2010), e1000890, 12.

\bibitem{MR0199696}
{\sc Simonenko, I.~B.}
\newblock Interpolation and extrapolation of linear operators in {O}rlicz
  spaces.
\newblock {\em Mat. Sb. (N.S.) 63 (105)\/} (1964), 536--553.

\bibitem{Stevens}
{\sc Stevens, A.}
\newblock The derivation of chemotaxis equations as limit dynamics of moderately interacting stochastic many-particle systems.
\newblock {\em SIAM J. Appl. Math. 61} 1 (2000), 183--212 .

\bibitem{TindallMPA}
{\sc Tindall, M.~J., Maini, P.~K., Porter, S.~L., and Armitage, J.~P.},
\newblock Overview of mathematical approaches used to model bacterial
              chemotaxis. {II}. {B}acterial populations.
\newblock {\em Bull. Math. Biol. 70}, 6 (2008), 1570--1607.             
   

\bibitem{Tristani}
{\sc Tristani, I.}
\newblock Boltzmann equation for granular media with thermal force in a weakly
  inhomogeneous setting.
\newblock {\em J. Funct. Anal. 270}, 5 (2016), 1922--1970.

\bibitem{ZinslMatthes}
{\sc Zinsl, J., and Matthes, D.}
\newblock Exponential convergence to equilibrium in a coupled gradient flow
  system modelling chemotaxis.
\newblock {\em Anal. PDE 8}, 2 (2015), 425--466.

\end{thebibliography}

\end{document}